\documentclass[]{interact}

\usepackage{braket,amsfonts}
\usepackage{subdepth}

\usepackage{array}
\usepackage{amsmath}
\allowdisplaybreaks
\usepackage{algorithmic}
\usepackage{algorithm}
\usepackage{enumitem}

\usepackage{graphicx,epstopdf}

\usepackage{url}
\usepackage{amsopn}
\usepackage{comment}
\usepackage{graphics}
\usepackage{amssymb}
\usepackage{amsmath}
\usepackage{bm}
\usepackage{xcolor}
\usepackage{graphicx}
\usepackage{todonotes}
\usepackage{mathrsfs}
\usepackage{indentfirst}
\usepackage{hyperref}
\hypersetup{
	breaklinks,
	colorlinks,
	citecolor={green!50!black},
	urlcolor={red!50!black},
	linkcolor={green!50!black}
}
\usepackage[nameinlink]{cleveref}
\crefname{subsection}{section}{sections}
\Crefname{algorithm}{Algorithm}{Algorithms}
\crefname{appendix}{Appendix}{Appendices}
\Crefname{appendix}{Appendix}{Appendices}

\Crefname{ALC@unique}{Line}{Lines}

\usepackage[numbers,sort&compress]{natbib}
\makeatletter
\renewcommand\NAT@bibsetnum[1]{\settowidth\labelwidth{\@biblabel{#1}}%
	\setlength{\leftmargin}{\bibindent}\addtolength{\leftmargin}{\dimexpr\labelwidth+\labelsep\relax}%
	\setlength{\itemindent}{-\bibindent}%
	\setlength{\listparindent}{\itemindent}
	\setlength{\itemsep}{\bibsep}\setlength{\parsep}{\z@}%
	\ifNAT@openbib
	\addtolength{\leftmargin}{\bibindent}%
	\setlength{\itemindent}{-\bibindent}%
	\setlength{\listparindent}{\itemindent}%
	\setlength{\parsep}{0pt}%
	\fi
}
\makeatother

\usepackage{etoolbox}
\makeatletter
\patchcmd{\NAT@test}{\else \NAT@nm}{\else \NAT@hyper@{\NAT@nm}}{}{}
\makeatother

\newtheorem{theorem}{Theorem}[section]
\crefname{theorem}{Theorem}{Theorems}

\crefname{cor}{Corollary}{Corollaries}
\newtheorem{lem}{Lemma}[section]
\crefname{lem}{Lemma}{Lemmas}
\newtheorem{rem}{Remark}[section]
\crefname{rem}{Remark}{Remarks}
\newtheorem{ex}{Example}
\crefname{ex}{Example}{Examples}

\crefname{ap}{Appendix}{Appendices}

\newcommand{\calA}{{\cal A}}
\newcommand{\calP}{{\cal P}}

\newcommand{\nul}{\mathop{\mathrm{null}}}

\usepackage{booktabs}
\usepackage{threeparttable}
\usepackage{float}
\usepackage{booktabs}%\toprule、\midrule、\bottomrule
\usepackage{multirow}%\multicolumn{}{}{}
\usepackage{longtable}

\usepackage{graphicx}
\usepackage{subfigure}

\usepackage[T1]{fontenc}
\usepackage{microtype}

\numberwithin{equation}{section}

\newcommand{\mycomment}[1]{}

\usepackage{lmodern}
\usepackage{anyfontsize}

\makeatletter

\newcommand{\Rmnum}[1]{\expandafter\@slowromancap\romannumeral #1@}
\makeatother

\begin{document}
	
\title{Randomized inexact block triangular preconditioners for double saddle-point systems in PDE-constrained optimization}
	
	    \author{Siqi Liang$^1$
		\thanks{$^\dagger$ Correspondence to hna@cau.edu.cn.}
		\and
		Na Huang$^\dagger$
		\thanks{$^1$Department of Applied Mathematics, College of Science, China Agricultural University, Beijing, China. E-mail: liangsiqi@cau.edu.cn.}
		\thanks{$^\dagger$Department of Applied Mathematics, College of Science, China Agricultural University, Beijing, China. E-mail: hna@cau.edu.cn.}
	}

	\maketitle
	
\begin{abstract}
We develop a new class of inexact block triangular preconditioners for double saddle-point systems arising from PDE-constrained optimization. The proposed preconditioners are constructed through matrix factorization techniques while preserving the inherent block structure of the original systems. A comprehensive spectral analysis of the preconditioned matrices is provided, yielding explicit bounds for both real and nonreal eigenvalues. To enable efficient construction of the inexact preconditioners, randomized strategies are introduced to select the required subblocks. We establish high-probability bounds for the expected approximation error, with the error estimates explicitly characterized in terms of the eigenvalues of the associated matrices. Numerical experiments demonstrate the effectiveness, robustness, and scalability of the proposed preconditioners, and validate the efficiency of the randomized construction strategies.
\end{abstract}
	

	\begin{amscode}
		65F08, 65F10, 65F50.
	\end{amscode}

	\section{Introduction}\label{sec:introduction}
    Linear systems of saddle‑point structure arise in a wide range of scientific computing applications, including PDE-constrained optimization \cite{rees2010optimal,bergamaschi2025spectral}, computational fluid dynamics \cite{rhebergen2015three,elman2014finite}, quadratic programming \cite{han2013local,huang2016consensus}, and least-squares problems \cite{yuan1996numerical,bojanczyk2003equality,bjorck2024numerical}. Given positive integers $n \ge m \ge p$, let $A \in \mathbb{R}^{n \times n}$ and $E \in \mathbb{R}^{p \times p}$ be symmetric positive definite (SPD) matrices, and let $B \in \mathbb{R}^{m \times n}$ and $C \in \mathbb{R}^{p \times m}$ be full row rank matrices. We consider the following linear system with a double saddle-point structure:
	\begin{equation}\label{eq:double saddle}
		\calA w:= \left( \begin{matrix}
			A&		0&		B^T\\
			0&		E&		C\\
			B&		C^T&	0\\
		\end{matrix} \right) \left( \begin{array}{c}
			x\\
			y\\
			z \\
		\end{array} \right) =\left( \begin{array}{l}
			b_1\\
			b_2\\
			b_3\\
		\end{array} \right) =: b,	
	\end{equation}    
	where $b_1 \in \mathbb{R}^n$, $b_2 \in \mathbb{R}^p$, and $b_3 \in \mathbb{R}^m$ are prescribed vectors, and $x \in \mathbb{R}^n$, $y \in \mathbb{R}^p$, and $z \in \mathbb{R}^m$ are the unknown vectors to be determined. Here, $(\cdot)^T$ denotes the transpose operator, and $0$ denotes a zero matrix of appropriate dimensions. In addition, there exists another class of double saddle-point linear systems, structurally related but not identical to \eqref{eq:double saddle}, arising from liquid crystal director modeling \cite{ramage2013preconditioned} and coupled Stokes--Darcy problems \cite{cai2009preconditioning}. Their numerical solution has attracted considerable attention in recent years; see, for example, \cite{ali2018iterative,huang2023gsor,dou2023class,beik2022preconditioning,liang2026partial} and the references therein. In this paper, however, we focus exclusively on system \eqref{eq:double saddle}.

	The diverse applications of system \eqref{eq:double saddle} have motivated extensive studies on efficient numerical solution methods, particularly for the case in which the subblock $E=0$. Existing methods include shift-splitting iterative methods \cite{cao2019shift,zhang2022lopsided,ahmad2025robust}, Uzawa-type methods \cite{huang2019uzawa,huang2020variable}, and Krylov subspace methods \cite{saad2003iterative}. Since the effectiveness of Krylov subspace methods depends strongly on the availability of suitable preconditioners, substantial effort has been devoted to the design of efficient preconditioning techniques for such systems. For the case $E=0$, the existing preconditioners mainly consist of block diagonal preconditioners \cite{huang2019spectral,abdolmaleki2022new}, block triangular preconditioners \cite{aslani2023block,balani2024some}, and preconditioners derived from splitting iterative methods \cite{cao2019shift,liang2024improvement,zhang2022lopsided,salkuyeh2021alternating,li2025uzawa}.

	For the case $E \neq 0$, system \eqref{eq:double saddle} can be reformulated as the following standard $2\times 2$ block saddle-point system:
    \begin{equation}\label{eq:2*2}
		\left( \begin{matrix}
	       A_1&		B_1^T\\
	       B_1&		0\\
        \end{matrix} \right) \left( \begin{array}{c}
			p\\
			z \\
		\end{array} \right)= \left( \begin{array}{l}
			f\\
			b_3\\
		\end{array} \right),	
	\end{equation}   
    where $A_1 = \left( \begin{smallmatrix}
    	A&		0\\
    	0&		E\\
    \end{smallmatrix} \right)$, $B_1=( \begin{matrix}
    B&		C^T\\
    \end{matrix} )$, $p=(x;y)$, and $f = (b_1;b_2)$. Numerical methods for standard saddle-point systems of the form \eqref{eq:2*2} have been extensively studied, including Uzawa‑type methods \cite{elman1994inexact,bai2008parameterized,bramble1997analysis}, shift‑splitting methods \cite{cao2014shift,cao2017preconditioned}, Hermitian and skew‑Hermitian splitting methods \cite{benzi2004preconditioner,cao2016simplified}, block diagonal preconditioners \cite{benzi2008some}, and block triangular preconditioners \cite{benzi2008some,simoncini2004block}. For more details, see \cite{benzi2005numerical} and the references therein. Although block aggregation can formally reformulate system \eqref{eq:double saddle} as a $2\times2$ saddle-point system, such a transformation is not necessarily beneficial for large-scale problems, as it may obscure the hierarchical coupling structure associated with multiple constraints. Specifically, for system \eqref{eq:double saddle} arising from distributed control problems, the subblocks $B$ and $C$ correspond to distinct constraint operators with different physical meanings, representing two independent constraint structures \cite{rees2010optimal}. Aggregating these blocks into a single constraint operator combines heterogeneous constraints and may discard important block-wise information that can be exploited for designing efficient and scalable solvers. Furthermore, as demonstrated by the numerical results in \Cref{sec:numres}, directly solving the aggregated formulation of \eqref{eq:double saddle} may result in reduced computational efficiency. Consequently, preconditioners developed for standard $2\times2$ saddle-point systems may not fully exploit the underlying block structure of \eqref{eq:double saddle}, thereby limiting their effectiveness for large-scale applications. Therefore, developing scalable structure-preserving preconditioners that efficiently exploit the hierarchical block structure of double saddle-point systems is of both practical and theoretical significance.

    Substantial effort has been devoted to the development of efficient preconditioners for the case $E \neq 0$ by exploiting the special block structure of $\calA$. In particular, \citet{benzi2011dimensional} studied the following system arising from the discretization of the Navier-Stokes equations: 
    \begin{equation}\label{eq:double saddle_equal}
		\mathcal{B} w:= \left( \begin{matrix}
			A&		0&		B^T\\
			0&		E&		C\\
			-B&		-C^T&	0\\
		\end{matrix} \right) \left( \begin{array}{c}
			x\\
			y\\
			z \\
		\end{array} \right) =\left( \begin{array}{c}
			b_1\\
			b_2\\
			-b_3\\
		\end{array} \right) =: \tilde{b}.
	\end{equation}
    Clearly, \eqref{eq:double saddle_equal} is equivalent to \eqref{eq:double saddle} under a simple permutation of variables. They split $\mathcal{B}$ along the velocity field components and proposed an alternating iterative method, which induces a dimensional splitting (DS) preconditioner of the form:
	\begin{equation}\label{eq:pre_DS}
		\calP_{\rm DS} = \frac{1}{2\alpha}\left( \begin{matrix}
			\alpha I+A&		0&		B^T\\
			0&		\alpha I&		0\\
			-B&		0&		\alpha I\\
		\end{matrix} \right) \left( \begin{matrix}
			\alpha I&		0&		0\\
			0&		\alpha I+E&		C\\
			0&		-C^T&		\alpha I\\
		\end{matrix} \right),
	\end{equation}
	where $\alpha>0$ is a parameter and $I$ denotes the identity matrix of appropriate dimension. Subsequently, \citet{benzi2011relaxed} removed the shift terms in the $(1,1)$ and $(2,2)$ blocks of $\calP_{\rm DS}$, and constructed a relaxed dimensional factorization (RDF) preconditioner.	Building on DS preconditioner, \citet{ai2024multi} introduced additional parameters and developed a multi-parameter dimensional splitting preconditioner for \eqref{eq:double saddle_equal}. \citet{ahmad2026class} studied the solvability of \eqref{eq:double saddle_equal} under certain conditions and proposed a generalized shift-splitting iterative method. Furthermore, this method induced generalized shift-splitting (GSS) preconditioner. When it is employed as a preconditioner for Krylov subspace methods, each iteration requires solving two linear subsystems with coefficient matrix $\alpha P + \omega A$, two with $\beta Q + \omega E$, and one with $\tau R + \omega^2 B(\alpha P + \omega A)^{-1} B^T + \omega^2 C^T (\beta Q + \omega D)^{-1} C$. To reduce the computational cost, two relaxed variants of the GSS preconditioner were introduced in \cite{ahmad2026class} by removing $\alpha P$ from the $(1,1)$ block and $\beta Q$ from the $(2,2)$ block of GSS, respectively. In addition, \citet{bradley2023eigenvalue} considered a more general class of double saddle-point systems together with corresponding preconditioning strategies, and developed a block diagonal (BD) preconditioner
	\begin{equation}\label{eq:pre_BD}
		\calP_{\rm BD}=\left( \begin{matrix}
			A&		0&		0\\
			0&		S&		0\\
			0&		0&		E+CS^{-1}C^T\\
		\end{matrix} \right)
	\end{equation}    
    for the linear system of form
	\begin{equation}\label{eq:double saddle_equal2}
		\mathcal{C} w:= \left( \begin{matrix}
			A&		B^T&		0\\
			B&		0&		C^T\\
			0&		C&	E\\
		\end{matrix} \right) \left( \begin{array}{c}
			x\\
			z\\
			y \\
		\end{array} \right) =\left( \begin{array}{c}
			b_1\\
			b_3\\
			b_2\\
		\end{array} \right) =: \widehat{b},
	\end{equation}
    where $S= BA^{-1}B^T$. Similarly, \eqref{eq:double saddle_equal2} is also equivalent to \eqref{eq:double saddle}. Based on the preconditioner proposed in \cite{pearson2024symmetric} for multiple saddle-point systems, \citet{bergamaschi2025spectral} investigated the spectral properties of the corresponding preconditioned matrix for \eqref{eq:double saddle_equal2}. Their analysis relies on a key result derived from the KKT conditions of an associated optimization problem, which will also play an important role in our analysis.

	Specialized preconditioning techniques have also been developed for system \eqref{eq:double saddle} arising from PDE-constrained optimization problems \cite{pearson2014preconditioners,fan2024preconditioners,ke2018some,rees2010block}. \citet{rees2010optimal} proposed block diagonal and constraint preconditioners. \citet{zhang2014block} developed block counter-diagonal and block counter-triangular preconditioners. \citet{pearson2012new} introduced improved Schur complement approximations and constructed block diagonal, block triangular, and symmetric indefinite preconditioners.

	Integrating randomized techniques into conventional methods is of significant theoretical and practical importance for reducing the computational complexity of large-scale matrix problems and improving computational and storage efficiency, making it an important research direction in numerical linear algebra \cite{halko2011finding,martinsson2020randomized,murray2023randomized,woodruff2014sketching}. The development of randomized preconditioning techniques for linear systems is still at an early stage. For certain classes of regularized linear systems, existing works have proposed the randomized Nystr\"om preconditioner \cite{frangella2023randomized}, randomized pivoted Cholesky preconditioner \cite{diaz2023robust}, and preconditioner via randomized range deflation \cite{balabanov2025preconditioning}.

To the best of our knowledge, randomized techniques have not yet been explored for the large-scale saddle-point system \eqref{eq:double saddle}. This gap motivates the integration of randomized ideas into preconditioner design, aiming to develop computationally efficient and memory-saving preconditioners that accelerate the solution of \eqref{eq:double saddle}. In this work, we propose and analyze a class of randomized inexact preconditioners for \eqref{eq:double saddle}. Moreover, in the analysis of randomized approximation errors, we extend existing theoretical results beyond the positive semidefinite setting and establish error bounds for symmetric indefinite matrices. The main contributions of this paper are summarized as follows.
	\begin{itemize}[label=$\bullet$]
	\item Motivated by matrix factorization techniques and considering both computational efficiency and storage requirements, we propose a class of inexact block triangular preconditioners for solving \eqref{eq:double saddle}. Unlike standard block aggregation approaches, the proposed framework preserves the hierarchical coupling structure among multiple constraints and avoids the loss of essential block information.

	\item By integrating randomized low-rank approximation techniques into the proposed framework, we develop efficient strategies for selecting the subblocks required in the inexact preconditioners. We establish that the expected approximation error is bounded with high probability, where the bound is explicitly characterized in terms of the eigenvalues of the associated matrix.

	\item Extensive numerical experiments are conducted to evaluate the performance of the proposed preconditioners against several existing approaches, including methods originally developed for \eqref{eq:2*2}. The results demonstrate the effectiveness and robustness of the proposed preconditioners, as well as the potential of randomized techniques for efficient preconditioner construction.
\end{itemize}

    The remainder of this paper is organized as follows. In \Cref{sec:precon}, we propose a class of inexact block triangular preconditioners and analyze the spectral properties of the corresponding preconditioned matrices. \Cref{sec:select} develops strategies for selecting the subblocks in the inexact preconditioners by incorporating randomized low-rank approximation techniques. Numerical experiments are presented in \Cref{sec:numres}. Finally, conclusions and directions for future research on randomized preconditioning techniques are discussed in \Cref{sec:conclusion}.
	
Before concluding this section, we introduce several notations used throughout the paper. The sets of real and complex numbers are denoted by $\mathbb{R}$ and $\mathbb{C}$, respectively. For any vector $h \in \mathbb{C}^r$, $h^*$ denotes its conjugate transpose. For a matrix $H \in \mathbb{R}^{r \times r}$, its spectral radius and spectrum are denoted by $\rho(H)$ and ${\rm sp}(H)$, respectively. If $H$ is symmetric, then $\lambda_{\max}(H)$ and $\lambda_{\min}(H)$ denote its largest eigenvalue and smallest eigenvalue, respectively. We use $\sigma_{\max}(G)$ and $\sigma_{\min}(G)$ to denote the largest and smallest singular values of $G \in \mathbb{R}^{r_1 \times r_2}$. The symbol $\|\cdot\|$ stands for the spectral norm for matrices and vectors. The null space of a matrix is denoted by $\nul(\cdot)$. The notation $|\cdot|$ denotes the modulus of a scalar, and $\mathrm{i}$ is the imaginary unit. Furthermore, $\Re(\cdot)$ and $\Im(\cdot)$ denote the real and imaginary parts, respectively. For brevity, we use $(x^T,y^T,z^T)^T \equiv (x;y;z)$.

	  	\section{Inexact block triangular preconditioners and spectral analysis}\label{sec:precon}

    In this section, by exploiting a factorization of $\calA$ and taking computational efficiency into consideration, we propose a class of inexact block triangular preconditioners for \eqref{eq:double saddle} and investigate spectral bounds for the corresponding preconditioned matrices. 
    
    Consider the following factorization:
	\begin{equation*}
		\calA = \left( \begin{matrix}
			I&		0&		0\\
			0&		I&		0\\
			BA^{-1}&		0&		I\\
		\end{matrix} \right) \left( \begin{matrix}
			A&		0&		0\\
			0&		E&		C\\
			0&		C^T&		-S\\
		\end{matrix} \right) \left( \begin{matrix}
			I&		0&		A^{-1}B^T\\
			0&		I&		0\\
			0&		0&		I\\
		\end{matrix} \right), 
	\end{equation*}
where $S=BA^{-1}B^T$. Motivated by this factorization, it is natural to employ $\left( \begin{smallmatrix}
			A&		0&		0\\
			0&		E&		C\\
			0&		C^T&		-S\\
\end{smallmatrix} \right)$ as a preconditioner. However, when it is used to precondition Krylov subspace methods, one needs to solve linear subsystems involving $A$, $E$, and $S+C^TE^{-1}C$ at each iteration. It should be noted that forming and solving systems with $S+C^TE^{-1}C$ generally incurs considerable computational and storage costs. Therefore, following the approach in \cite{liang2025inexact,bergamaschi2025spectral}, we replace $S+C^TE^{-1}C$ with an SPD matrix $Q \in \mathbb{R}^{m \times m}$ and introduce the following inexact block triangular preconditioner	
	\begin{equation}\label{eq:pre_inexact}
		\calP = \left( \begin{matrix}
			\widehat{A}&		0&		0\\
			0&		\widehat{E}&		C\\
			0&		0&		-Q\\
		\end{matrix} \right),
	\end{equation}	
where $\widehat{A}$ and $\widehat{E}$ are SPD approximations of $A$ and $E$, respectively. Clearly, applying $\calP$ as a preconditioner for Krylov subspace methods requires solving three SPD linear subsystems with coefficient matrices $\widehat{A}$, $\widehat{E}$, and $Q$ at each iteration.

It should be emphasized that even when $\calP$ is interpreted in a $2\times 2$ block form, it does not correspond to the upper triangular preconditioners for the standard saddle-point system studied in \cite{benzi2005numerical, simoncini2004block}. In particular, the upper-right block of $\calP$ is $(0^T,C^T)^T$, rather than $(B^T,C^T)^T$.

In what follows, we estimate the spectral bounds for the preconditioned matrix $\calP^{-1}\calA$. Let $\lambda$ be an arbitrary eigenvalue of $\calP^{-1}\calA$ and $\zeta$ the corresponding eigenvector. Then the eigenvalue problem $\calP^{-1}\calA \zeta = \lambda \zeta$ can be equivalently rewritten as the generalized eigenvalue problem $\mathcal{D}^{-\frac{1}{2}}\calA\mathcal{D}^{-\frac{1}{2}}\phi=\lambda\mathcal{D}^{-\frac{1}{2}}\calP\mathcal{D}^{-\frac{1}{2}}\phi$, where $\phi = \mathcal{D}^{\frac{1}{2}}\zeta = (x;y;z)$ and $\mathcal{D} = \left( \begin{smallmatrix}
			\widehat{A}&		0&		0\\
			0&		\widehat{E}&		0\\
			0&		0&	 Q\\
\end{smallmatrix} \right)$.
Combining this with \eqref{eq:double saddle} and \eqref{eq:pre_inexact} yields
\begin{align}
		& \widetilde{A}x+\bar{B}^Tz=\lambda x, \label{eq:eigfunction_inexact1}\\
		& \widetilde{E}y+\bar{C}z=\lambda y+\lambda \bar{C} z, \label{eq:eigfunction_inexact2}\\
		& \bar{B}x+\bar{C}^Ty = -\lambda z. \label{eq:eigfunction_inexact3}
\end{align}
Here $\widetilde{A}=\widehat{A}^{-\frac{1}{2}}A\widehat{A}^{-\frac{1}{2}}$, $\widetilde{E}=\widehat{E}^{-\frac{1}{2}}E\widehat{E}^{-\frac{1}{2}}$, $\bar{B}=Q^{-\frac{1}{2}}B\widehat{A}^{-\frac{1}{2}}$, and $\bar{C}=\widehat{E}^{-\frac{1}{2}}CQ^{-\frac{1}{2}}$. 

We now establish bounds for $\lambda$ using \eqref{eq:eigfunction_inexact1}-\eqref{eq:eigfunction_inexact3} by considering the real and nonreal cases separately. To this end, we first introduce several quantities that will be used throughout the analysis.

For a symmetric matrix $M$ and a nonzero vector $w$, the Rayleigh quotient is defined by $r(M,w) = \frac{w^TMw}{w^Tw}$. It is well known that $r(M,w)\in [\, \lambda_{\min}(M), \, \lambda_{\max}(M)\,]$. We then define the following notations:
\begin{align}
		\gamma_A &= r(\widetilde{A},w),
        &\gamma_E &= r(\widetilde{E},w),
        &\gamma_B &= r(\bar{B},w),
        &\gamma_C &= r(\bar{C},w),
        \label{eq:eig-gamma}\\    
        \gamma_{\min}^A&=\lambda_{\min}(\widetilde{A}),
		&\gamma_{\min}^E&=\lambda_{\min}(\widetilde{E}),
        &\gamma_{\min}^B&=\lambda_{\min}(\bar{B}\bar{B}^T),
        &\gamma_{\min}^C&=\lambda_{\min}(\bar{C}^T\bar{C}),
        \label{eq:eig-min}\\
		\gamma_{\max}^A&=\lambda_{\max}(\widetilde{A}), 
		&\gamma_{\max}^E&=\lambda_{\max}(\widetilde{E}), 
		&\gamma_{\max}^B&=\lambda_{\max}(\bar{B}\bar{B}^T),
		&\gamma_{\max}^C&=\lambda_{\max}(\bar{C}^T\bar{C}).
        \label{eq:eig-max}
\end{align}
Without loss of generality, we assume that $\gamma^A_{\min} < 1 < \gamma^A_{\max}$ and $\gamma^E_{\min} < 1 < \gamma^E_{\max}$, which can be readily ensured provided that $\widetilde{A} \neq I$ and $\widetilde{E} \neq I$.

\subsection{Bounds for real eigenvalues}\label{sec:spectral-1}

In this subsection, we investigate the spectral bounds for real $\lambda$, deriving them by analyzing three cases associated with $z=0$, $0\neq z\in\nul(\bar{C})$, and $z \notin \nul(\bar{C})$. 

    \begin{lem}\label{lem:bound1}
		Assume that $A$ and $E$ are SPD, and that $B$ and $C$ have full row rank. For any given SPD matrices $\widehat{A}$, $\widehat{E}$ and $Q$, if $z=0$, then 
        $\min \{ \gamma _{\min}^{A},\gamma _{\min}^{E}\}
        \le\lambda
        \le \max \{ \gamma _{\max}^{A},\,\gamma _{\max}^{E} \}$.
    \end{lem}

\begin{proof}
    If $z=0$, then \eqref{eq:eigfunction_inexact1}-\eqref{eq:eigfunction_inexact3} reduce to
    $\widetilde{A}x =\lambda x$, $\widetilde{E}y=\lambda y$, and $\bar{B}x+\bar{C}^Ty=0$. Therefore, for any nonzero pair $(x;y)$ satisfying the last equality above, $\lambda$ is an eigenvalue of either $\widetilde{A}$ or $\widetilde{E}$, from which the desired result follows.
\end{proof}

\begin{lem}\label{lem:bound2}
		Under the same assumptions as in \Cref{lem:bound1}, if $0\neq z \in \nul(\bar{C})$, then $\gamma_{\min}^B / \gamma_{\max}^A \le \lambda \le \gamma^A_{\max}$.
	\end{lem}
	
	\begin{proof}
        It suffices to consider the case $\lambda\notin{\rm sp}(\widetilde{A})\cup{\rm sp}(\widetilde{E})$. From \eqref{eq:eigfunction_inexact1} and \eqref{eq:eigfunction_inexact2} we obtain
		\begin{equation}\label{eq:inexact_relation}
			x = (\lambda I-\widetilde{A})^{-1}\bar{B}^Tz \quad \text{and} \quad y = (1-\lambda)(\lambda I-\widetilde{E})^{-1}\bar{C}z.
		\end{equation}
		Since $z \in \nul(\bar{C})$, by \eqref{eq:inexact_relation}, it gives $y=0$.
        Substituting \eqref{eq:inexact_relation} into \eqref{eq:eigfunction_inexact3} and premultiplying by $z^T/z^Tz$, we have $\tfrac{z^T\bar{B}(\lambda I-\widetilde{A})^{-1}\bar{B}^Tz}{z^Tz} + \lambda = 0$. Setting $u=\bar{B}^Tz$, this relation can be rewritten as $\tfrac{u^T(\lambda I-\widetilde{A})^{-1}u}{u^Tu} \cdot \tfrac{z^T\bar{B}\bar{B}^Tz}{z^Tz} + \lambda = 0.$ Together with \eqref{eq:eig-gamma}, it leads to
		\begin{equation}\label{eq:p_lam}
			\frac{1}{\lambda-\gamma_A}\cdot \gamma_B+\lambda= \frac{\lambda^2-\gamma_A\lambda+\gamma_B}{\lambda-\gamma_A}:=\frac{p(\lambda)}{\lambda-\gamma_A} = 0.
		\end{equation}
		This implies that $\lambda$ is a root of the quadratic polynomial $p(\lambda)$. Since $\lambda$ is real, it follows that $\gamma_A^2 \ge 4\gamma_B$ and $\lambda = \frac{\gamma_A \pm \sqrt{\gamma_A^2-4\gamma_B}}{2}$. Consequently,
        \begin{equation}\label{eq:lam_1}
            \lambda \le \frac{\gamma_A+\sqrt{\gamma_A^2-4\gamma_B}}{2} \le \frac{\gamma_{\max}^A+\sqrt{(\gamma_{\max}^A)^2-4\gamma_{\min}^B}}{2} \le \gamma_{\max}^A,
            \end{equation}
            %\lambda^p_+(\gamma_{\max}^A ,\gamma_{\min}^B)
        and
    \begin{equation}\label{eq:lam_2}
            \lambda  \ge \frac{\gamma_A-\sqrt{\gamma_A^2-4\gamma_B}}{2} = \frac{2\gamma_B}{\gamma_A+\sqrt{\gamma^2_A-4\gamma_B}} 
            \ge \frac{2\gamma_B}{\gamma_A+\gamma_A} \ge \frac{\gamma_{\min}^B}{\gamma_{\max}^A},
    \end{equation}
following the result.
\end{proof}

%%%%%%%%%%%%%%%%%%%%%%%%%%%
We now consider the last case and first introduce some properties of the following cubic polynomial:
    \begin{equation}\label{eq:pi}
        \pi(\lambda) = \gamma_B(\lambda-\gamma_E)+(1-\lambda)(\lambda-\gamma_A)\gamma_C+\lambda(\lambda-\gamma_A)(\lambda-\gamma_E).
    \end{equation}
    Observe that $\pi(\lambda)$ can be rewritten in the form              \begin{align}\label{eq:relation_poly}
		\pi(\lambda) & = \lambda^3 - (\gamma_A+\gamma_C+\gamma_E)\lambda^2 + (\gamma_B+\gamma_C+\gamma_A\gamma_C+\gamma_A\gamma_E)\lambda - \gamma_A\gamma_C-\gamma_B\gamma_E \nonumber \\
		& = \lambda^3 - (\gamma_A+\gamma_C)\lambda^2 + (\gamma_B+\gamma_C+\gamma_A\gamma_C)\lambda - \gamma_A\gamma_C-\gamma_E p(\lambda) \nonumber\\
		& = \lambda^3 -\gamma_A
		\lambda^2 + \gamma_B \lambda-\gamma_E p(\lambda) - [\gamma_C \lambda^2-(\gamma_C+\gamma_A\gamma_C)\lambda+\gamma_A\gamma_C] \nonumber\\
		& = (\lambda-\gamma_E)p(\lambda)+\gamma_C(1-\lambda)(\lambda-\gamma_A),
    \end{align}
where $p(\lambda)$ is given in \eqref{eq:p_lam}. This shows that any real root of $\pi(\lambda)=0$ must satisfy $\lambda>0$. Otherwise, if $\lambda\le 0$, then, by the positive definiteness of $\widetilde{A}$ and $\widetilde{E}$ and \eqref{eq:eig-gamma}, we have $\lambda-\gamma_E<0$, $1-\lambda>0$, and $\lambda-\gamma_A<0$. Combining these inequalities with $p(\lambda)>0$ for $\lambda\le 0$ yields $\pi(\lambda)<0$, contradicting the fact that $\lambda$ is a root of $\pi(\lambda)$.
    
\begin{comment}
    Hence, if $\pi(\lambda)$ has three real roots, we denote them, in nondecreasing order, by 
$$ 0 < \mu_a(\gamma_A,\gamma_B,\gamma_C,\gamma_E) \le \mu_b(\gamma_A,\gamma_B,\gamma_C,\gamma_E)\le\mu_u(\gamma_A,\gamma_B,\gamma_C,\gamma_E),$$ 
abbreviated as $0<\mu_a\le\mu_b\le\mu_u$ whenever no ambiguity arises. If $\pi(\lambda)$ has one real root and a pair of complex conjugate roots, we denote the real root by $\mu(\gamma_A,\gamma_B,\gamma_C,\gamma_E) > 0$, abbreviated as $\mu>0$ if no ambiguity. According to \eqref{eq:relation_poly}, we obtain
    \begin{equation*}
        \lim_{\lambda \rightarrow -\infty} \pi ( \lambda )  =-\infty  \quad  \quad \lim_{\lambda \rightarrow +\infty} \pi ( \lambda )  =+\infty, \quad \text{and} \quad \pi(0) = -\gamma_B\gamma_E-\gamma_A\gamma_C < 0.
    \end{equation*}
\end{comment}

We denote the minimum real root of $\pi(\lambda)$ by $\mu_ l(\gamma_A,\gamma_B,\gamma_C,\gamma_E)$ and the maximum real root by $\mu_u(\gamma_A,\gamma_B,\gamma_C,\gamma_E)$; for brevity, we use $0< \mu_l \le \mu_u $ when there is no ambiguity. According to \eqref{eq:relation_poly}, we obtain $\lim_{\lambda \rightarrow -\infty} \pi ( \lambda )  =-\infty$, $\lim_{\lambda \rightarrow +\infty} \pi ( \lambda )  =+\infty$, and $\pi(0) = -\gamma_B\gamma_E-\gamma_A\gamma_C < 0$.

We recall the following result from \citet{bergamaschi2025spectral}, whose proof is based on the KKT conditions of an associated optimization problem.
	
\begin{lem}\cite[Lemma 2.2]{bergamaschi2025spectral}\label{lem:optim}
		Let $q(\lambda; \gamma)$ be a polynomial in $\lambda$ that depends on the parameter $\gamma = (\gamma_1, \cdots, \gamma_d)$, and $\gamma_j \in [\,\gamma_{\min}^j, \, \gamma_{\max}^j\,]$, $j = 1, \cdots, d$. If $\lambda$ satisfies $q(\lambda; \gamma)=0$, partial derivative $\frac{\partial q}{\partial \lambda}(\lambda; \gamma) \neq 0$, and $\lambda$ is a local extremum, then precisely one of the three cases below is valid:
		\begin{enumerate}[label=(\arabic*)]
			\item $\delta \frac{\partial q}{\partial \gamma_j}(\lambda; \gamma) \ge 0$ and $\gamma_j = \gamma_{\min}^j$,\\
			\item $\delta \frac{\partial q}{\partial \gamma_j}(\lambda; \gamma) \le 0$ and $\gamma_j = \gamma_{\max}^j$,\\
			\item $\delta \frac{\partial q}{\partial \gamma_j}(\lambda; \gamma) = 0$ and $\gamma_j \in (\gamma_{\min}^j, \gamma_{\max}^j)$,
		\end{enumerate}
		where $\delta \in \left\{ -1,1 \right\} $ is defined by 
		$
		\delta =\begin{cases}
			-\operatorname{sgn} \left( \frac{\partial q(\lambda ;\gamma)}{\partial \lambda} \right) ,& \text{if $\lambda$ is a local minimum},\\
			+\operatorname{sgn} \left( \frac{\partial q(\lambda ;\gamma )}{\partial \lambda} \right) ,&	\text{if $\lambda$ is a local maximum}.
		\end{cases}
		$
\end{lem}

\begin{lem}\label{lem:bound3}
		Under the same assumptions as in \Cref{lem:bound1}, if $z \notin \nul(\bar{C})$, then $\lambda\in\left[ \, \rho_l,\, \rho_u \, \right ]$, where
        \begin{align}
            & \rho_l := \min \Bigl\{\,    \frac{\gamma_{\min}^B}{\gamma_{\max}^A} ,\, \mu_l( \gamma _{\max}^{A},\gamma _{\min}^{B},\gamma _{\min}^{C},\gamma _{\max}^{E}),\,  \mu_l( \gamma _{\max}^{A},\gamma _{\min}^{B},\gamma _{\min}^{C},\gamma _{\min}^{E} ) \, \Bigl\}, \label{eq:bound_lower} \\
            & \rho_u:=\max \bigl\{ \, \gamma_{\max}^{A}, \, \mu _u( \gamma _{\min}^{A},\gamma _{\min}^{B},\gamma _{\max}^{C},\gamma _{\max}^{E} ), \, \mu _u( \gamma _{\min}^{A},\gamma _{\min}^{B},\gamma _{\max}^{C},\gamma _{\min}^{E} ) \,\bigl\} \label{eq:bound_upper}.
        \end{align}
	\end{lem}

	\begin{proof}
    By substituting \eqref{eq:inexact_relation} into \eqref{eq:eigfunction_inexact3} and premultiplying by $z^T/(z^T z)$, we obtain
		\begin{equation*}
			\frac{z^T\bar{B}(\lambda I-\widetilde{A})^{-1}\bar{B}^Tz}{z^Tz} +(1-\lambda)\frac{z^T\bar{C}^T(\lambda I - \widetilde{E})^{-1}\bar{C}z}{z^Tz} + \lambda = 0.
		\end{equation*}
		Let $u = \bar{B}^T z$ and $v = \bar{C} z$. The above equation can be rewritten as
		\begin{equation*}
			\frac{u^T(\lambda I-\widetilde{A})^{-1}u}{u^Tu} \cdot \frac{z^T\bar{B}\bar{B}^Tz}{z^Tz} +(1-\lambda)\frac{v^T(\lambda I - \widetilde{E})^{-1}v}{v^Tv}\cdot \frac{z^T\bar{C}^T\bar{C}z}{z^Tz} + \lambda = 0,
		\end{equation*}
		that is, $\frac{\gamma_B}{\lambda-\gamma_A}+(1-\lambda)\frac{\gamma_C}{\lambda-\gamma_E}+\lambda =: \frac{\pi(\lambda)}{(\lambda-\gamma_A)(\lambda-\gamma_E)}=0$, where $\pi(\lambda)$ is defined in \eqref{eq:pi}. Hence, $\lambda$ is a real root of $\pi(\lambda)$, which implies $\lambda > 0$. We now derive bounds for $\lambda$ in three steps.

\textbf{Step I: Upper bounds for $\lambda$ when $\pi(\lambda)$ has three real roots.} By \Cref{lem:bound1}, it is sufficient to restrict our attention to the case $\mu_u > \gamma_A$, $\mu_u > \gamma_E$, and $\mu_u > 1$. Indeed, otherwise, $\mu_u \le \max \{ \gamma _{\max}^{A}, \gamma _{\max}^{E} \}$ under the assumption $\gamma _{\max}^{A}>1$. Furthermore, it follows from \eqref{eq:relation_poly} that
\begin{align*}
			\pi(\lambda)&=(\lambda-\gamma_E)p(\lambda)+\gamma_C(1-\lambda)(\lambda-\gamma_A)\\
            &=(\lambda-\gamma_E)(\lambda^2-\gamma_A\lambda+\gamma_B)+\gamma_C(1-\lambda)(\lambda-\gamma_A)\\
			& =\lambda(\lambda-\gamma_E)(\lambda-\gamma_A)+\gamma_B(\lambda-\gamma_E)+\gamma_C(1-\lambda)(\lambda-\gamma_A)\\
			& = (\lambda-\gamma_A)[\lambda(\lambda-\gamma_E) + \gamma_C(1-\lambda)]+\gamma_B(\lambda-\gamma_E)
			 = (\gamma_A-\lambda)\frac{\partial \pi}{\partial \gamma _A}+\gamma_B(\lambda-\gamma_E).
\end{align*}
This implies that
{\small\begin{equation}\label{eq:partial}
			\frac{\partial \pi}{\partial \gamma _A} \!=\! \frac{\pi(\lambda)+\gamma_B(\gamma_E-\lambda)}{\gamma_A-\lambda}, \quad \frac{\partial \pi}{\partial \gamma _B} \!=\! \lambda - \gamma_E, \quad \frac{\partial \pi}{\partial \gamma _C} \!=\! (1-\lambda)(\lambda-\gamma_A), \quad \frac{\partial \pi}{\partial \gamma _E} \!=\! -p(\lambda).
		\end{equation}}
		Therefore, we can easily obtain 
		\begin{equation}\label{eq:mu_c_2}
		    \begin{aligned}
			\frac{\partial \pi}{\partial \gamma_A}\bigl(\mu_u\bigr) & = \frac{\gamma_B(\gamma_E-\mu_u)}{\gamma_A-\mu_u}> 0, &
			\frac{\partial \pi}{\partial \gamma _B}\bigl(\mu_u\bigr) &= \mu_u - \gamma_E >0,\\
			\frac{\partial \pi}{\partial \gamma _C}\bigl(\mu_u\bigr) &= (1-\mu_u)(\mu_u-\gamma_A) <0, & \frac{\partial \pi}{\partial \gamma _E}\bigl(\mu_u\bigr) &= -p(\mu_u).
		\end{aligned}
		\end{equation}
If $p(\mu_u)=0$, i.e., $\mu_u$ is a root of $p(\lambda)$, then it follows from \Cref{lem:bound2} that $\mu_u \le \gamma_{\max}^A$. Otherwise, $p(\mu_u)\neq 0$. Since $\mu_u$ is the largest real root of the cubic polynomial $\pi(\lambda)$, we get $\frac{\partial \pi}{\partial \lambda}\bigl(\mu_u\bigr) >0$. This, together with \eqref{eq:mu_c_2} and \Cref{lem:optim}, yields that $\mu_u \le \max \bigl\{\,\mu_u(\gamma^A_{\min},\gamma^B_{\min},\,\gamma^C_{\max},\gamma^E_{\max}),\mu_u(\gamma^A_{\min},\gamma^B_{\min},\gamma^C_{\max},\gamma^E_{\min})\,\bigl\}$. Combining this two cases, we obtain
        \begin{equation}\label{eq:mu_c}
            \begin{aligned}
                \mu_u \le \max \bigl\{ & \, \gamma_{\max}^{A}, \, \mu_u(\gamma^A_{\min},\gamma^B_{\min},\gamma^C_{\max},\gamma^E_{\max}),   \, \mu_u(\gamma^A_{\min},\gamma^B_{\min},\gamma^C_{\max},\gamma^E_{\min})\, \bigl\}.
            \end{aligned}
        \end{equation}

\textbf{Step II: Lower bounds for $\lambda$ when $\pi(\lambda)$ has three real roots.} Likewise, in view of \Cref{lem:bound1}, it suffices to consider the case where $\mu_l < \gamma_A$, $\mu_l < \gamma_E$, and $\mu_l < 1$. Otherwise, since $\gamma _{\min}^{A}<1$, we have $\mu_l \ge \min \{  \gamma _{\min}^{A}, \, \gamma _{\min}^{E} \}$, which already provides the desired lower bound. Then from \eqref{eq:partial}, it leads to
		\begin{equation}\label{eq:mu_a_2}
		    \begin{aligned}
			\frac{\partial \pi}{\partial \gamma_A}\bigl(\mu_l\bigr) & = \frac{\gamma_B(\gamma_E-\mu_l)}{\gamma_A-\mu_l}> 0, &
			\frac{\partial \pi}{\partial \gamma _B}\bigl(\mu_l\bigr) &= \mu_l - \gamma_E <0,\\
			\frac{\partial \pi}{\partial \gamma _C}\bigl(\mu_l\bigr) &= (1-\mu_l)(\mu_l-\gamma_A) <0, & \frac{\partial \pi}{\partial \gamma _E}\bigl(\mu_l\bigr) &= -p(\mu_l).
		\end{aligned}
		\end{equation}
If $p(\mu_l)=0$, then \Cref{lem:bound2} gives $\mu_l \ge \gamma_{\min}^B/\gamma_{\max}^A$. On the other hand, if $p(\mu_l)\neq0$, then, by the fact that $\mu_l$ is the smallest real root of $\pi(\lambda)$, we have $\frac{\partial \pi}{\partial \lambda}\bigl(\mu_l\bigr)>0$. Therefore, \Cref{lem:optim} and \eqref{eq:mu_a_2} yield $\mu_l \ge \min \bigl \{\,\mu_l(\gamma^A_{\max},\gamma^B_{\min},\gamma^C_{\min},\gamma^E_{\max}),\,\mu_l(\gamma^A_{\max},\gamma^B_{\min},\gamma^C_{\min},\gamma^E_{\min})\, \bigl\}$. By considering both cases, we arrive at
        \begin{equation}\label{eq:mu_a}
                \mu_l \ge \min \Bigl \{  \, \frac{\gamma_{\min}^B}{\gamma_{\max}^A}, \,\mu_l(\gamma^A_{\max},\gamma^B_{\min},\gamma^C_{\min},\gamma^E_{\max}),\,\mu_l(\gamma^A_{\max},\gamma^B_{\min},\gamma^C_{\min},\gamma^E_{\min})\,  \Bigl \}.
        \end{equation}

\textbf{Step III: Lower and upper bounds for $\lambda$ when $\pi(\lambda)$ has one real root and a pair of complex conjugate roots.} In this case, we have $\mu_l = \mu_u$. Proceeding as in Steps I and II, the unique real root satisfies
    \begin{align*}
        & \min  \Bigl \{ \,  \frac{\gamma_{\min}^B}{\gamma_{\max}^A}, \, \mu_l(\gamma^A_{\max},\gamma^B_{\min},\gamma^C_{\min},\gamma^E_{\max}),\,\mu_l(\gamma^A_{\max},\gamma^B_{\min},\gamma^C_{\min},\gamma^E_{\min})\, \Bigl \} \le \mu_l = \mu_u\\
        & \le \max \bigl\{ \, \gamma_{\max}^{A}, \, \mu_u(\gamma^A_{\min},\gamma^B_{\min},\gamma^C_{\max},\gamma^E_{\max}),   \, \mu_u(\gamma^A_{\min},\gamma^B_{\min},\gamma^C_{\max},\gamma^E_{\min})\, \bigl\}.
    \end{align*}
This along with \eqref{eq:mu_c} and \eqref{eq:mu_a} follows the result.
\end{proof}

By combining \Cref{lem:bound1,lem:bound2,lem:bound3}, we obtain the following result.

\begin{theorem}\label{theorem:inexact1}
		Assume that $A$ and $E$ are SPD, and that $B$ and $C$ have full row rank. For any given SPD matrices $\widehat{A}$, $\widehat{E}$ and $Q$, all the real eigenvalues of $\calP^{-1}\calA$ lie in the interval
		$\left[ \, \min\{\rho_l,\,\gamma_{\min}^E\},\, \max\{\rho_u,\,\gamma_{\max}^E\} \, \right ]$, where $\rho_l$ and $\rho_u$ are defined in \eqref{eq:bound_lower} and \eqref{eq:bound_upper}, respectively.
	\end{theorem}

\subsection{Bounds for nonreal eigenvalues}\label{sec:spectral-2}

In this subsection, we focus on the distribution of the complex eigenvalues of $\calP^{-1}\calA$ with nonzero imaginary parts. 

Recall the matrices $\widetilde{A}=\widehat{A}^{-\frac{1}{2}}A\widehat{A}^{-\frac{1}{2}}$, $\widetilde{E}=\widehat{E}^{-\frac{1}{2}}E\widehat{E}^{-\frac{1}{2}}$, $\bar{B}=Q^{-\frac{1}{2}}B\widehat{A}^{-\frac{1}{2}}$, and $\bar{C}=\widehat{E}^{-\frac{1}{2}}CQ^{-\frac{1}{2}}$. Then, by \eqref{eq:double saddle} and \eqref{eq:pre_inexact}, one readily verifies that
		 \begin{equation*}
		 	\calP^{-1}\calA = \left( \begin{matrix}
		 		\hat{A}^{-1}A&		0&		\hat{A}^{-1}B^T\\
		 		\hat{E}^{-1}CQ^{-1}B&		\hat{E}^{-1}E+\hat{E}^{-1}CQ^{-1}C^T&		\hat{E}^{-1}C\\
		 		-Q^{-1}B&		-Q^{-1}C^T&		0\\
		 	\end{matrix} \right),
		 \end{equation*}
which is similar to $N = \left( \begin{smallmatrix}
		 		\widetilde{A}&		0&		\bar{B}^T\\
		 		\bar{C}\bar{B}&		\widetilde{E}+\bar{C}\bar{C}^T&		\bar{C}\\
		 		-\bar{B}&		-\bar{C}^T&		0\\
		 	\end{smallmatrix} \right)
            =N_1+N_2$
with
	\begin{equation}\label{eq:N1 and N2}
		N_1=\left( \begin{matrix}
			\widetilde{A}&		\frac{1}{2}\bar{B}^T\bar{C}^T&		0\\
			\frac{1}{2}\bar{C}\bar{B}&		\widetilde{E}+\bar{C}\bar{C}^T&		0\\
			0&		0&		0\\
		\end{matrix} \right) \quad \text{and} \quad   N_2=\left( \begin{matrix}
			0&		-\frac{1}{2}\bar{B}^T\bar{C}^T&		\bar{B}^T\\
			\frac{1}{2}\bar{C}\bar{B}&		0&		\bar{C}\\
			-\bar{B}&		-\bar{C}^T&		0\\
		\end{matrix} \right).
	\end{equation}
Clearly, $N_1$ and $N_2$ are the symmetric and skew-symmetric parts of $N$, respectively. We therefore proceed by studying the spectral properties of $N_1$ and $N_2$. To this end, we first recall Weyl's inequality.

\begin{lem} \cite[Theorem 4.3.1]{horn2012matrix}\label{lem:Weyl}
    Let $A, B \in \mathbb{C}^{n \times n}$ be Hermitian matrices, and let their eigenvalues be ordered nonincreasingly, i.e., $\lambda_1(\cdot) \ge \lambda_2(\cdot) \ge \cdots \ge \lambda_n(\cdot)$. Then, for every $k = 1, 2, \dots, n$, $\lambda_k(A) + \lambda_n(B) \le \lambda_k(A + B) \le \lambda_k(A) + \lambda_1(B)$.
\end{lem}

    	\begin{lem}\label{lem:imag_1}
		Under the same assumptions of \Cref{theorem:inexact1}, the eigenvalues of $N_1$ are either zero or lie in the interval $[\, \omega_l,\, \omega_u \,]$ with $\omega_l= \min \left\{ \gamma _{\min}^{A}, \gamma _{\min}^{E}+\gamma _{\min}^{C} \right\} -\frac{1}{2}\sigma _{\max}(\bar{B}^T\bar{C}^T)$ and $\omega_u=\max \left\{ \gamma _{\max}^{A}, \gamma _{\max}^{E}+\gamma _{\max}^{C} \right\} +\frac{1}{2}\sigma _{\max}(\bar{B}^T\bar{C}^T)$.
		%where $\sigma _{\max}(\bar{B}^T\bar{C}^T)$ is the largest singular value of $\bar{B}^T\bar{C}^T$.
	\end{lem}
	
	\begin{proof}
		From \eqref{eq:N1 and N2}, it is clear  that the eigenvalues of $N_1$ are either zero or coincide with those of
		\begin{equation*}
			\left( \begin{matrix}
				\widetilde{A}&		\frac{1}{2}\bar{B}^T\bar{C}^T\\
				\frac{1}{2}\bar{C}\bar{B}&		\widetilde{E}+\bar{C}\bar{C}^T\\
			\end{matrix} \right) = \left( \begin{matrix}
			\widetilde{A}&		0\\
		0 &	\widetilde{E}+\bar{C}\bar{C}^T\\
			\end{matrix} \right)+ \left( \begin{matrix}
			0&		\frac{1}{2}\bar{B}^T\bar{C}^T\\
			\frac{1}{2}\bar{C}\bar{B}&		0\\
			\end{matrix} \right).
		\end{equation*}
		This together with Weyl's inequality in \Cref{lem:Weyl} yields the desired result.
	\end{proof}

    \begin{lem}\label{lem:imag_2}
    Under the same assumptions of \Cref{theorem:inexact1}, every eigenvalue $\lambda({\rm i} N_2)$ of ${\rm i} N_2$ satisfies $\left|\lambda({\rm i} N_2)\right|\le\max\left\{\, \frac{1}{2} \sqrt{\gamma _{\max}^{B}\gamma _{\max}^{C}},\, \sqrt{\gamma _{\max}^{B}},\, \sqrt{\gamma _{\max}^{C}} \,\right \}$.
	\end{lem}
	
	\begin{proof}
    For any vector $v=(x;y;z)\in \mathbb{C}^{n+p+m}$ with $\|v\|=1$, by \eqref{eq:eig-max}, we have
        \begin{align*}
                &\left| v^*N_2v \right|  =\left| \frac{1}{2}\left( y^*\bar{C}\bar{B}x-x^*\bar{B}^T\bar{C}^Ty \right) +\left( x^*\bar{B}^Tz-z^*\bar{B}x \right) +\left( y^*\bar{C}z-z^*\bar{C}^Ty \right) \right|\\
                & \le \frac{1}{2}\left| y^*\bar{C}\bar{B}x \right|+\frac{1}{2}\left| x^*\bar{B}^T\bar{C}^Ty \right|+\left| x^*\bar{B}^Tz \right|+\left| z^*\bar{B}x \right|+\left| y^*\bar{C}z \right|+\left| z^*\bar{C}^Ty \right|\\
                & =  \left| y^*\bar{C}\bar{B}x \right|+2\left| z^*\bar{B}x \right|+2\left| y^*\bar{C}z \right|
                \le \| \bar{B} \| \, \| \bar{C} \|\, \| x \|\, \| y \| +2\| \bar{B} \|\, \| x \|\, \| z \| + 2\| \bar{C} \|\, \| y \| \, \| z \| \\
                &\le \max\left\{ \frac{1}{2}\| \bar{B} \| \, \| \bar{C} \|, \, \| \bar{B} \|, \, \| \bar{C} \| \right\} (2\| x \| \, \| y \| + 2\| x \| \, \| z \| +2\| y \| \, \| z \|)\\
                &\le\max\left\{\, \frac{1}{2}\| \bar{B} \| \, \| \bar{C} \| ,\, \| \bar{B} \|,\, \| \bar{C} \| \,\right\} (\| x \| + \| y \| + \| z \|)^2\\
                &=\max\left\{\, \frac{1}{2}\| \bar{B} \| \, \| \bar{C} \|,\, \| \bar{B} \|,\, \| \bar{C} \| \,\right\} = \max\left\{\, \frac{1}{2} \sqrt{\gamma _{\max}^{B}\gamma _{\max}^{C}},\, \sqrt{\gamma _{\max}^{B}},\, \sqrt{\gamma _{\max}^{C}} \,\right \}.
\end{align*}
This along with the fact that ${\rm i} N_2$ is a Hermitian matrix and the Courant-Fischer theorem leads to 
        \begin{equation*}
            \left|\lambda({\rm i} N_2)\right|
            \le \max_{\|v\|=1}\left| v^*({\rm i}N_2)v \right|
            =\max_{\|v\|=1}\left| v^*N_2v \right|
            \le \max\left\{\, \frac{1}{2} \sqrt{\gamma _{\max}^{B}\gamma _{\max}^{C}},\, \sqrt{\gamma _{\max}^{B}},\, \sqrt{\gamma _{\max}^{C}} \,\right \}.
        \end{equation*}
		This completes the proof.	
	\end{proof}

\begin{theorem}\label{theorem:inexact2}
	Under the same assumptions of \Cref{theorem:inexact1}, every nonreal eigenvalue $\lambda$ of $\calP^{-1}\calA$ satisfies
	\begin{equation*}
		\Re{(\lambda)} \in [\,\min \{ \omega_l,0 \},\, \omega_u\,] \quad \text{and} \quad \left| \Im (\lambda ) \right|\le \max\left\{\, \tfrac{1}{2} \sqrt{\scriptstyle\gamma _{\max}^{B}\gamma _{\max}^{C}},\, \sqrt{\scriptstyle\gamma _{\max}^{B}},\, \sqrt{\scriptstyle\gamma _{\max}^{C}} \,\right \},
	\end{equation*}
	where $\omega_l$ and $\omega_u$ are given in \Cref{lem:imag_1}.
\end{theorem}
	
\begin{proof}
Note that $\calP^{-1}\calA$ is similar to $N$. Thus, $\lambda$ is also an eigenvalue of $N$. Let $\zeta$ be an eigenvector of $N$ corresponding to $\lambda$. Since $N=N_1+N_2$, where $N_1$ is symmetric and $N_2$ is skew-symmetric, it follows that $\Re{(\lambda)} = \frac{\zeta^*N_1\zeta}{\zeta^*\zeta}$ and $\Im{(\lambda)} = -\frac{\zeta^*({\rm i}N_2)\zeta}{\zeta^*\zeta}$. This along with the Courant-Fischer theorem leads to $\lambda_{\min}(N_1) \le \Re{(\lambda)} \le \lambda_{\max}(N_1)$ and $-\lambda_{\max}({\rm i}N_2) \le \Im{(\lambda)} \le -\lambda_{\min}({\rm i}N_2)$.
Combining with \Cref{lem:imag_1,lem:imag_2} completes the proof.
\end{proof}

\begin{rem}
We emphasize that, although the double saddle-point system can be formally rewritten as a $2\times2$ saddle-point system, the spectral analysis presented in this work is not based on the aggregated formulation \eqref{eq:2*2}. The block triangular preconditioner derived from the aggregated system in \cite{simoncini2004block} contains both $B$ and $C$ in its off-diagonal block, whereas the proposed preconditioner preserves the original hierarchical block structure and only involves $C$ in the corresponding block. Consequently, the preconditioned matrices arising from the proposed preconditioner are structurally different from those associated with the classical $2\times2$ saddle-point preconditioners in \cite{simoncini2004block}. Therefore, the existing spectral bounds for standard $2\times2$ saddle-point preconditioners cannot be directly invoked to characterize the proposed preconditioner.
\end{rem}

    \section{Randomized construction of \texorpdfstring{$Q$}{}}\label{sec:select}

As discussed in \Cref{sec:precon}, the matrix $Q$ should provide a good approximation to $BA^{-1}B^T+C^TE^{-1}C$. However, explicitly forming $BA^{-1}B^T+C^TE^{-1}C$ is computationally expensive and may incur prohibitive storage requirements. To address this issue, we propose a randomized strategy for low‑cost construction of $Q$.

Randomized matrix approximation techniques can efficiently capture the essential spectral information of a matrix using substantially fewer degrees of freedom, thereby offering substantial savings in both computational and storage costs \cite{halko2011finding,martinsson2020randomized}. To exploit these advantages, we employ a hybrid strategy that combines diagonal approximation with randomized low-rank approximation to select $Q$. Because $BA^{-1}B^T$ and $C^TE^{-1}C$ share the same algebraic structure, we present the randomized approximation procedure and the corresponding theoretical analysis only for $C^TE^{-1}C$. The low‑cost construction of $BA^{-1}B^T$ follows analogously and is omitted for brevity.
    
Let $E_D = {\rm diag}(E)$ denote the diagonal part of $E$. Then $C^TE^{-1}C 
        = C^TE_D^{-1}C + C^TE^{-1}C-C^TE_D^{-1}C
        = C^TE_D^{-1}C + \Delta_E$, 
where $\Delta_E = C^T(E^{-1}-E_D^{-1})C$ represents the approximation error incurred by replacing
$C^TE^{-1}C$ with $C^TE_D^{-1}C$.
The matrix $C^TE_D^{-1}C$ is computationally inexpensive to construct and apply, but it ignores the off-diagonal information contained in $E$. Although the error matrix $\Delta_E$ captures these neglected effects, its explicit computation is generally prohibitively expensive. To strike a balance between computational efficiency and approximation quality, we seek a computationally tractable approximation $\widehat{\Delta}_E$ to $\Delta_E$. Consequently, we obtain
    \begin{equation}\label{eq:CEC_approx}
		C^TE^{-1}C \approx C^TE_D^{-1}C + \widehat{\Delta}_E.
	\end{equation}

To obtain a computationally efficient approximation of $\Delta_E$, we employ a randomized low-rank approximation technique to extract its dominant spectral information. Let $\Omega \in \mathbb{R}^{m \times k}$ be a sparse sketch matrix\footnote{A sketch matrix is a random projection matrix that maps a high-dimensional matrix into a lower-dimensional space while preserving its essential properties with high probability.}, where $k \ll m$. We begin by forming the sample matrix
	\begin{equation*}
		W = \Delta_E \Omega 
        = C^T\Big[E^{-1}(C\Omega)-E_D^{-1}(C\Omega)\Big]
        =: C^T(Y-Y_D).
	\end{equation*} 
In numerical implementation, $Y$ can be computed approximately by solving the linear systems $EY = C\Omega$ using an incomplete Cholesky factorization, whereas $Y_D$ can be obtained efficiently by exploiting the diagonal structure of $E_D$. Since the columns of $W$ are random samples of the action of $\Delta_E$, the range of $W$ is expected to capture the dominant spectral information of $\Delta_E$. Consequently, an orthonormal basis for ${\rm range}(W)$ can be used to approximate the dominant eigenspace of $\Delta_E$.

To this end, we compute the thin QR factorization $W = VR$, where $V\in\mathbb{R}^{m\times k}$ has orthonormal columns and
$R\in\mathbb{R}^{k\times k}$ is upper triangular. The columns of $V$ form an orthonormal basis for the sampled subspace associated with the dominant spectral components of $\Delta_E$. We then construct a symmetric low-rank approximation of $\Delta_E$ in the form $\widehat{\Delta}_E=VHV^T$, where $H\in\mathbb{R}^{k\times k}$ is chosen such that $\widehat{\Delta}_E$ reproduces the action of $\Delta_E$ on the sampled subspace as accurately as possible. Specifically, we require $VHV^T\Omega \approx W$. This condition ensures that $\widehat{\Delta}_E$ captures the dominant behavior of $\Delta_E$ along the random probing directions represented by $\Omega$. By enforcing the above approximation condition, we obtain $H$ by solving the least-squares problem $\min\limits_{H}\| HV^T\Omega - V^TW\|_F^2$. A direct calculation then leads to the following expression:
\begin{equation}\label{eq:barH}
		H = (V^TW)(\Omega^TW)^{\dagger}(V^TW)^T,
\end{equation}
where $(\cdot)^{\dagger}$ denotes the Moore--Penrose pseudoinverse. In practice, to guarantee invertibility and enhance numerical stability, we compute $H$ via a regularized formulation with a small parameter $\varepsilon>0$, i.e., $H = (V^TW)(\Omega^TW+\varepsilon I)^{-1}(V^TW)^T$. Finally, together with \eqref{eq:CEC_approx}, we obtain 
\begin{equation*}
        C^TE^{-1}C \approx C^TE_D^{-1}C + \widehat{\Delta}_E = C^TE_D^{-1}C+VHV^T.
\end{equation*}
The procedure for approximating $C^T E^{-1} C$ is summarized in \Cref{algorithm:Q}.

\begin{algorithm}[htbp]
	\caption{Randomized Low-Rank Approximation of $C^TE^{-1}C$.}
	\label{algorithm:Q}
	\begin{algorithmic}[1]
		\STATE \textbf{Input:} 
			  $E \in \mathbb{R}^{p \times p}$ is SPD, 
			$C \in \mathbb{R}^{p \times m}$ is full row rank, $k$ is sketch size, and $\varepsilon>0$ is a regularization parameter.

			\STATE 
			\textbf{Form diagonal-based approximation:} Compute $E_D = {\rm diag}(E)$.
			
			\STATE
			\textbf{Draw random sketch matrix:}
			$\Omega \in \mathbb{R}^{m \times k}$.
			
			\STATE 
			\textbf{Form sample matrix $W$:} Compute the incomplete Cholesky factor $E_L$ of $E$, $C_\Omega = C \Omega$, $Y = E_L' \setminus (E_L \setminus C_\Omega)$, $Y_D = E_D^{-1}C_\Omega$, and $W = C^T(Y-Y_D)$.
			
			\STATE 
			\textbf{Compute thin QR decomposition:}
			$[V, \sim] = \text{qr}(W,0)$.
			
			\STATE
			\textbf{Construct low-rank approximation $H$:} Compute $Z = \Omega^TW$, $M = V^T W$, and $H = M(Z+\varepsilon I)^{-1}M^T$.
			
		\STATE \textbf{Output:} 
			$C^TE_D^{-1}C +  VHV^T \approx C^TE^{-1}C $.
	\end{algorithmic}
\end{algorithm}

The approximation of $BA^{-1}B^{T}$ can be carried out in the same manner as in \Cref{algorithm:Q}. Consequently, the construction of $Q$ reduces to the separate approximation of $C^{T}E^{-1}C$ and $BA^{-1}B^{T}$ using a hybrid strategy that combines diagonal and randomized low-rank approximations. This approach significantly reduces the computational cost while preserving approximation accuracy, thereby making it particularly suitable for large-scale sparse problems.

From \Cref{algorithm:Q}, it is evident that the quality of the final approximation to the target matrix $C^TE^{-1}C$ is determined by how accurately the randomized error-correction matrix $\widehat{\Delta}_E = VHV^T$ captures the error matrix $\Delta_E$. Note that $W=\Delta_E\Omega=VR$, giving $V^TW=V^T\Delta_E\Omega=R$. Combining this with \eqref{eq:barH} yields
\begin{equation}\label{eq:Nystrom}
    \begin{aligned}
        \widehat{\Delta}_E & = VHV^T = V(V^TW)(\Omega^TW)^{\dagger}(V^TW)^TV^T = VR(\Omega^TW)^{\dagger}R^TV^T\\
        & = \Delta_E\Omega(\Omega^T\Delta_E\Omega)^{\dagger}(\Delta_E\Omega)^T.
    \end{aligned}
\end{equation}
Thus, our approximation (\Cref{algorithm:Q}) is closely related to the Nystr{\"o}m method \cite{halko2011finding,martinsson2020randomized}, a widely used approach for low-rank approximation of symmetric positive semidefinite (SPSD) matrices. 

While the theoretical properties of the Nystr{\"o}m method in the SPSD setting are well established \cite{halko2011finding,martinsson2020randomized}, these analyses rely heavily on matrix square roots in error decompositions and norm estimates. Consequently, they cannot be directly extended to the present setting, where $\Delta_E$ is symmetric but may be indefinite. To date, theoretical investigations of the Nystr{\"o}m method and its variants for symmetric indefinite matrices remain relatively limited. Relevant developments include approaches based on core-matrix truncation \cite{nakatsukasa2023randomized,cai2022fast} and submatrix shifting \cite{ray2022sublinear}. Inspired by the analytical framework developed in \cite{nakatsukasa2020fast,nakatsukasa2023randomized}, we overcome the limitations of the SPSD theory and establish rigorous error bounds for the approximation produced by \Cref{algorithm:Q}. In particular, we prove that the approximation error is bounded in expectation conditioned on a high-probability event. To this end, we first review some preliminaries.

\begin{lem}\cite[Lemma 2.3.3]{golub2013matrix}\label{lem:Nue}
If a matrix $T$ satisfies $\|T\|<1$, then $I-T$ is nonsingular, and $(I-T)^{-1}=\sum\limits_{i=0}^{\infty}T^i$, with $\|(I-T)^{-1}\|\le \frac{1}{1-\|T\|}$.
\end{lem}

\begin{lem}\cite[Theorem 1.1]{rudelson2009smallest}\label{lem:small_singular_gauss}
Let $G\in\mathbb{R}^{m\times n}$ with $m\ge n$ be a random matrix whose entries are independent, mean-zero, sub-Gaussian random variables with unit variance. Then, for every $\varepsilon>0$,        \begin{equation*}
     \mathbb{P}\left( \sigma_{\min}(G) \le \varepsilon\big(\sqrt{m}-\sqrt{n-1}\big) \right)
            \le  (c_1\varepsilon )^{m-n+1}+e^{-c_2m},
\end{equation*}
where $c_1,\,c_2>0$ are constants depending only polynomially on the sub-Gaussian moment parameter.
\end{lem}

\begin{lem}\cite[Corollary 7.3.3]{vershynin2018high}\label{lem:large_singular_gauss}
Let $G\in\mathbb{R}^{m\times n}$ be a random matrix whose entries are independent standard Gaussian random variables $N(0,1)$. Then, for any $t\ge 0$,
\begin{equation*}
 \mathbb{P}\left\{\|G\| \ge \sqrt{m}+\sqrt{n}+t\right\} \le 2 e^{-c_3t^2},
\end{equation*}
where $c_3>0$ is an absolute constant independent of $m, n, t$.
\end{lem}

Let $\mathcal{E}$ denote the approximation error of \Cref{algorithm:Q}. From \eqref{eq:CEC_approx} and \eqref{eq:Nystrom}, we have
\begin{equation}\label{eq:err}
         \mathcal{E} = \Delta_E- \widehat{\Delta}_E = \Delta_E-\Delta_E\Omega(\Omega^T\Delta_E\Omega)^{\dagger}(\Delta_E\Omega)^T.
\end{equation}
Let the eigendecomposition of $\Delta_E$ be given by
     \begin{equation}\label{eq:eigendecomp}
    \Delta_E = U\Lambda U^T
    = \left( \begin{matrix}
        U_1 & U_2 
    \end{matrix} \right)
    \left( \begin{matrix}
        \Lambda_1 & 0  \\
        0 & \Lambda_2  \\
    \end{matrix} \right)
    \left( \begin{array}{c}
	    U_{1}^{T}\\
	    U_{2}^{T}\\
    \end{array} \right),
    \end{equation}
where $U\in \mathbb{R}^{m \times m}$ is orthogonal, $\Lambda={\rm diag}\{\lambda_1(\Delta_E),\ldots,\lambda_m(\Delta_E)\}$, with eigenvalues of $\Delta_E$ ordered such that  $|\lambda_1(\Delta_E)|\ge|\lambda_2(\Delta_E)|\ge\ldots\ge|\lambda_m(\Delta_E)|$, and $U_1\in\mathbb{R}^{m \times k}$ and $\Lambda_1\in\mathbb{R}^{k\times k}$ with the sketch size $k$. Throughout the remainder of this section, we assume that $\lambda_k(\Delta_E) \neq 0$ and that $\Omega \in \mathbb{R}^{m \times k}$ is a standard Gaussian random matrix. The condition $\lambda_k(\Delta_E) \neq 0$ ensures that $\Lambda_1$ is nonsingular. Let $\Omega_1=U_1^T\Omega \in \mathbb{R}^{k \times k}$ and $\Omega_2=U_2^T\Omega \in \mathbb{R}^{(m-k) \times k}$. By rotational invariance of Gaussian matrices, $\Omega_1$ and $\Omega_2$ are independent standard Gaussian matrices \cite{halko2011finding}. Moreover, $\Omega_1$ is nonsingular almost surely \cite{tao2023topics,vershynin2018high}.

\begin{lem}\label{lem:dagger}
   Let $K=\Omega_2 \Omega_1^{-1}$ and $\widehat{K}=\Lambda_1+K^T\Lambda_2 K$. Define the event $\Psi = \left\{ \| \left| \Lambda _2 \right|^{\frac{1}{2}}K \| ^{2}\le 0.5\left| \lambda _k\left( \Delta _E \right) \right| \right\}$. Then, conditioned on $\Psi$, both $\widehat{K}$ and $\Omega^{T}\Delta_E\Omega$ are nonsingular almost surely. 
\end{lem}

     \begin{proof}
          Note that $\Omega_2=K\Omega_1$ and $\widehat{K}=\Lambda_1+K^T\Lambda_2 K$, it follows from \eqref{eq:eigendecomp} that
     \begin{align}
             &\Omega^T\Delta_E\Omega  = \Omega^T  U \Lambda U^T \Omega = \left( \begin{matrix} \Omega^TU_1 & \Omega^TU_2  \end{matrix} \right) 
             \left( \begin{matrix}
              \Lambda_1 & 0  \\
               0 & \Lambda_2  \\
               \end{matrix} \right)
               \left( \begin{array}{c}
               U_{1}^{T}\Omega\\
               U_{2}^{T}\Omega\\
               \end{array} \right) \nonumber\\
               & = \Omega_1^T \Lambda_1 \Omega_1 + \Omega_2^T \Lambda_2 \Omega_2  =  \Omega_1^T \Lambda_1 \Omega_1 + \Omega_1^T K^T \Lambda_2 K \Omega_1 = \Omega_1^T \widehat{K} \Omega_1.\label{eq:Omgea_Delta}
     \end{align}
Since $\Omega_1$ is nonsingular almost surely, it suffices to show that $\widehat{K}$ is nonsingular. Conditioning on $\Psi$, we obtain
    \begin{align*}
        \|\Lambda _{1}^{-1} K^T\Lambda _2K \| & \le \| \Lambda _{1}^{-1} \|  \, \| K^T\Lambda_2K \| \le \| \Lambda _{1}^{-1} \| \,  \| \left| \Lambda _2 \right|^{\frac{1}{2}}K \|^{2} \\
        & \le \frac{1}{|\lambda_k(\Delta_E)|}\cdot 0.5|\lambda_k(\Delta_E)|=0.5<1.
    \end{align*}
This, together with \Cref{lem:Nue}, yields that $I + \Lambda_{1}^{-1} K^T \Lambda_{2} K$ is nonsingular. Using $\widehat{K} = \Lambda_1 + K^T \Lambda_2 K = \Lambda_1 (I + \Lambda_{1}^{-1} K^T \Lambda_{2} K)$ and the fact that $\Lambda_1$ is nonsingular, we conclude that $\widehat{K}$ is nonsingular, which completes the proof.
\end{proof}

\begin{lem} \label{lem:project}
Conditioned on the event $\Psi$, let $P=\Lambda U^T \Omega(\Omega^T\Delta_E\Omega)^{-1}\Omega^TU$. Then $U^T\mathcal{E}U = (I-P)\Lambda(I-U^T\Omega M)$ holds for any $M \in \mathbb{R}^{k \times m}$, where $\Psi$, $\mathcal{E}$, and $U$ are defined in \Cref{lem:dagger}, \eqref{eq:err} and \eqref{eq:eigendecomp}, respectively.
\end{lem}     

     \begin{proof}
      From \eqref{eq:eigendecomp}, we get 
         \begin{align*}
             (I-P)\Lambda U^T \Omega & = \Lambda U^T\Omega-P\Lambda U^T \Omega = \Lambda U^T\Omega- \Lambda U^T \Omega(\Omega^T\Delta_E\Omega)^{-1}\Omega^TU\Lambda U^T \Omega\\
             & = \Lambda U^T\Omega-\Lambda U^T\Omega=0.
         \end{align*}
         This implies that $(I-P)\Lambda U^T\Omega M=0$ holds for any $M \in \mathbb{R}^{k \times m}$. Using \eqref{eq:err} and \eqref{eq:eigendecomp}, it leads to
         \begin{align*}
             \mathcal{E} &= \Delta_E-\Delta_E\Omega(\Omega^T\Delta_E\Omega)^{-1}(\Delta_E\Omega)^T = U\Lambda U^T-U\Lambda U^T\Omega(\Omega^T\Delta_E\Omega)^{-1}\Omega^TU\Lambda U^T\\
             & = U \left[\Lambda -\Lambda U^T\Omega(\Omega^T\Delta_E\Omega)^{-1}\Omega^TU\Lambda \right] U^T = U (\Lambda - P \Lambda) U^T = U (I - P ) \Lambda U^T.
         \end{align*}
          Then we can derive that 
          \begin{equation*}
              U^T\mathcal{E}U = (I - P ) \Lambda = (I-P)\Lambda(I-U^T\Omega M),
          \end{equation*}
          which follows the result.
     \end{proof}

     \begin{lem}\label{lem:Neumann}
Let $K$ and the event $\Psi$ be defined in \Cref{lem:dagger}. Then, conditioned on $\Psi$, we have $\| ( I+K^T\Lambda _2K\Lambda _{1}^{-1} ) ^{-1} \| <2$.
\end{lem}

\begin{proof}
Since $\| \Lambda _{1}^{-1} \| = |\lambda_k(\Delta_E)|^{-1}$ and $\| K^T\Lambda _2K \|  \le\| \left| \Lambda _2 \right|^{\frac{1}{2}}K \|^{2}$, conditioned on the event $\Psi$, it holds that
\begin{equation*}
    \| K^T\Lambda _2K\Lambda _{1}^{-1} \| \le \| K^T\Lambda _2K \|  \, \| \Lambda _{1}^{-1} \|  \le 0.5.
\end{equation*}
According to \Cref{lem:Nue}, we have
\begin{equation*}
    \| ( I+K^T\Lambda _2K\Lambda _{1}^{-1} ) ^{-1}  \| \le \frac{1}{1- \| K^T\Lambda _2K\Lambda _{1}^{-1} \| }
    \le \frac{1}{1-0.5}=2,
\end{equation*}
which completes the proof.
     \end{proof}

Leveraging the error characterization in \eqref{eq:Nystrom}, the following theorem establishes a bound for the error incurred by \Cref{algorithm:Q} in approximating $\Delta_E$.

\begin{theorem}\label{theorem:err}
Let $\Psi$ be the event defined in \Cref{lem:dagger}, and let $\Lambda_2$ be given by \eqref{eq:eigendecomp}. Conditioned on the event $\Psi$, the approximation error $\mathcal{E}$ generated by \Cref{algorithm:Q} satisfies
\begin{align*}
\mathbb{E} \left[  \| \mathcal{E} \| \,|\,\Psi \right] & \le\left| \lambda_k\left( \Delta_E \right) \right| + 2\sqrt{2} \| \left| \Lambda_2 \right|^{\frac{1}{2}} \| \sqrt{\left| \lambda_k \left( \Delta_E \right) \right|}+2 \| \Lambda_2 \|.
\end{align*}
\end{theorem}

     \begin{proof}
         From \eqref{eq:Omgea_Delta}, it is easy to compute             \begin{align}\label{eq:comput_P}
                 P & = \Lambda U^T \Omega(\Omega^T\Delta_E\Omega)^{-1}\Omega^TU = \Lambda U^T \Omega(\Omega_1^T \widehat{K} \Omega_1)^{-1}\Omega^TU \nonumber\\
                 & = \Lambda U^T \Omega \Omega_1^{-1}\widehat{K}^{-1}\Omega_1^{-T}\Omega^TU 
                 =  \left( \begin{matrix}
                 \Lambda _1&		0\\
                 0&		\Lambda _2\\
                 \end{matrix} \right) \left( \begin{array}{c}
                 \Omega _1\\
                 \Omega _2\\
                 \end{array} \right) \Omega _{1}^{-1}\widehat{K}^{-1}\Omega _{1}^{-T}\left( \begin{matrix}
                 \Omega _{1}^{T}&		\Omega _{2}^{T}\\
                 \end{matrix} \right) \nonumber \\
                 & = \left( \begin{array}{c}
                 \Lambda _1\\
                 \Lambda _2K\\
                 \end{array} \right) \widehat{K}^{-1}\left( \begin{matrix}
                 I&		K^T\\
                 \end{matrix} \right) = \left( \begin{matrix}
                 \Lambda _1\widehat{K}^{-1} &		\Lambda_1\widehat{K}^{-1}K^T \\
                 \Lambda_2K\widehat{K}^{-1} &		\Lambda_2K\widehat{K}^{-1}K^T\\ \end{matrix} \right). 
            \end{align}
         Let $M = \left( \begin{matrix} \Omega _{1}^{-1}& 0\\ \end{matrix} \right) \in \mathbb{R}^{k \times m}$, then we have
         \begin{equation*}
             U^T\Omega M = \left( \begin{array}{c} 
             \Omega _1\\
             \Omega _2\\
             \end{array} \right) \left( \begin{matrix} \Omega _{1}^{-1}& 0\\ \end{matrix} \right) = \left( \begin{matrix}
	           I&		0\\
	           K&		0\\
            \end{matrix} \right).
         \end{equation*}
This, together with \Cref{lem:project} and \eqref{eq:comput_P}, yields that
             \begin{align}\label{eq:UEU}
                 U^T\mathcal{E}U  &= (I-P)\Lambda(I-U^T\Omega M) \nonumber \\
                 & = \left( \begin{matrix}
                 \Lambda _1-\Lambda _1\widehat{K}^{-1}\Lambda _1&		-\Lambda _1\widehat{K}^{-1}K^T\Lambda _2\\
                 -\Lambda _2K\widehat{K}^{-1}\Lambda _1&		\Lambda _2-\Lambda _2K\widehat{K}^{-1}K^T\Lambda _2\\ \end{matrix} \right)  \left( \begin{matrix}
	                0&		0\\
	                -K&		I\\
                 \end{matrix} \right) \nonumber \\
                 & = \left( \begin{matrix}
                 \Lambda _1\widehat{K}^{-1}K^T\Lambda _2K&		-\Lambda _1\widehat{K}^{-1}K^T\Lambda _2\\
                 -\Lambda _2K+\Lambda _2K\widehat{K}^{-1}K^T\Lambda _2K&		\Lambda _2 -\Lambda _2K\widehat{K}^{-1}K^T\Lambda _2\\ \end{matrix} \right) := \left( \begin{matrix} \mathcal{E} _1&		\mathcal{E} _2\\ 
                 \mathcal{E} _3&		\mathcal{E} _4\\ \end{matrix} \right).
             \end{align}
By the properties of the spectral norm and the fact that $\Lambda_2$ is diagonal, it gives
\begin{equation}\label{norm2-p1}
    \| K^T \Lambda _2K \| = \| K^T |\Lambda _2|^{\frac{1}{2}} \operatorname{sgn} (\Lambda_2) |\Lambda _2|^{\frac{1}{2}}  K \| \le \| |\Lambda _2|^{\frac{1}{2}}K \|^2
\end{equation}
and
\begin{equation}\label{norm2-p2}
\| K^T\Lambda _2 \| = \| \Lambda _2 K \|
= \| |\Lambda _2|^{\frac{1}{2}}\operatorname{sgn} (\Lambda_2)|\Lambda _2|^{\frac{1}{2}}K   \| \le  \|  |\Lambda _2|^{\frac{1}{2}}\|\,\| |\Lambda _2|^{\frac{1}{2}}K \|.
\end{equation}
Here and in what follows, $\operatorname{sgn}(\Lambda_2)$ denotes the diagonal matrix of the same dimension as $\Lambda_2$, whose diagonal entries are given by the signs of the corresponding diagonal entries of $\Lambda_2$.

In view of $\widehat{K}=\Lambda_1+K^T\Lambda_2K$, by \eqref{eq:UEU}, \eqref{norm2-p1}, and \Cref{lem:Neumann}, it implies that, conditioned on the event $\Psi$,
             \begin{align}\label{eq:E1}
	               \| \mathcal{E} _1 \| & =  \| \Lambda_1\widehat{K}^{-1}K^T\Lambda_2K \| =  \| \Lambda _1 ( \Lambda _1+K^T\Lambda _2K ) ^{-1}K^T\Lambda _2K \|\nonumber\\
	               & = \| ( I+K^T\Lambda _2K\Lambda _{1}^{-1} ) ^{-1}K^T\Lambda _2K \|\nonumber\\
                   & \le \| ( I+K^T\Lambda _2K\Lambda _{1}^{-1} ) ^{-1} \| \,  \| K^T\Lambda _2K \| \le 2 \| |\Lambda _2|^{\frac{1}{2}}K  \|^{2}.
              \end{align}
Similarly, from \eqref{norm2-p2}, we obtain
        \begin{align} \label{eq:E2}
              \| \mathcal{E} _2 \|  & = \| -\Lambda _1\widehat{K}^{-1}K^T\Lambda _2\| =  \| \Lambda _1 ( \Lambda _1+K^T\Lambda _2K ) ^{-1}K^T\Lambda _2 \| \nonumber\\
	          & \le \| ( I+K^T\Lambda _2K\Lambda _{1}^{-1} ) ^{-1} \| \, \| K^T\Lambda _2 \| 
              \le  2 \| |\Lambda _2|^{\frac{1}{2}} \| \,  \| |\Lambda _2|^{\frac{1}{2}}K \|. 
        \end{align}
Using \Cref{lem:Neumann} again, it gives
\begin{align*}
 \| \widehat{K}^{-1}\|
=\|\Lambda_1^{-1} ( I+K^T\Lambda_2K\Lambda_1^{-1} ) ^{-1}\|
\le \|\Lambda_1^{-1}\| \, \|( I+K^T\Lambda_2K\Lambda_1^{-1} ) ^{-1}\| 
\le 2 \|\Lambda_1^{-1}\|.
\end{align*}
This, along with \eqref{eq:UEU}, \eqref{norm2-p1} and \eqref{norm2-p2}, leads to
\begin{align*} 
\| \mathcal{E} _3 \| & 
\le \| \Lambda _2K \| +  \| \Lambda _2K \widehat{K}^{-1}K^T\Lambda _2K \|
\le \| \Lambda _2K \| +  \| \Lambda _2K\| \,\|\widehat{K}^{-1}\| \,\|K^T\Lambda _2K \|\\
&\le \| |\Lambda _2|^{\frac{1}{2}} \| \, \| |\Lambda _2|^{\frac{1}{2}}K \| + 2 \|\Lambda_1^{-1}\| \, \| |\Lambda _2|^{\frac{1}{2}}\| \,  \| |\Lambda _2|^{\frac{1}{2}}K\|^{3}
\end{align*}
and
\begin{align*}
\| \mathcal{E} _4 \| & \le  \| \Lambda _2 \| +  \| \Lambda _2K\widehat{K}^{-1}K^T\Lambda _2 \|
\le  \| \Lambda _2 \| +  \| \Lambda _2K\|\,\|\widehat{K}^{-1}\|\,\| K^T\Lambda _2 \|\\
& \le \| \Lambda _2 \| + 2 \|\Lambda_1^{-1}\| \, \| |\Lambda _2|^{\frac{1}{2}} \| ^{2}  \,  \| |\Lambda _2|^{\frac{1}{2}}K \|^{2}.
\end{align*}
Therefore, combining the above two inequalities with \eqref{eq:UEU}, \eqref{eq:E1}, \eqref{eq:E2}, and conditioning on the event $\Psi$, we arrive at
\begin{align*}
&\mathbb{E}  \left[ \| \mathcal{E} \|\,|\,\Psi \right]  \le \mathbb{E} \left[ \| \mathcal{E}_1 \|\,|\,\Psi \right] + \mathbb{E} \left[  \| \mathcal{E}_2 \| \,|\,\Psi \right] + \mathbb{E} \left[ \| \mathcal{E}_3 \| \,|\,\Psi \right] + \mathbb{E} \left[ \| \mathcal{E}_4 \| \,|\,\Psi \right]\\
& \le \left| \lambda _k\left( \Delta _E \right) \right| + 2 \| |\Lambda _2|^{\frac{1}{2}} \| \sqrt{0.5\left| \lambda _k\left( \Delta _E \right) \right|} +  \| |\Lambda_2|^{\frac{1}{2}}\|\sqrt{0.5 \left| \lambda _k\left( \Delta _E \right) \right|} \\
&\quad +2 \|\Lambda_1^{-1}\| \, \| |\Lambda _2|^{\frac{1}{2}} \| \big(0.5\left| \lambda _k\left( \Delta _E \right) \right|\big)^{\frac{3}{2}} + \| \Lambda _2 \| + \|\Lambda_1^{-1}\| \| |\Lambda _2|^{\frac{1}{2}}\|^2 \left| \lambda _k\left( \Delta _E \right) \right|\\
& = \left| \lambda_k\left( \Delta_E \right) \right| + 2\sqrt{2} \| \left| \Lambda_2 \right|^{\frac{1}{2}} \| \sqrt{\left| \lambda_k \left( \Delta_E \right) \right|}+2 \| \Lambda_2 \| ,
\end{align*}
where the last equality holds by $\| \Lambda_{1}^{-1} \|=\left| \lambda_k ( \Delta_E ) \right|^{-1}$. This completes the proof.
\end{proof}

\begin{rem}
As shown in \Cref{theorem:err}, the approximation error admits a favorable norm bound whenever the tail singular values of $\Delta_E$ decay sufficiently rapidly.
\end{rem}

We now turn to the probability that the event $\Psi$ occurs. For any given constant $t>0$, define the following two events: 
$$\Psi_1 = \left\{ \| \Omega _{1}^{-1} \| \le t \right\} \qquad {\rm and} \qquad 
\Psi_2 = \left\{ \| \Omega _2 \| \le \frac{\sqrt{0.5\left| \lambda _k\left( \Delta _E \right) \right|}}{t \| \left| \Lambda _2 \right|^{\frac{1}{2}} \|} \right\}.$$
Note that $\Omega_1\in\mathbb{R}^{k \times k}$ is a standard Gaussian matrix. Then, by \Cref{lem:small_singular_gauss}, we obtain
    \begin{equation*}
        \mathbb{P} \left( \sigma _{\min}\left( \Omega _1 \right) \le \frac{1}{t} \right) \le \frac{c_1}{t\left( \sqrt{k}-\sqrt{k-1} \right)}+e^{-c_2k}.
    \end{equation*}    
This along with $\| \Omega _{1}^{-1} \|=\frac{1}{\sigma_{\min}(\Omega_1)}$ leads to
        \begin{align}\label{eq:psi_1}
            \mathbb{P} (\Psi_1 ) & =  \mathbb{P}\left( \| \Omega _{1}^{-1} \| \le t \right) =  \mathbb{P} \left(\frac{1}{\sigma_{\min}(\Omega_1)} \le t \right) = \mathbb{P} \left( \sigma_{\min}(\Omega_1) \ge \frac{1}{t} \right) \nonumber\\
            & = 1 - \mathbb{P} \left( \sigma _{\min}\left( \Omega _1 \right) \le \frac{1}{t} \right) \ge 1 - \frac{c_1}{t\left( \sqrt{k}-\sqrt{k-1} \right)}-e^{-c_2k}.
        \end{align}
In addition, it follows from \Cref{lem:large_singular_gauss} and the fact that $\Omega_2 \in \mathbb{R}^{(m-k)\times k}$ is also a standard Gaussian matrix that $\mathbb{P} \left( \left\| \Omega _2 \right\| \le \sqrt{m-k}+\sqrt{k}+t \right) \ge 1-2e^{-c_3t^2}$. Hence,
    \begin{equation}\label{eq:psi_2}
        \mathbb{P} (\Psi_2 )  =  \mathbb{P}\left( \| \Omega _2 \| \le \frac{\sqrt{0.5\left| \lambda _k\left( \Delta _E \right) \right|}}{t \| \left| \Lambda _2 \right|^{\frac{1}{2}} \|} \right) \ge 1-2e^{-c_3t^2}
    \end{equation}
provided by
    \begin{equation}\label{eq:event_con}
        \sqrt{0.5\left| \lambda _k\left( \Delta _E \right) \right|}\ge t\left( \sqrt{m-k}+\sqrt{k}+t \right) \| \left| \Lambda _2 \right|^{\frac{1}{2}}\|.
    \end{equation}
According to $K=\Omega_2\Omega _{1}^{-1}$ in \Cref{lem:dagger}, we have $\| \left| \Lambda_2 \right|^{\frac{1}{2}}K \|=  \| \left| \Lambda_2 \right|^{\frac{1}{2}}\Omega_2\Omega _{1}^{-1} \|   \le \| \left| \Lambda_2 \right|^{\frac{1}{2}} \| \, \left\| \Omega _2 \right\| \, \| \Omega _{1}^{-1} \|$, which implies $\Psi_1 \cap \Psi_2 \subseteq \Psi$. Therefore, if \eqref{eq:event_con} holds, by \eqref{eq:psi_1}, \eqref{eq:psi_2}, and the independence of $\Omega_1$ and $\Omega_2$, we know that
        \begin{align*}
            \mathbb{P} ( \Psi ) & \ge \mathbb{P}  ( \Psi _1 ) \, \mathbb{P}  ( \Psi _2 ) \ge \left(1 - \frac{c_1}{t\left( \sqrt{k}-\sqrt{k-1} \right)}-e^{-c_2k} \right) \left(1-2e^{-c_3t^2} \right) \\
            & \ge 1 - \frac{c_1}{t\left( \sqrt{k}-\sqrt{k-1} \right)}-e^{-c_2k}-2e^{-c_3t^2}.
        \end{align*}
This indicates that the event $\Psi$ holds with high probability under rapid decay of the tail singular values of $\Delta_E$.

\begin{rem}
    To further illustrate the applicability of condition $\Psi$, we perform a Monte Carlo study for a class of symmetric indefinite matrices. Specifically, we construct an $m=500$-dimensional matrix $\Delta_E=U\Lambda U^T$, where $U$ is a random orthogonal matrix. The eigenvalues are generated as follows: for $1\leq i\leq395$, $\lambda_i$ is assigned either $(0.05)^{i-1}$ or $-(0.05)^{i-1}$ with equal probability, resulting in a spectrum with both positive and negative eigenvalues and rapidly decaying magnitudes. To avoid numerical underflow caused by extremely small eigenvalues, we set $\lambda_i=0$ for $396\leq i\leq500$. We define $\mathrm{ratio}= \big\| |\Lambda_2|^{1/2}K \big\|^2  / |\lambda_k(\Delta_E)|$, so that condition $\Psi$ is satisfied if $\mathrm{ratio}\leq0.5$.

A Monte Carlo simulation is conducted with a fixed random seed $\mathrm{rng}(0)$. We generate $1000$ independent Gaussian sketch matrices and retain the samples satisfying $\sigma_{\min}(\Omega_1)\geq0.2$ to ensure that $\Omega_1$ is reasonably well-conditioned. For $k=10$, approximately $57\%$ of the retained samples satisfy $\mathrm{ratio}\leq0.5$, demonstrating that condition $\Psi$ can occur with non-negligible probability. We further consider $1\leq k\leq50$ and compute the empirical probability $P_k$ of event $\Psi$ for each $k$. The histogram of $\mathrm{ratio}$ for $k=10$ and the variation of $P_k$ with respect to $k$ are presented in \Cref{fig:event}.
    
As expected, the empirical probability gradually decreases as $k$ increases, since $|\lambda_k(\Delta_E)|$ becomes smaller for larger $k$. Nevertheless, event $\Psi$ occurs with positive probability over the entire tested range of $k$, indicating that condition $\Psi$ is not overly restrictive. 
    \begin{figure}[htb]
    	\centering
    	\subfigbottomskip=2pt
    	\subfigure{
    		\includegraphics[width=0.47 \linewidth]{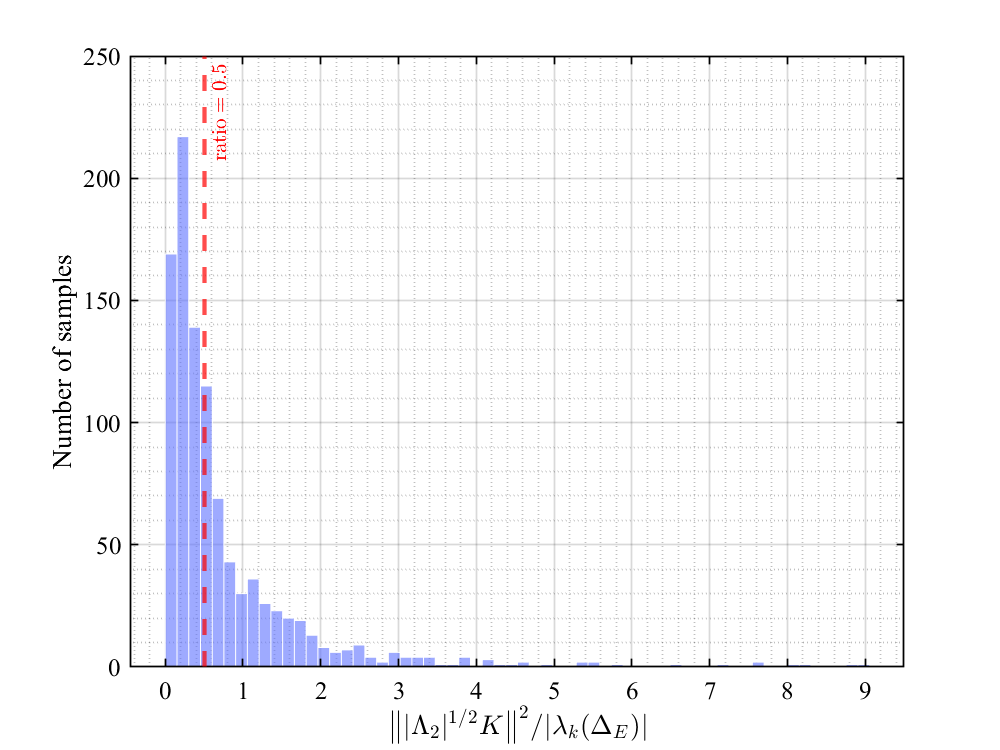}
    	}
    	\subfigure{
    		\includegraphics[width=0.47 \linewidth]{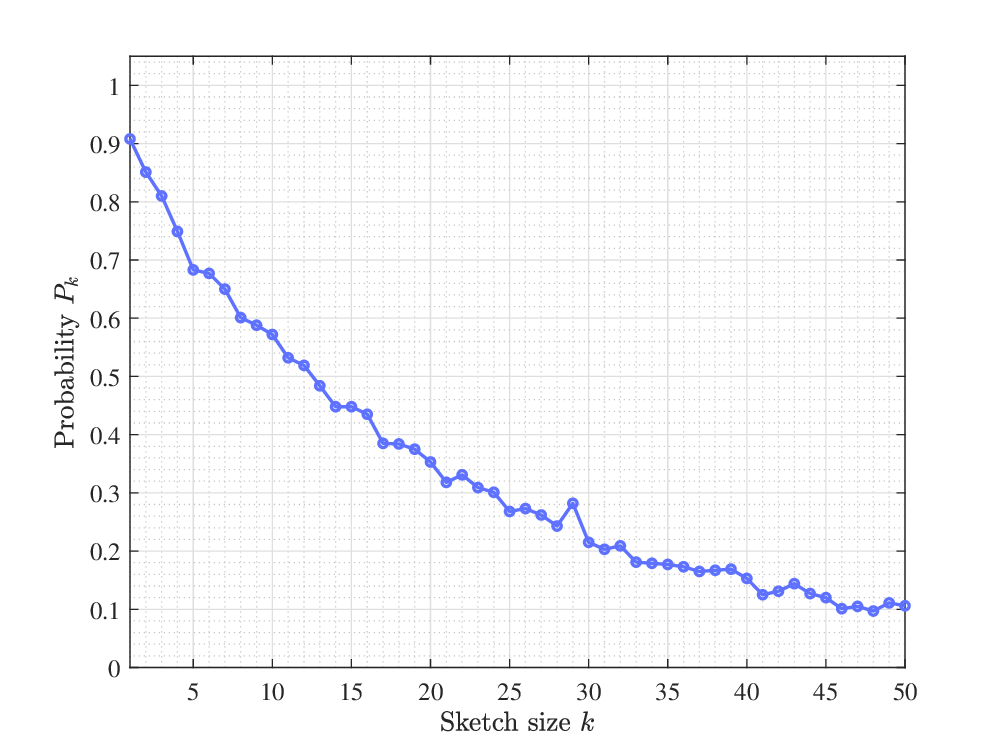}
    	}
    	\caption{Monte Carlo results: the distribution of $\mathrm{ratio}$ for $k=10$ (left) and the empirical probability of event $\Psi$ for different values of $k$ (right).}
    	\label{fig:event}
    \end{figure}
    \end{rem}

\section{Numerical experiments}\label{sec:numres}
	
In this section, numerical experiments are conducted to assess the effectiveness and robustness of the proposed preconditioners through comparisons with existing methods. Furthermore, the viability of the randomized subblock selection strategy presented in \Cref{algorithm:Q} is examined and validated. All experiments were performed on a laptop equipped with an Intel(R) Core(TM) i7-10750H CPU (2.60 GHz) and 16.0 GB of RAM, and all computations were implemented in MATLAB R2020a.
	
In our experiments, the preconditioned GMRES method is used to solve the saddle-point systems \eqref{eq:double saddle}-\eqref{eq:double saddle_equal} and \eqref{eq:double saddle_equal2}. The test problems are selected in accordance with the problem classes for which the preconditioners in the literature were originally designed. We compare our preconditioner $\calP$ in \eqref{eq:pre_inexact} (denoted ``IMD'') and its randomized subblock selection strategy (denoted ``RIMD'') with $\calP_{\rm DS}$ in \eqref{eq:pre_DS} (denoted ``DS''), $\calP_{\rm RDF}$ proposed in \cite{benzi2011relaxed} (denoted ``RDF''), $\calP_{\rm GSS}$ in \cite{ahmad2026class} (denoted ``GSS''), $\calP_{\rm BD}$ in \eqref{eq:pre_BD} (denoted ``BD''), and $\calP_{\rm BPP}$ studied in \cite{bergamaschi2025spectral} (denoted ``BPP''), where
	\begin{align*}
		\calP_{\rm RDF}&\!=\!\frac{1}{\alpha} \! \left( \begin{matrix}
			A&		0&		B^T\\
			0&		\alpha I&		0\\
			-B&		0&		\alpha I\\
		\end{matrix} \right) \! \left( \begin{matrix}
			\alpha I&		0&		0\\
			0&		E&		C\\
			0&		-C^T&		\alpha I\\
		\end{matrix} \right)\!, \, \calP_{\rm GSS} \!=\!\left( \begin{matrix}
			\alpha P+\omega A&		0&		\omega B^T\\
			0&		\beta Q+\omega E&		\omega C\\
			-\omega B&		-\omega C^T&		\tau R\\
		\end{matrix} \right),\\
		\calP_{\rm BPP} &=\left( \begin{matrix}
			\widehat{A}&		0&		0\\
			B&		-\widehat{S}&		0\\
			0&		C&		\widehat{X}\\
		\end{matrix} \right) \left( \begin{matrix}
			\widehat{A}&		0&		0\\
			0&		\widehat{S}&		0\\
			0&		0&		\widehat{X}\\
		\end{matrix} \right)^{-1}\left( \begin{matrix}
			\widehat{A}&		B^T&		0\\
			0&		-\widehat{S}&		C^T\\
			0&		0&		\widehat{X}\\
		\end{matrix} \right).
	\end{align*}	
Here, $\alpha$, $\beta$, $\tau$, $\omega$ are positive parameters, and $P$, $Q$, $R$ are SPD matrices. In $\calP_{\rm BPP}$, $\widehat{A}$, $\widehat{S}$, and $\widehat{X}$ being SPD approximations of $A$, $S=B\widehat{A}^{-1}B^T$, and $X = E + C\widehat{S}^{-1}C^T$, respectively. Moreover, to verify the necessity of designing preconditioners by exploiting the special structure of $\calA$, we also compare our preconditioners with $\calP_{\rm HSS}$ (denoted ``HSS'') in \cite{benzi2004preconditioner}, $\calP_{\rm D}$ (denoted ``DIAG'') in \cite{benzi2008some}, $\calP_{\rm T}$ (denoted ``TBD'') in \cite{benzi2008some}, and $\calP_{\rm TPSS}$ (denoted ``TPSS'') in \cite{song2022two} that are developed for the $2 \times 2$ block saddle-point system in \eqref{eq:2*2}. Here,
\begin{align*}
    & \calP_{\rm HSS} =\frac{1}{2\alpha}\left(
     \begin{matrix}
      \alpha I+A_1 & 0\\
      0 & \alpha I
      \end{matrix}\right)
      \left(\begin{matrix}
      \alpha I & B_1^T\\
      -B_1 & \alpha I
      \end{matrix}\right), \\
      & \calP_{\rm D}=\left( \begin{matrix}
	          A_1&		0\\
	        0&		S_1\\
        \end{matrix} \right), \quad \calP_{\rm T} = \left( \begin{matrix}
	        A_1&		B_1^T\\
	         0&		S_1\\
        \end{matrix} \right),
    \quad \text{and} \quad
        \calP_{\rm TPSS} = \left( \begin{matrix}
	          P_1+\tau A_1&		tB_1^T\\
	        tB_1&		-Q_1\\
        \end{matrix} \right)
\end{align*}
    with $S_1=B_1A^{-1}B_1^T$, $P_1$ and $Q_1$ being SPD, and $\alpha$, $\tau$, and $t$ being positive constants.

In all tests, we set $\widehat{A} = A_LA^T_L$ in $\calP$ and $\calP_{\rm BPP}$, where $A_L$ is the drop tolerance-based incomplete Cholesky factorization of $A$ produced by the MATLAB function \textit{ichol(A, struct(`type', `ict', `droptol', 1e-02, `michol', `on'))}. Similarly, we take $\widehat{E} = E_LE^T_L$ in $\calP$, where $E_L$ denotes the incomplete Cholesky factor of $E$, generated by the MATLAB function \textit{`ichol'}, with the parameters in the function being identical to those used for computing $A_L$. For the construction of $Q$ in $\mathcal{P}$, IMD directly sets $Q = S + C^T \mathrm{diag}(E)^{-1}C$, whereas in RIMD, $Q = S + C^TE_D^{-1}C +  VHV^T$, where $C^TE_D^{-1}C +  VHV^T$ is computed according to \Cref{algorithm:Q} with $k=10$, $\varepsilon=10^{-8}$, and $\Omega$ chosen as a sparse Gaussian random matrix. To facilitate efficient evaluation of $H$, we slightly relax the accuracy requirement in Step 6 of \Cref{algorithm:Q} and set $H = M(Z+\varepsilon I)^{-1}$. Following \cite{ahmad2026class}, the parameters in $\mathcal{P}_{\rm GSS}$ are chosen as $\alpha = \beta = 10^{-2}$, $\tau = 10^{-3}$, $P=A$, $Q = CC^T$, and $R = I$. For $\mathcal{P}_{\rm TPSS}$, we take $P_1 = 0.5{\rm diag}(A_1)$, $Q_1 = 5B_1 {\rm diag}(A_1)^{-1}B^T$, and $\tau = t = 1$ according to \cite{song2022two}. All involved SPD linear subsystems, such as those with coefficient matrices $\alpha I + A$, $\alpha I + E$, $\alpha I + S$, $\alpha I + B(\alpha I + A)^{-1}B^T$, $\alpha I + C^T(\alpha I + E)^{-1}C$, $C^T E^{-1} C$, and $E + C S^{-1} C^T$, are solved via Cholesky factorization.
	
The zero vector is used as the initial guess, and the iteration is terminated when either the iteration count exceeds $10^5$ or the relative residual at the $k$-th iteration satisfies
${\rm RES}:= \left\| b-Aw_k \right\| /\left\| b \right\| \le 10^{-8}$.
For brevity, we denote the number of iterations, CPU time (in seconds), and relative residual by ``IT'', ``CPU'', and ``RES'', respectively.

\begin{ex}\label{ex1}
\rm{\textbf{The Poisson control problem \cite{rees2010optimal,ahmad2026class,bradley2023eigenvalue}.}} Consider the distributed Poisson control problem:
\begin{align}
	\min_{u,f} \quad &\frac{1}{2} \|u - \widehat{u}\|_{L_2(\varGamma)}^2 + \beta \|f\|_{L_2(\varGamma)}^2 \label{eq:ex1_obj}\\
	\text{s.t.} \quad &-\nabla^2 u = f \quad \text{in } \varGamma, \label{eq:ex1_cons1}\\
	&\qquad~\, u = g \quad \text{on } \partial\varGamma \label{eq:ex1_cons2}, 
\end{align}
		where function $\widehat{u}$ is the known desired state, $f$ is the control, $\varGamma = [0,1] \times [0,1]$ is the domain with boundary $\partial\varGamma$, and the regularization parameter $\beta$ is set to $10^{-2}$, $10^{-3}$, and $10^{-5}$. We aim to find a state $u$ that satisfies the constraints \eqref{eq:ex1_cons1}-\eqref{eq:ex1_cons2} and is as close as possible to $\widehat{u}$ in the sense of the $L_2$ norm. For systems \eqref{eq:ex1_obj}-\eqref{eq:ex1_cons2} with Dirichlet boundary conditions, we discretize them using the Galerkin finite element method and apply the Lagrange multiplier method to obtain the following linear system:
		\begin{equation}\label{eq:ex1_dis}
			 \left( \begin{matrix}
				2\beta M&		0&		-M\\
				0&		M&		K^T\\
				-M&		K&	0\\
			\end{matrix} \right) \left( \begin{array}{c}
				\mathbf{f}\\
				\mathbf{u}\\
				\lambda \\
			\end{array} \right) =\left( \begin{array}{c}
				0\\
				\mathbf{b}\\
				\mathbf{d}\\
			\end{array} \right),	
		\end{equation}    
		where $M$ and $K$ are SPD mass matrix and stiffness matrix, $\mathbf{u}$ and $\mathbf{f}$ are the discrete forms of $u$ and $f$, respectively; $\lambda$ denotes the Lagrange multiplier; $\mathbf{b}$ and $\mathbf{d}$ are the constant vectors derived from $\widehat{u}$ and the boundary conditions, respectively. The MATLAB codes in \cite{tyronerees} are adopted to generate the test matrices. In the tests, we consider the mesh refinement level $\ell$ and the corresponding problem size $n+p+m$ as follows:
        \begin{center}
	    \begin{tabular}{ c r r r r}
		\toprule
		$\ell$ & $n$ & $p$ & $m$ & $n+p+m$\\
		\midrule
		$2^5$ & 961 & 961 & 961 & 2,883\\
		$2^6$ & 3,969 & 3,969 & 3,969 & 11,907\\
		$2^7$ & 16,129 & 16,129 & 16,129 & 48,387\\
		$2^8$ & 65,025 & 65,025 & 65,025 & 195,075\\
		$2^9$ & 261,121 & 261,121 &  261,121 & 783,363\\
		\bottomrule
	\end{tabular}
\end{center}
		\end{ex}

        \begin{table}[H]\small
			\caption{Numerical results for saddle-point system from \Cref{ex1} with $\beta = 10^{-2}$.}
			\centering
			\label{table:ex1_result}
			\setlength{\tabcolsep}{2.5mm}{
				\begin{tabular}{@{}cccccccccccc@{}}
					\toprule
					\multirow{2}{*}{Methods} & \multirow{2}{*}{} & \multicolumn{5}{c}{Mesh refinement level}                        \\ \cmidrule(l){3-7}
					&    & $2^5$  & $2^6$  & $2^7$  & $2^8$  & $2^9$    \\ \midrule
                    HSS 
                    & IT & 54 & 42 & 56 & 172 & - \\ 
                    & CPU & 0.13 & 1.19 & 16.84 & 253.88 & - \\
                    & RES & 5.42e-09 & 7.47e-09 & 8.04e-09 & 4.78e-09 & - \\
                    DIAG 
                    & IT & 3 & 3 & 3 & - & - \\ 
                    & CPU & 0.12 & 3.91 & 114.67 & - & - \\
                    & RES & 4.94e-12 & 3.36e-11 & 5.47e-10 & - & - \\
                    TBD
                    & IT & 2 & 2 & 2 & - & - \\
                    & CPU & 0.12 & 3.27 & 109.75 & - & - \\
                    & RES & 8.57e-12 & 6.60e-11 & 1.06e-09 & - & - \\
                    TPSS
                    & IT & 44 & 44 & - & - & - \\
                    & CPU & 0.79 & 24.31 & - & - & - \\
                    & RES & 8.97e-09 & 7.74e-09 & - & - & - \\
                    DS
                    & IT & 28 & 30 & 58 & - & - \\
                    & CPU & 0.35 & 15.81 & 734.17 & - & - \\
                    & RES & 9.26e-09 & 5.71e-09 & 6.93e-09 & - & - \\
                    RDF
                    & IT & 24 & 17 & 13 & - & - \\
                    & CPU & 0.19 & 4.51 & 125.15 & - & - \\
                    & RES & 5.52e-09 & 3.25e-09 & 4.29e-09 & - & - \\
                    GSS
                    & IT & 4 & 4 & 4 & - & - \\
                    & CPU & 0.27 & 5.12 & 198.94 & - & - \\
                    & RES & 3.04e-10 & 2.78e-10 & 2.74e-10 & - & - \\
                    BD
                    & IT & 14 & 14 & 14 & - & - \\
                    & CPU & 0.15 & 3.69 & 131.21 & - & - \\
                    & RES & 2.53e-09 & 2.44e-09 & 2.38e-09 & - & - \\
                    BPP
                    & IT & 2 & 2 & 3 & - & - \\
                    & CPU & 0.17 & 3.19 & 113.71 & - & - \\
                    & RES & 4.26e-11 & 3.68e-11 & 8.10e-12 & - & - \\
                    IMD
                    & IT & 28 & 28 & 28 & 28 & 28 \\
                    & CPU & 0.07 & 0.45 & 5.72 & 62.71 & 687.89 \\
                    & RES & 7.26e-09 & 8.51e-09 & 8.37e-09 & 7.98e-09 & 7.50e-09 \\
                    RIMD
                    & IT & 39 & 39 & 39 & 39 & 39 \\
                    & CPU & 0.16 & 0.61 & 2.99 & 29.57 & 161.01 \\
                    & RES & 5.86e-09 & 6.93e-09 & 6.46e-09 & 5.87e-09 & 5.50e-09 \\        \bottomrule
					\multicolumn{7}{l}{\footnotesize \textsuperscript{*}Note: Here and throughout, the symbol ``-'' denotes that the corresponding method failed.}
			\end{tabular}}
		\end{table}

        \begin{figure}[H]
			\centering
			\subfigbottomskip=2pt
			\subfigure{
				\includegraphics[width=0.45 \linewidth]{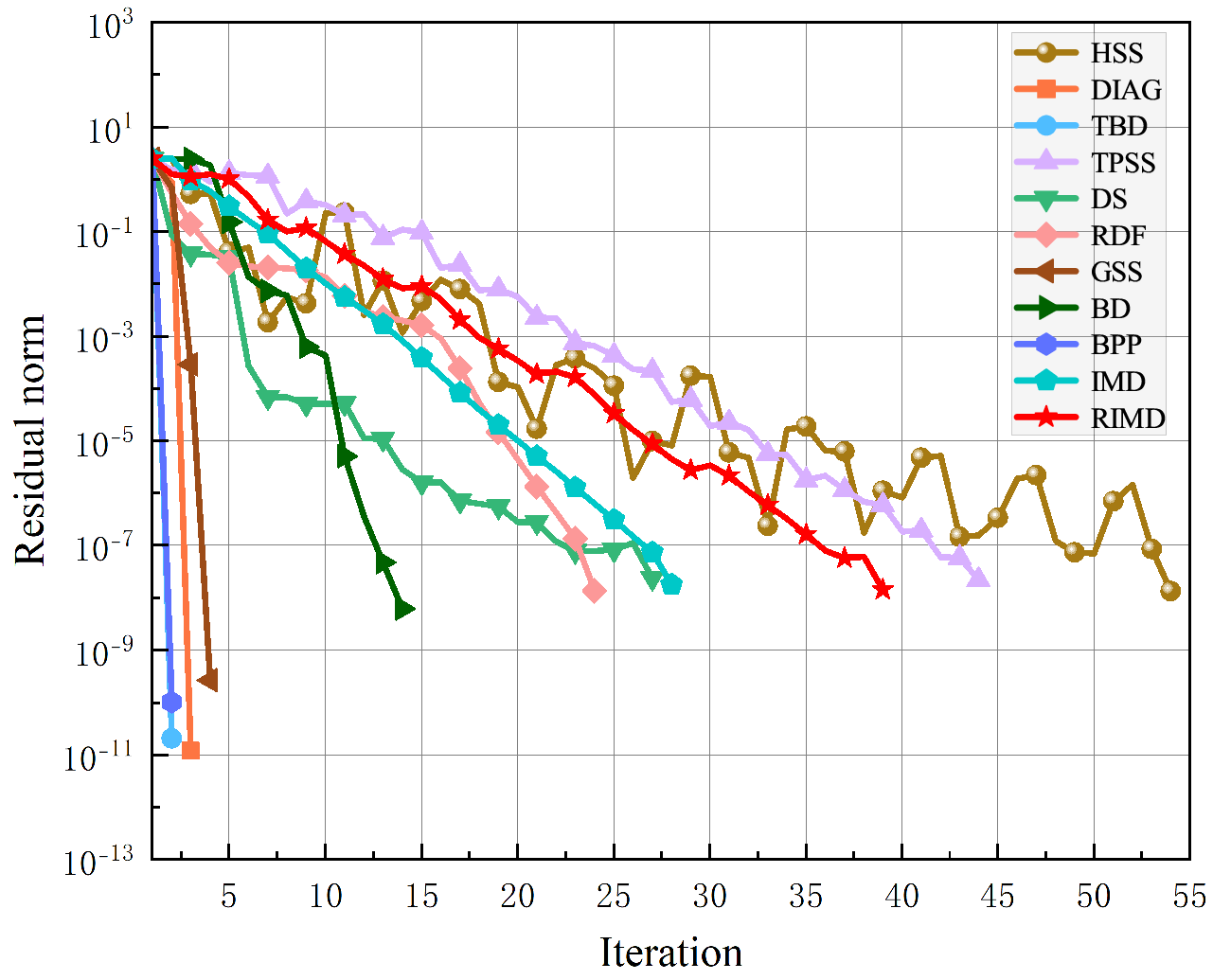}
			}
			\subfigure{
				\includegraphics[width=0.45 \linewidth]{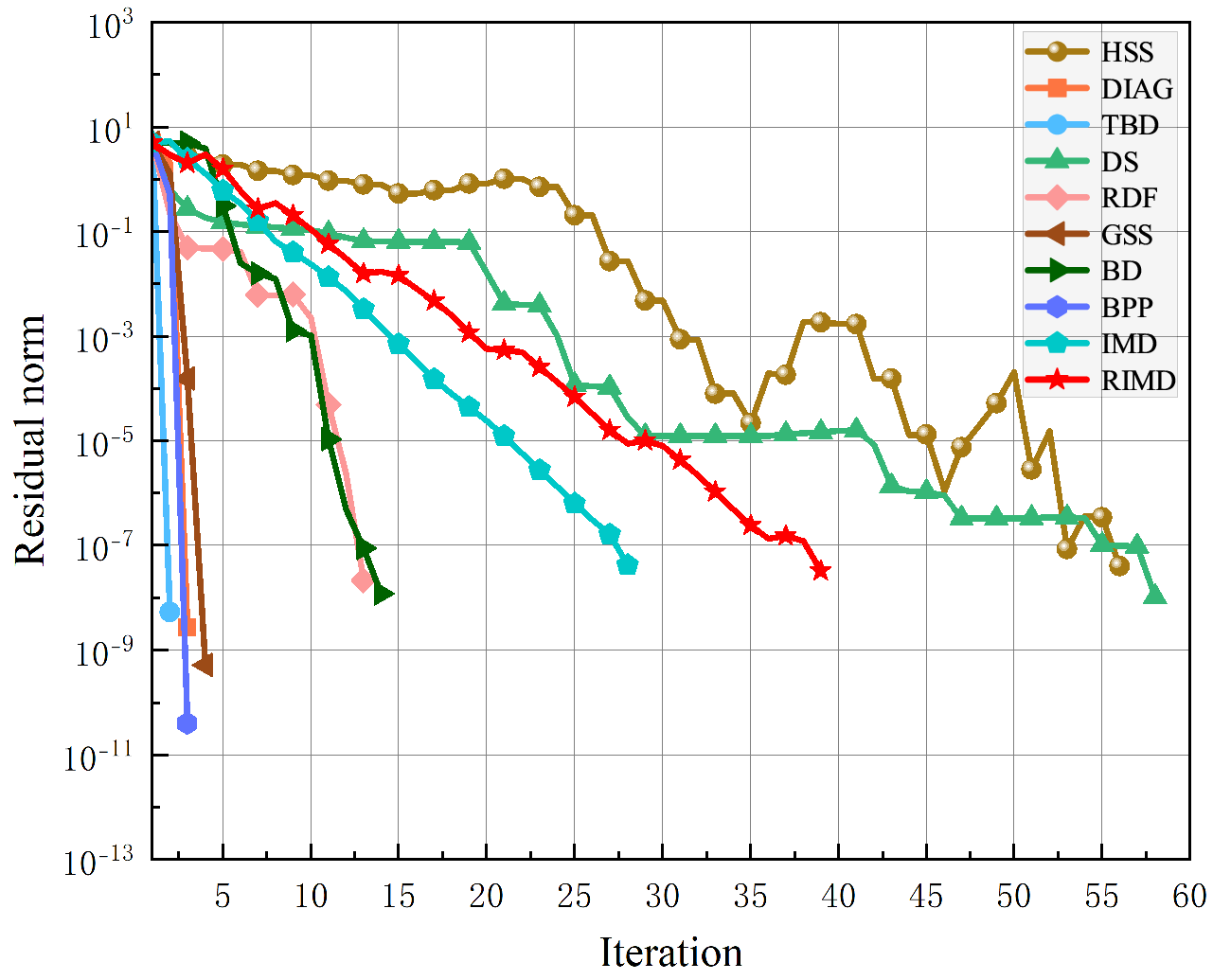}
			}
			\caption{Convergence curves of different methods for solving the double saddle-point system from \Cref{ex1} with $\beta=10^{-2}$ and $\ell$ being $2^5$ (left) and $2^7$ (right).}
			\label{fig:res_ex1}
		\end{figure}

        The parameter $\alpha$ in $\calP_{\rm HSS}$ is determined by numerical tests. With the regularization parameter fixed at $\beta=10^{-2}$ and the mesh refinement level $\ell=2^5$, ten logarithmically‑uniform sample points are taken over the interval $[10^{-7},10^{2}]$. The value yielding the minimal CPU runtime is selected, which gives $\alpha=10^{-2}$. This parameter is adopted for all tests in \Cref{ex1}. Since $S = BA^{-1}B^T = \frac{1}{2\beta}M$ can be easily obtained by discretization in \Cref{ex1}, we take $\widehat{S} = S$ and $\widehat{X} = E + CS^{-1}C^T$ in $\calP_{\rm BPP}$.  The parameter $\alpha$ in $\calP_{\rm DS}$ and $\calP_{\rm RDF}$ is set to $10^{-2}$. Numerical results for \Cref{ex1} are given in \Cref{table:ex1_result,table:ex1_result2,table:ex1_result3}. The CPU time reported in \Cref{table:ex1_result,table:ex1_result2,table:ex1_result3} includes the time required to construct and compute the corresponding preconditioner. The residual norms versus the number of iterations for different methods are shown in \Cref{fig:res_ex1}. We plot the eigenvalue distributions of the original coefficient matrix and the preconditioned matrices in \Cref{fig:eigenvalue_ex1}. As the eigenvalue distribution of $\calP^{-1}\calA$ is almost the same under the two tested strategies for selecting $Q$, we show only one of them in \Cref{fig:eigenvalue_ex1}. 

         \begin{table}[H]\small
			\caption{Numerical results for saddle-point system from \Cref{ex1} with $\beta = 10^{-3}$.}
			\centering
			\label{table:ex1_result2}
			\setlength{\tabcolsep}{2.5mm}{
				\begin{tabular}{@{}cccccccccccc@{}}
					\toprule
					\multirow{2}{*}{Methods} & \multirow{2}{*}{} & \multicolumn{5}{c}{Mesh refinement level}                        \\ \cmidrule(l){3-7}
					&    & $2^5$  & $2^6$  & $2^7$  & $2^8$  & $2^9$    \\ \midrule
                    HSS
                    & IT & 62 & 53 & 77 & 197 & - \\ 
                    & CPU & 0.19 & 1.68 & 23.34 & 307.20 & - \\
                    & RES & 3.89e-09 & 4.27e-09 & 7.15e-09 & 4.40e-09 & - \\
                    DIAG    & IT & 3 & 3 & 3 & - & - \\
                    & CPU & 0.11 & 4.60 & 115.86 & - & - \\
                    & RES & 2.89e-12 & 2.13e-11 & 3.33e-10 & - & - \\
                    TBD
                    & IT & 2 & 2 & 2 & - & - \\
                    & CPU & 0.13 & 3.17 & 111.87 & - & - \\
                    & RES & 8.63e-12 & 6.09e-11 & 7.46e-10 & - & - \\
                    TPSS
                    & IT & 46 & 46 & - & - & - \\
                    & CPU & 0.81 & 28.50 & - & - & - \\
                    & RES & 8.85e-09 & 6.33e-09 & - & - & - \\
                    DS
                    & IT & 30 & 36 & 69 & - & - \\
                    & CPU & 0.42 & 19.16 & 771.37 & - & - \\
                    & RES & 7.46e-09 & 2.07e-09 & 2.10e-09 & - & - \\
                    RDF
                    & IT & 29 & 28 & 22 & - & - \\
                    & CPU & 0.20 & 3.66 & 141.33 & - & - \\
                    & RES & 8.42e-09 & 6.82e-09 & 3.29e-09 & - & - \\
                    GSS
                    & IT & 4 & 4 & 4 & - & - \\
                    & CPU & 0.37 & 8.33 & 221.78 & - & - \\
                    & RES & 1.07e-09 & 9.63e-10 & 9.19e-10 & - & - \\
                    BD
                    & IT & 21 & 21 & 21 & - & - \\
                    & CPU & 0.16 & 5.36 & 133.75 & - & - \\
                    & RES & 3.88e-09 & 4.07e-09 & 8.58e-09 & - & - \\
                    BPP
                    & IT & 2 & 2 & 2 & - & - \\
                    & CPU & 0.11 & 4.61 & 113.94 & - & - \\
                    & RES & 4.14e-12 & 2.22e-10 & 8.29e-10 & - & - \\
                    IMD
                    & IT & 30 & 30 & 30 & 30 & 30 \\
                    & CPU & 0.06 & 0.55 & 6.20 & 67.09 & 710.18 \\
                    & RES & 7.19e-09 & 8.99e-09 & 9.43e-09 & 9.52e-09 & 9.50e-09 \\
                    RIMD
                    & IT & 43 & 45 & 43 & 43 & 43 \\
                    & CPU & 0.18 & 0.73 & 6.76 & 34.12 & 181.22 \\
                    & RES & 8.07e-09 & 7.14e-09 & 7.14e-09 & 7.09e-09 & 6.88e-09 \\\hline 
			\end{tabular}}
		\end{table}

        \begin{table}[htb]\small
			\caption{Numerical results for saddle-point system from \Cref{ex1} with $\beta = 10^{-5}$.}
			\centering
			\label{table:ex1_result3}
			\setlength{\tabcolsep}{2.5mm}{
				\begin{tabular}{@{}cccccccccccc@{}}
					\toprule
					\multirow{2}{*}{Methods} & \multirow{2}{*}{} & \multicolumn{5}{c}{Mesh refinement level}                        \\ \cmidrule(l){3-7}
					&    & $2^5$  & $2^6$  & $2^7$  & $2^8$  & $2^9$    \\ \midrule
                    HSS
                    & IT & 122 & 96 & 132 & 276 & - \\ 
                    & CPU & 0.80 & 3.97 & 39.43 & 428.34 & - \\
                    & RES & 5.26e-09 & 1.90e-09 & 9.21e-09 & 7.30e-09 & - \\
                    DIAG
                    & IT & 3 & 3 & 3 & - & -\\
                    & CPU & 0.11 & 4.19 & 113.70 & - & -\\
                    & RES & 6.94e-14 & 1.66e-13 & 7.79e-12 & - & -\\
                    TBD
                    & IT & 2 & 2 & 2 & - & -\\
                    & CPU & 0.12 & 3.95 & 113.42 & - & -\\
                    & RES & 1.02e-13 & 3.73e-13 & 1.68e-11 & - & -\\
                    TPSS
                    & IT & 48 & 51 & - & - & -\\
                    & CPU & 1.15 & 30.94 & - & - & -\\
                    & RES & 6.49e-09 & 6.65e-09 & - & - & -\\
                    DS
                    & IT & 64 & 61 & 95 & - & -\\
                    & CPU & 0.58 & 24.74 & 725.28 & - & -\\
                    & RES & 4.70e-09 & 6.42e-09 & 6.73e-09 & - & -\\
                    RDF
                    & IT & 27 & 27 & 28 & - & -\\
                    & CPU & 0.19 & 6.21 & 142.52 & - & -\\
                    & RES & 7.58e-09 & 9.72e-09 & 9.75e-09 & - & -\\
                    GSS
                    & IT & 5 & 5 & 4 & - & -\\
                    & CPU & 0.32 & 8.63 & 224.34 & - & -\\
                    & RES & 1.64e-09 & 4.41e-10 & 2.15e-09 & - & -\\
                    BD
                    & IT & 50 & 49 & 47 & - & -\\
                    & CPU & 0.30 & 6.97 & 164.43 & - & -\\
                    & RES & 9.20e-09 & 4.51e-09 & 5.04e-09 & - & -\\
                    BPP
                    & IT & 2 & 2 & 2 & - & -\\
                    & CPU & 0.12 & 3.08 & 115.96 & - & -\\
                    & RES & 4.41e-13 & 2.28e-11 & 3.05e-11 & - & -\\
                    IMD
                    & IT & 48 & 48 & 51 & 51 & 51\\
                    & CPU & 0.08 & 0.71 & 9.61 & 107.93 & 1212.34\\
                    & RES & 3.45e-09 & 4.85e-09 & 3.80e-09 & 4.10e-09 & 3.77e-09\\
                    RIMD
                    & IT & 62 & 63 & 62 & 62 & 62\\
                    & CPU & 0.28 & 0.98 & 9.02 & 53.21 & 271.09\\
                    & RES & 4.49e-09 & 7.76e-09 & 6.43e-09 & 7.28e-09 & 8.24e-09\\   \hline
			\end{tabular}}
		\end{table}

        From the numerical results in \Cref{table:ex1_result,table:ex1_result2,table:ex1_result3}, we can observe that IMD and RIMD exhibit a significant advantage in terms of CPU time, especially when the problem size is large. The superiority becomes even more pronounced as the problem scale increases, while the other methods fail to solve large-scale problems (when the mesh refinement level is greater than $2^7$ or $2^8$). Based on the results obtained from HSS, DIAG, TBD, and TPSS, it is evident that the CPU time of preconditioners developed for standard saddle-point system grows rapidly as the problem size expands. TPSS fails when the mesh refinement level is greater than $2^6$, and all three methods DIAG, TBD, and TPSS fail when the mesh refinement level is greater than $2^7$. The number of iterations required by HSS varies significantly with $\ell$, showing poor $\ell$-robustness. These results illustrate the necessity of investigating preconditioning techniques by taking advantage of the special $3 \times 3$ block structure of the double saddle-point system \eqref{eq:double saddle} for improved computational efficiency. A comparison between IMD and RIMD shows that RIMD is superior to IMD in CPU time when solving large-scale problems, which verifies that our selection strategy for $Q$ in \Cref{sec:select} is effective and that the randomized low-rank approximation approach is feasible for constructing inexact preconditioners. As $\beta$ decreases, the number of iterations required by almost all tested methods increases within a certain range. With a fixed value of $\beta$, the iteration number of the RIMD is nearly unchanged with the increase of mesh refinement level $\ell$, which demonstrates the $\ell$-robustness of the proposed method. Moreover, it can be seen that the inexact preconditioner yields better performance when we compare the numerical results of IMD (or RIMD) and BPP with those of other methods. \Cref{fig:res_ex1} shows that all the tested methods converge within a finite number of iterations. We can observe from \Cref{fig:eigenvalue_ex1} that all the preconditioners tested improve the spectral distribution of the original coefficient matrix, leading to a clustered eigenvalue distribution of the preconditioned matrices.

        \begin{figure}[H]
			\centering
			\subfigbottomskip=1pt
			\subfigure[$\calA$]{
				\includegraphics[width=0.31 \linewidth]{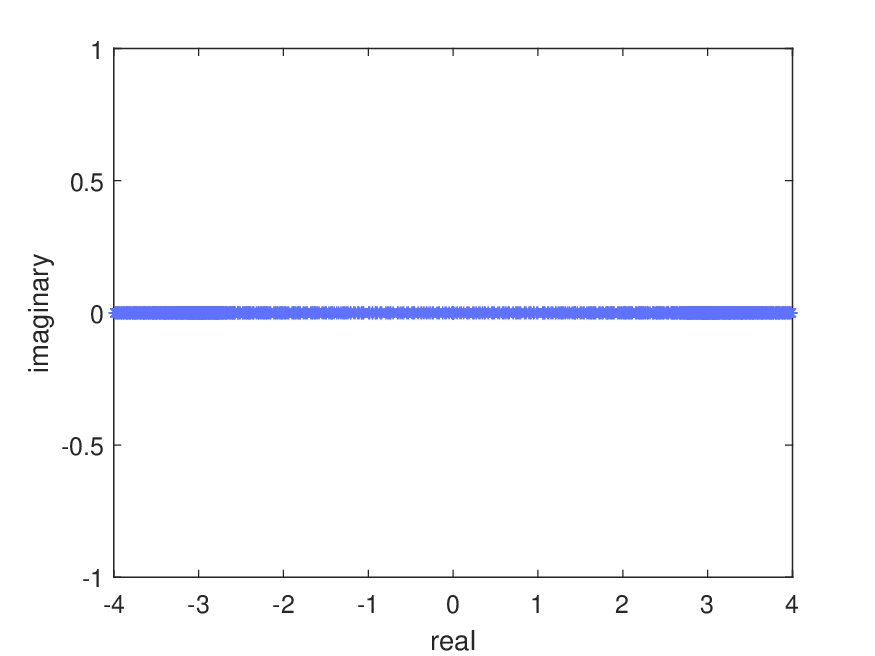}
				}
                \subfigure[$\calP_{\rm HSS}^{-1}\cal{A}$]{
				\includegraphics[width=0.31 \linewidth]{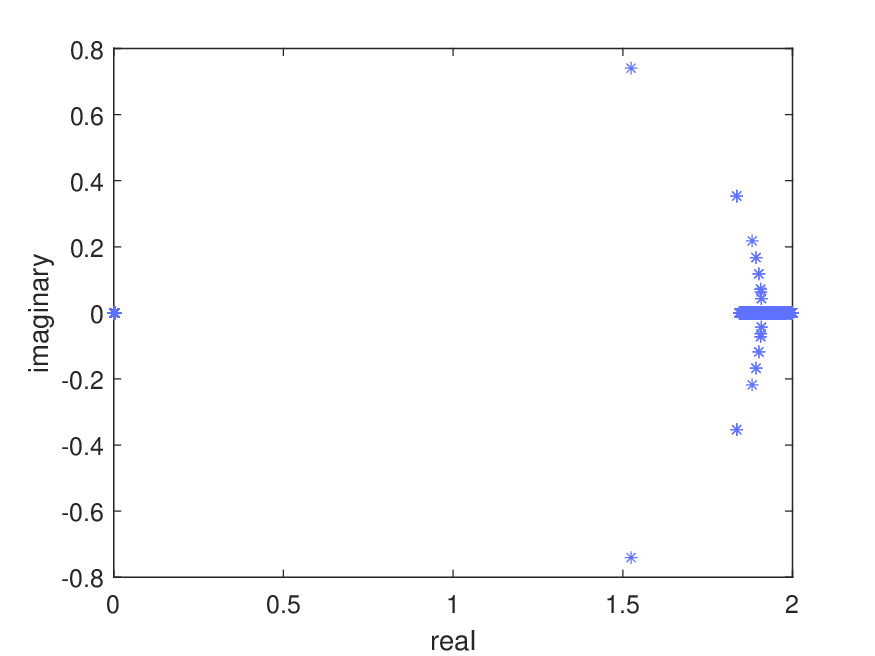}
				}
                
                \subfigure[$\calP_{\rm DIAG}^{-1}\cal{A}$]{
				\includegraphics[width=0.31 \linewidth]{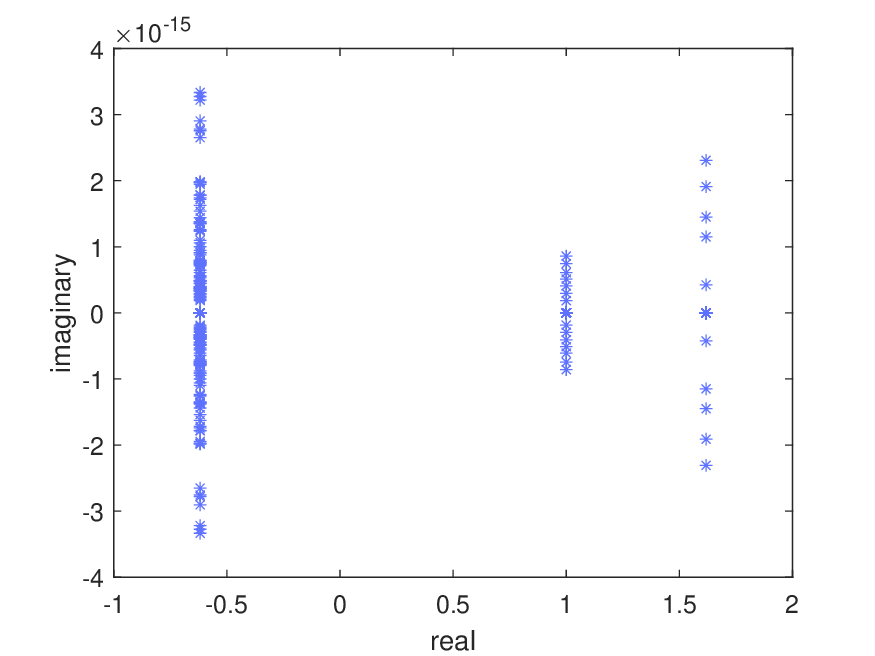}
				}
                \subfigure[$\calP_{\rm TBD}^{-1}\cal{A}$]{
				\includegraphics[width=0.31 \linewidth]{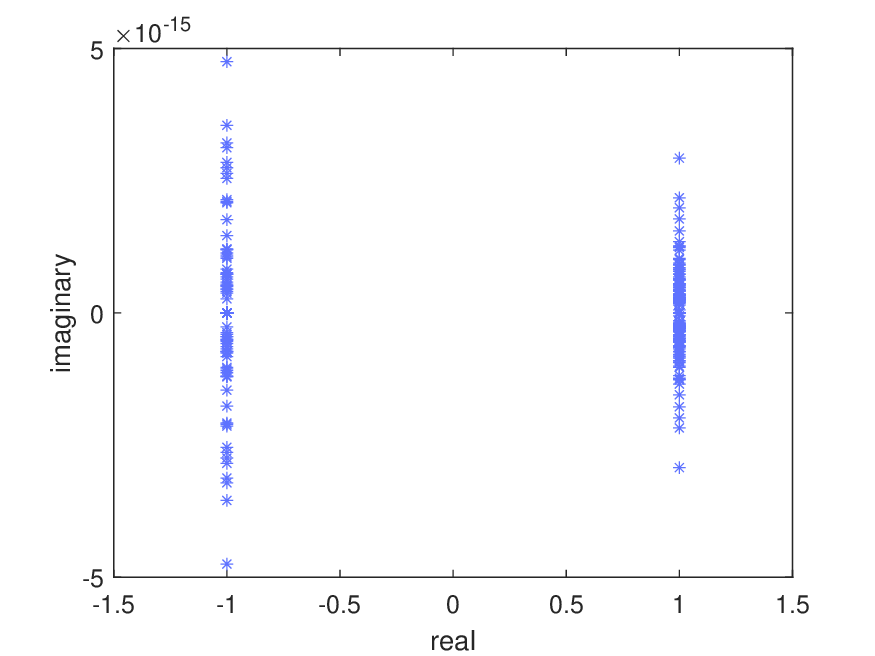}
				}
                \subfigure[$\calP_{\rm TPSS}^{-1}\cal{A}$]{
				\includegraphics[width=0.31 \linewidth]{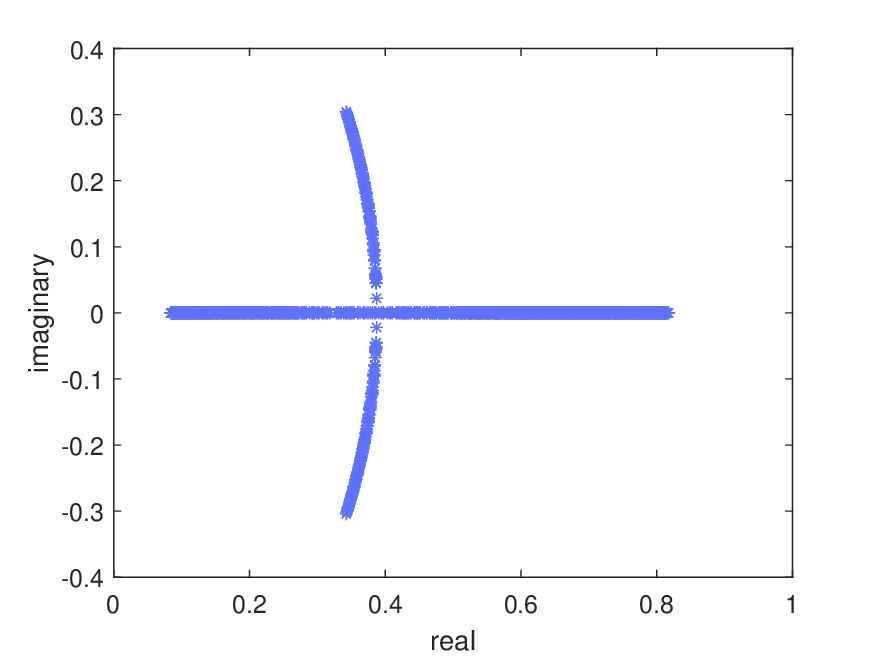}
				}
                
                \subfigure[$\calP_{\rm DS}^{-1}\mathcal{B}$]{
				\includegraphics[width=0.31 \linewidth]{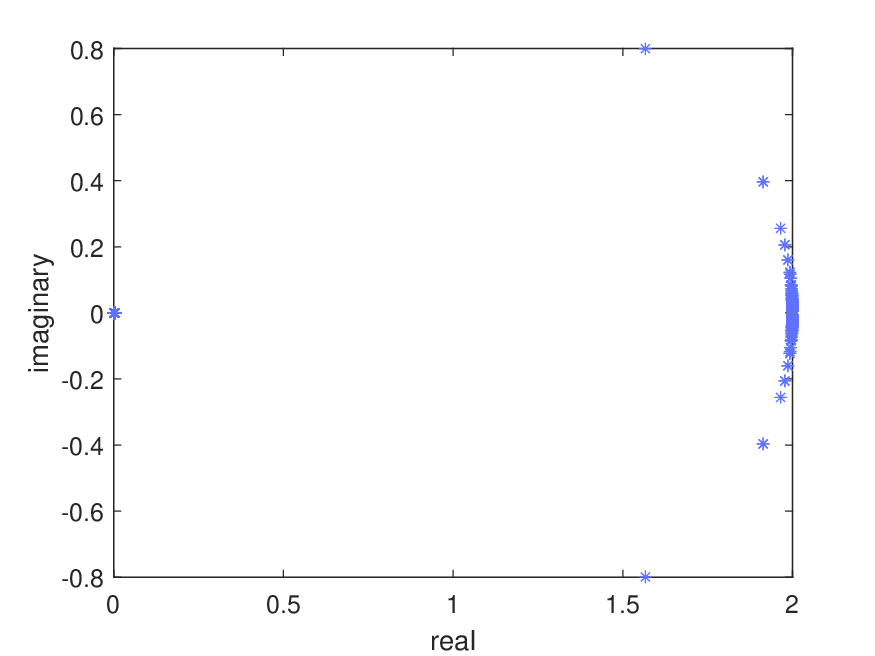}
				}
			\subfigure[$\calP_{\rm RDF}^{-1}\mathcal{B}$]{
				\includegraphics[width=0.31 \linewidth]{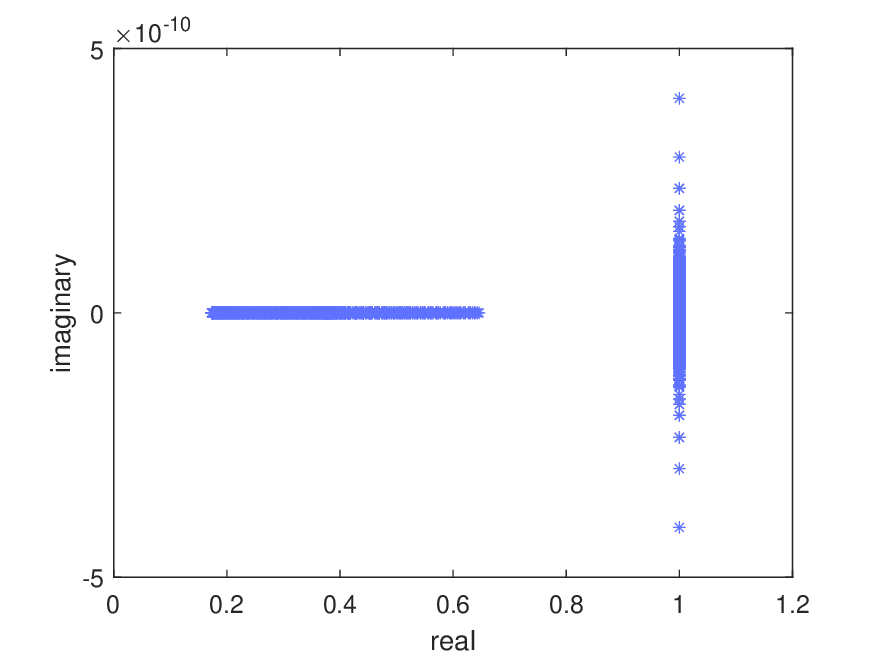}
				}                
			\subfigure[$\calP_{\rm GSS}^{-1}\mathcal{B}$]{
				\includegraphics[width=0.31 \linewidth]{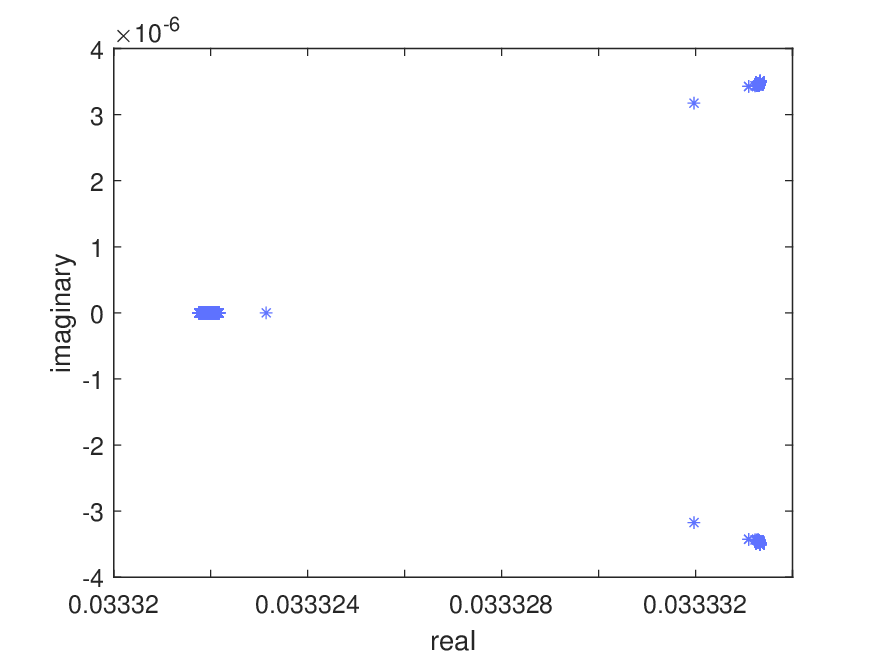}
				}
                
			\subfigure[$\calP_{\rm BD}^{-1}\mathcal{C}$]{
				\includegraphics[width=0.31 \linewidth]{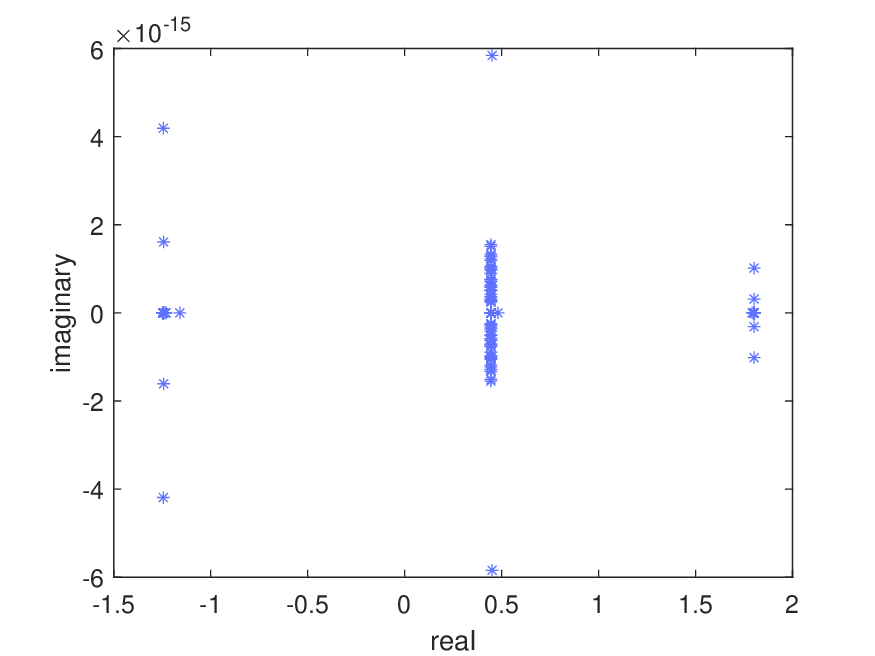}
			}
			\subfigure[$\calP_{\rm BPP}^{-1}\mathcal{C}$]{
				\includegraphics[width=0.31 \linewidth]{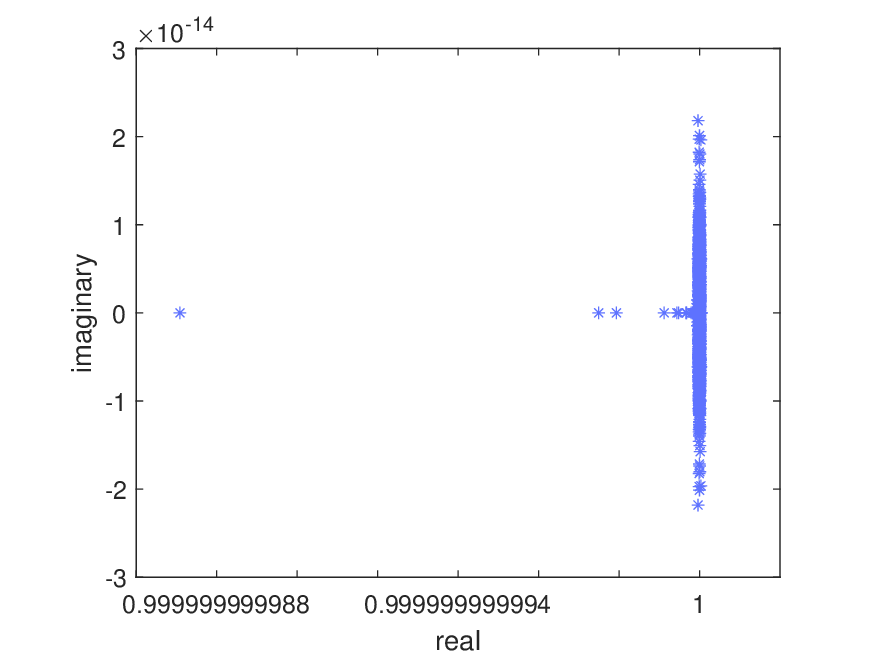}
			}
            \subfigure[$\calP^{-1}\calA$]{
				\includegraphics[width=0.31 \linewidth]{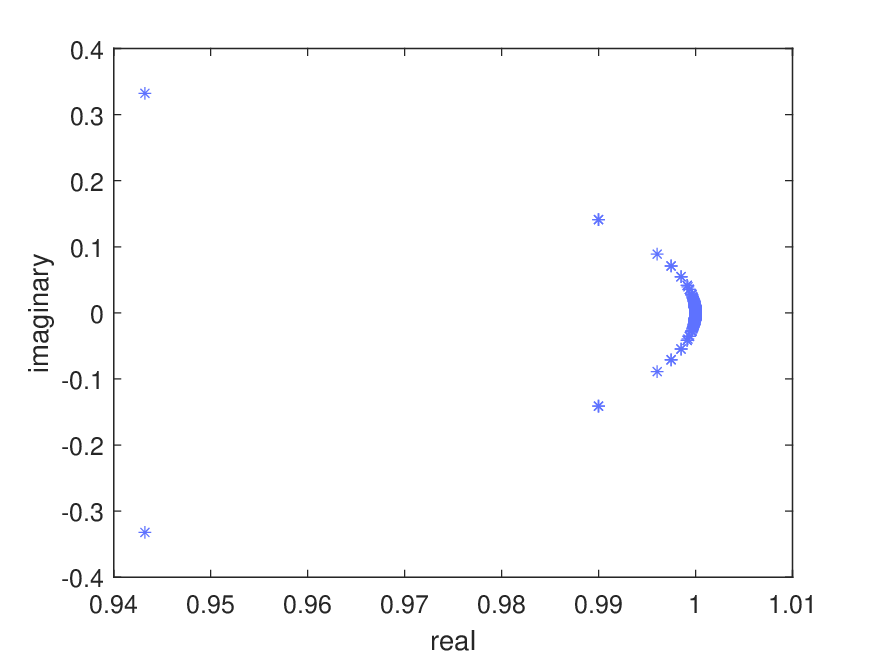}
			}
			\caption{Eigenvalue distributions of original coefficient matrix and the preconditioned matrices for \Cref{ex1} with $\beta=10^{-2}$ and $\ell=2^5$.}
			\label{fig:eigenvalue_ex1}
		\end{figure}

    \begin{table}[H]\small
			\caption{Numerical results for saddle-point system from \Cref{ex2} with $\gamma=10^{-3}$.}
			\centering
			\label{table:ex2_result}
			\setlength{\tabcolsep}{2.5mm}{
				\begin{tabular}{@{}ccccccccccccccc@{}}
					\toprule
					\multirow{2}{*}{Methods} & \multirow{2}{*}{} & \multicolumn{5}{c}{$h$}                        \\ \cmidrule(l){3-7}
					&    & $2^{-3}$  & $2^{-4}$  & $2^{-5}$  & $2^{-6}$  & $2^{-7}$   \\ \midrule
                    HSS 
                    & IT & 99 & 86 & 127 & - & - \\ 
                    & CPU & 0.05 & 0.16 & 2.55 & - & - \\
                    & RES & 4.80e-09 & 5.38e-09 & 9.99e-09 & - & - \\
					DIAG
                    & IT & 2 & 2 & 2 & 2 & 2\\
                    & CPU & 0.01 & 0.04 & 0.20 & 5.89 & 153.62\\
                    & RES & 7.07e-14 & 8.62e-13 & 2.28e-11 & 5.07e-10 & 9.18e-09\\
                    TBD
                    & IT & 1 & 1 & 1 & 1 & 1\\
                    & CPU & 0.01 & 0.02 & 0.20 & 5.56 & 156.17\\
                    & RES & 5.55e-14 & 8.47e-13 & 2.29e-11 & 5.08e-10 & 9.19e-09\\
                    TPSS
                    & IT & 32 & 32 & 30 & 33 & -\\
                    & CPU & 0.02 & 0.11 & 1.63 & 56.10 & -\\
                    & RES & 4.99e-09 & 2.64e-09 & 5.82e-09 & 4.52e-09 & -\\
                    DS
                    & IT & 22 & 24 & 27 & - & -\\
                    & CPU & 0.03 & 0.08 & 0.52 & - & -\\
                    & RES & 4.67e-09 & 8.64e-09 & 9.58e-09 & - & -\\
                    RDF
                    & IT & 22 & 24 & 27 & - & -\\
                    & CPU & 0.03 & 0.05 & 0.33 & - & -\\
                    & RES & 3.11e-09 & 7.90e-09 & 4.57e-09 & - & -\\
                    GSS
                    & IT & 5 & 5 & 5 & - & -\\
                    & CPU & 0.03 & 0.05 & 0.35 & - & -\\
                    & RES & 6.89e-10 & 1.84e-09 & 7.57e-09 & - & -\\
                    BD
                    & IT & 25 & 25 & 25 & 25 & -\\
                    & CPU & 0.03 & 0.06 & 0.33 & 11.67 & -\\
                    & RES & 2.01e-09 & 3.99e-09 & 4.84e-09 & 5.14e-09 & -\\
                    BPP
                    & IT & 15 & 16 & 16 & 18 & 26\\
                    & CPU & 0.02 & 0.03 & 0.16 & 0.76 & 3.89\\
                    & RES & 3.06e-09 & 2.40e-09 & 3.29e-09 & 4.89e-09 & 6.25e-09\\
                    IMD
                    & IT & 25 & 26 & 29 & 29 & 34\\
                    & CPU & 0.01 & 0.03 & 0.09 & 0.38 & 2.76\\
                    & RES & 3.46e-09 & 9.75e-09 & 4.53e-09 & 9.75e-09 & 9.18e-09\\
                    RIMD
                    & IT & 27 & 28 & 29 & 29 & 34\\
                    & CPU & 0.02 & 0.03 & 0.09 & 0.38 & 2.23\\
                    & RES & 3.77e-09 & 3.39e-09 & 2.43e-09 & 8.91e-09 & 9.05e-09\\
					\hline
			\end{tabular}}
		\end{table}

	\begin{ex}\label{ex2}
		\rm{\textbf{Full Observation Problem \cite{bergamaschi2025spectral}.}} We focus on the PDE-constrained optimization problem of the form:
        \begin{equation}\label{eq:ex2_obj}
            \begin{aligned}
			\min_{y,f} \quad &\tfrac{1}{2} \|y - \widehat{y}\|_{L_2(\varGamma)}^2 + \tfrac{\gamma}{2} \|f\|_{L_2(\varGamma)}^2 \\
			\text{s.t.} \quad &-\Delta y + y + f = 0 \quad \text{in } \varGamma, \\
			&\qquad\qquad~~\,\frac{\partial y}{\partial \mathbf{n}} = g \quad \text{on } \partial\varGamma, 
		\end{aligned}
        \end{equation}
		where $\widehat{y}$ is the provided desired state, $f$ is the control, $\varGamma = (0,1)^2$ is the domain with boundary $\partial\varGamma$, $\mathbf{n}$ is the unit outward normal vector of $\partial\varGamma$, and $\gamma>0$ is a regularization parameter. We set $\gamma$ is $10^{-3}$, $10^{-4}$, and $10^{-5}$. Discretizing \eqref{eq:ex2_obj} using P1 finite elements yields the following linear system in the form of a double saddle-point system:
		\begin{equation}\label{eq:ex2_dis}
			\left( \begin{matrix}
				\gamma M&		0&		M\\
				0&		M&		L\\
				M&		L&	0\\
			\end{matrix} \right) \left( \begin{array}{c}
				f_d\\
				y_d\\
				p_d \\
			\end{array} \right) =\left( \begin{array}{c}
				0\\
				\widehat{y}_d\\
				0\\
			\end{array} \right),	
		\end{equation}    
		where $M$ corresponds to the mass matrix, $L$ represents the sum of a stiffness matrix and a mass matrix. The unknown vectors $f_d$, $y_d$, and $p_d$ denote the discretized control, state, and adjoint variable. The constant vector $\widehat{y}_d$ is derived from $\widehat{y}$ and we set it as a Gaussian function, that is, $\widehat{y}_d = {\rm exp[-50((x_1-\frac{1}{2})^2+(x_2-\frac{1}{2})^2)]}$ with the spatial coordinates $x_1$ and $x_2$. In \Cref{ex2}, we test the mesh size $h$ and the corresponding problem size $n+p+m$ as follows:
        \begin{center}
	    \begin{tabular}{ c r r r r}
		\toprule
		$h$ & $n$ & $p$ & $m$ & $n+p+m$\\
		\midrule
		$2^{-3}$ & 81 & 81 & 81 & 243\\
		$2^{-4}$ & 289 & 289 & 289 & 867 \\
		$2^{-5}$ & 1,089 & 1,089 & 1,089 & 3,267 \\
		$2^{-6}$ & 4,255 & 4,225 & 4,225 & 12,675\\
		$2^{-7}$ & 16,641 & 16,641 &  16,641 & 49,923\\
		\bottomrule
	\end{tabular}
\end{center}        
        \end{ex}

        \begin{figure}[H]
			\centering
			\subfigbottomskip=2pt
			\subfigure{
				\includegraphics[width=0.45 \linewidth]{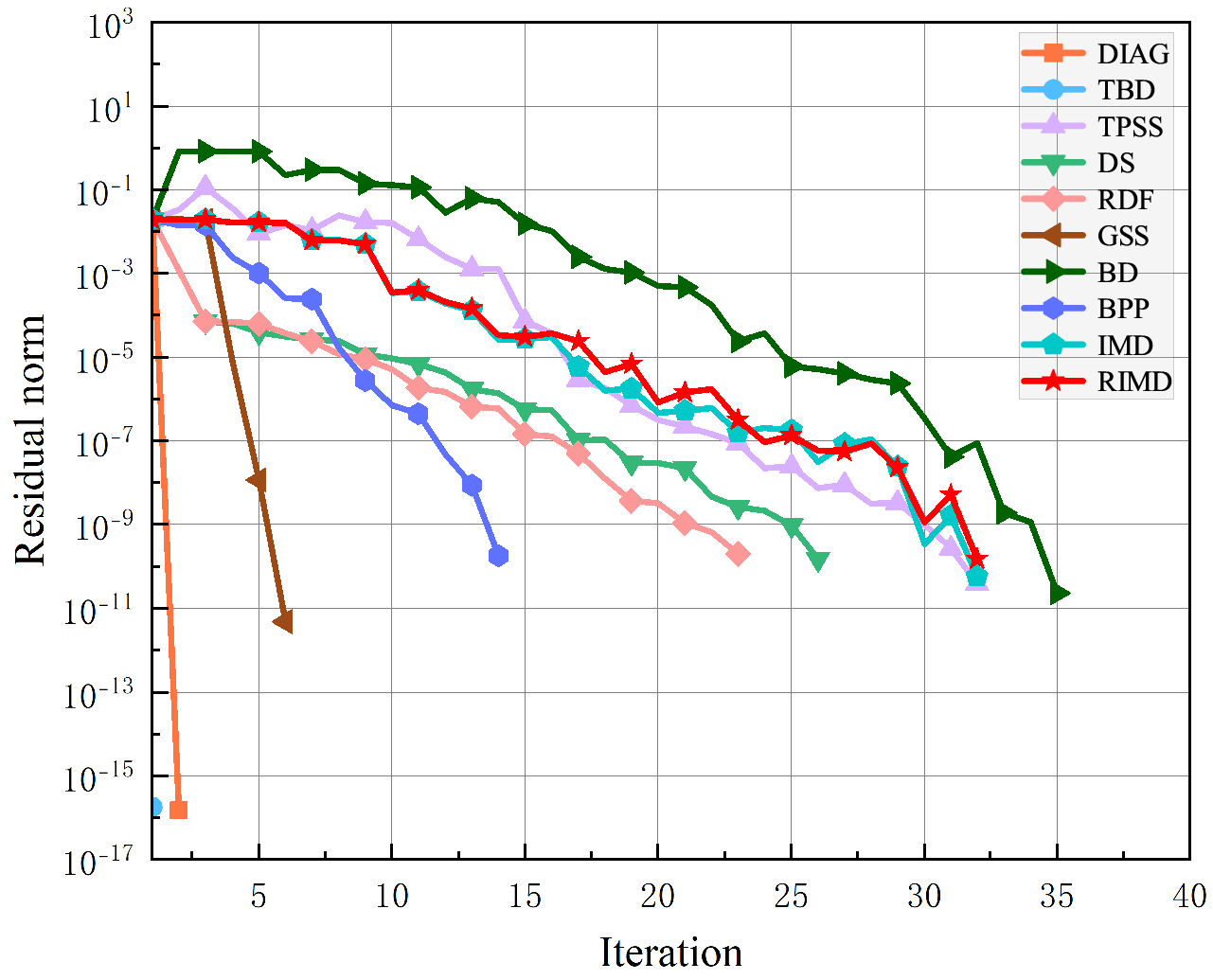}
			}
			\subfigure{
				\includegraphics[width=0.45 \linewidth]{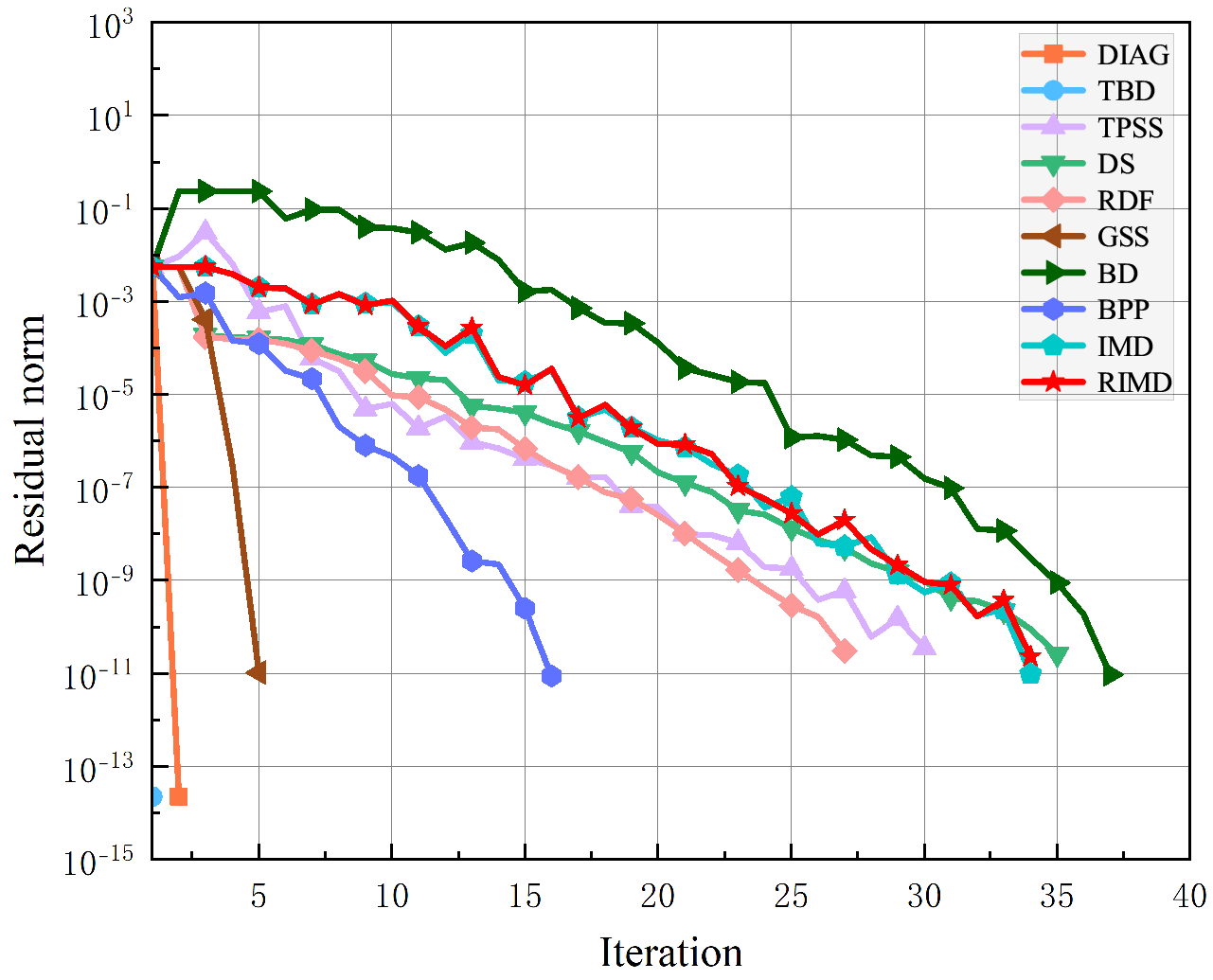}
			}
			\caption{Convergence curves of different methods for solving the double saddle-point system from \Cref{ex2} with $\gamma=10^{-4}$ and $h$ being $2^{-3}$ (left) and $2^{-5}$ (right).}
			\label{fig:res_ex2}
		\end{figure}

        In \Cref{ex2}, we set $\alpha=10^{-5}$ in $\calP_{\rm HSS}$, following a similar selection strategy as \Cref{ex1}. $\widehat{S}$ and $\widehat{X}$ are taken as $\widehat{S} = S_LS_L^T$ and $\widehat{X} = X_LX_L^T$ in $\calP_{\rm BPP}$, where $S_L$ and $X_L$ are the incomplete Cholesky factorization of $S$ and $E + C \widehat{S}^{-1}C^T$, computed by MATLAB's function ``ichol($\cdot$, opts)'' with opts.type=`ict', opts.michol=`on', and opts.droptol=1e-2 (for $S_L$) or opts.droptol=1e-6 (for $X_L$), respectively.  The parameter $\alpha$ in $\calP_{\rm DS}$ and $\calP_{\rm RDF}$ is set to $10^{-8}$. We report the numerical results for \Cref{ex2} in \Cref{table:ex2_result,table:ex2_result2,table:ex2_result3} and plot the convergence curves of the tested methods for $h = 2^{-3}$ and $h = 2^{-5}$ in \Cref{fig:res_ex2}. Since the HSS method requires a large number of iterations to satisfy the stopping criterion, we omit its convergence curve in \Cref{fig:res_ex2} for better visual clarity. \Cref{fig:eigenvalue_ex2} depicts the eigenvalue distributions of the original matrix and preconditioned matrices. As was done in \Cref{ex1}, we only show the eigenvalue distribution of $\calP^{-1}\calA$ corresponding to $Q = S + C^T{\rm diag}(E)^{-1}C$ in \Cref{fig:eigenvalue_ex2}.

        \begin{figure}[H]
			\centering
			\subfigbottomskip=2pt
			\subfigure[$\calA$]{
				\includegraphics[width=0.31 \linewidth]{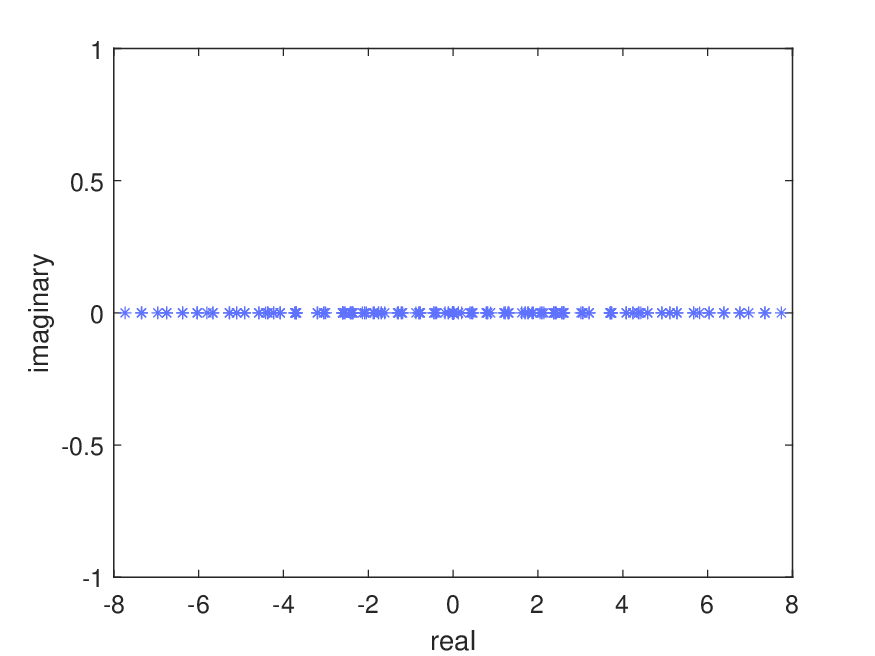}
				}
                \subfigure[$\calP_{\rm HSS}^{-1}\cal{A}$]{
				\includegraphics[width=0.31 \linewidth]{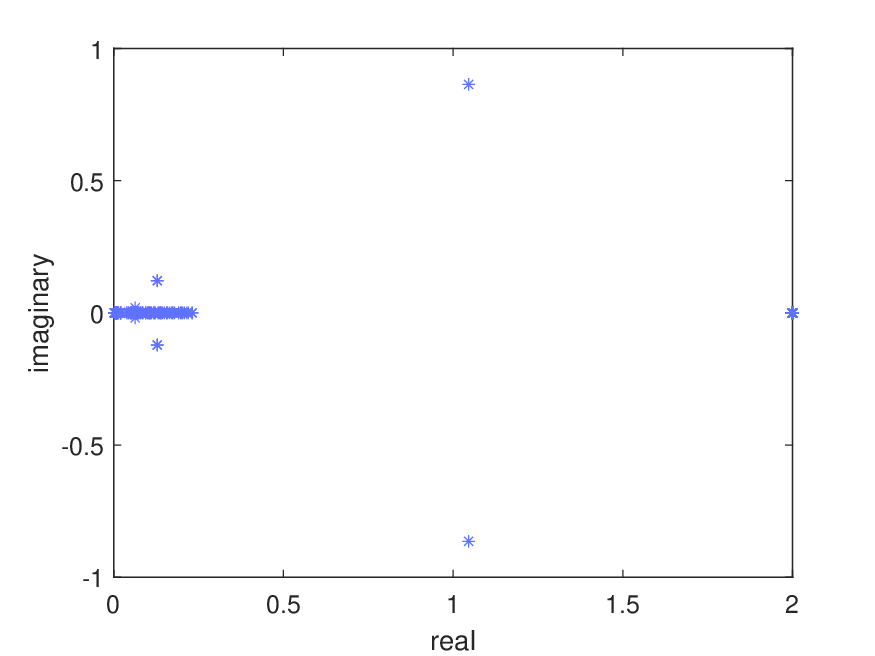}
				}
                
                \subfigure[$\calP_{\rm DIAG}^{-1}\cal{A}$]{
				\includegraphics[width=0.31 \linewidth]{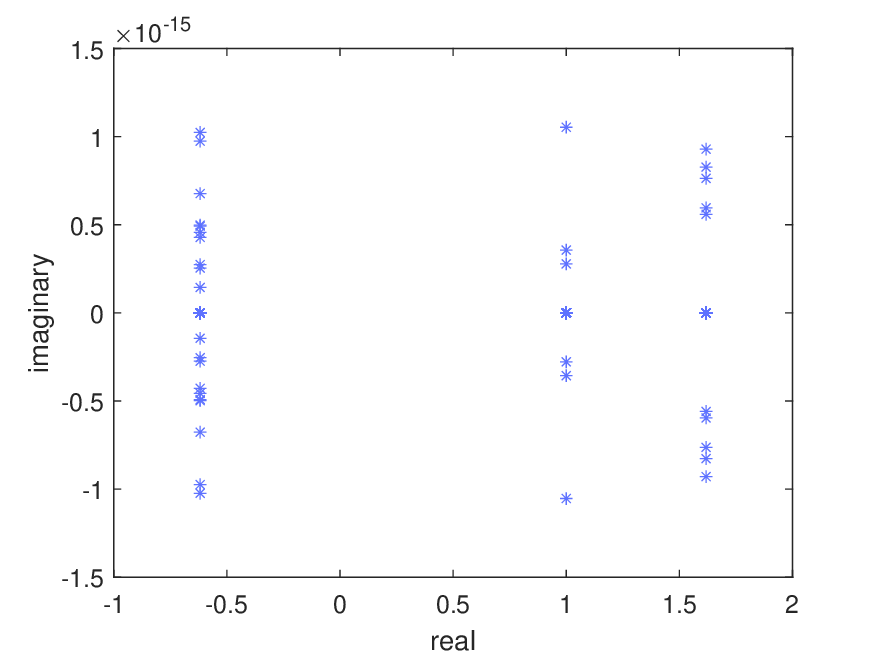}
				}
                \subfigure[$\calP_{\rm TBD}^{-1}\cal{A}$]{
				\includegraphics[width=0.31 \linewidth]{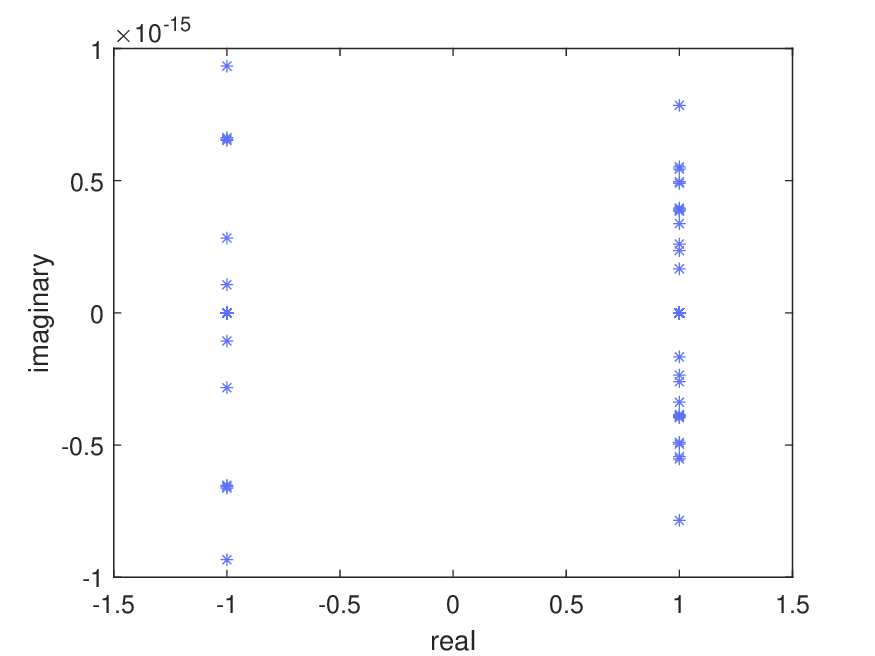}
				}
                \subfigure[$\calP_{\rm TPSS}^{-1}\cal{A}$]{
				\includegraphics[width=0.31 \linewidth]{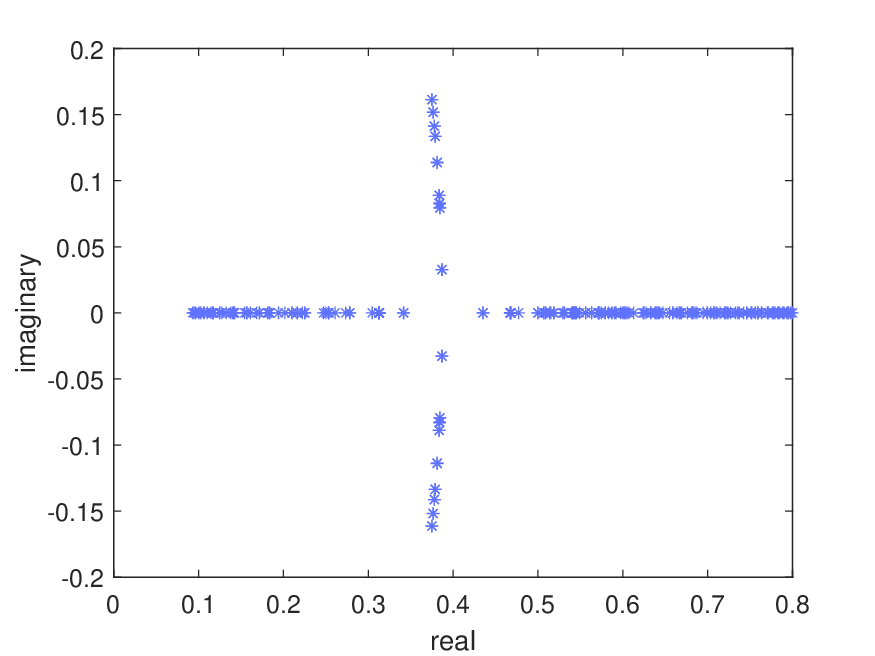}
				}
                
                \subfigure[$\calP_{\rm DS}^{-1}\mathcal{B}$]{
				\includegraphics[width=0.31 \linewidth]{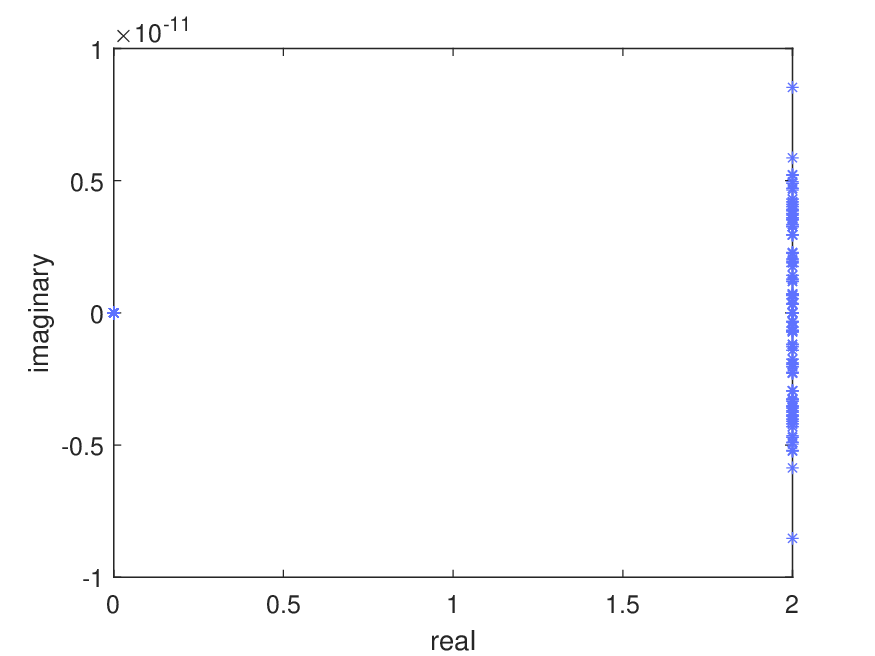}
				}
			\subfigure[$\calP_{\rm RDF}^{-1}\mathcal{B}$]{
				\includegraphics[width=0.31 \linewidth]{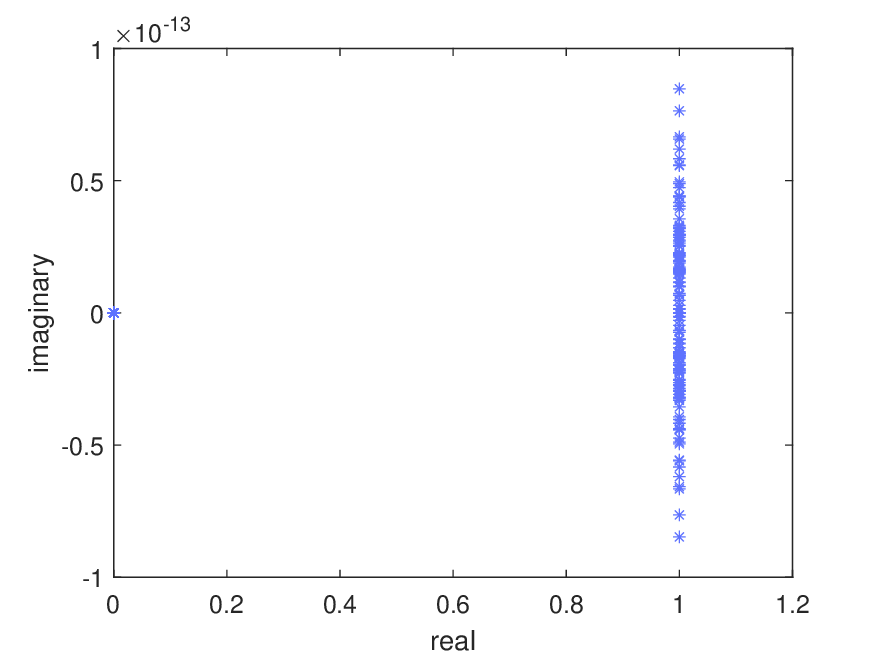}
				}
			\subfigure[$\calP_{\rm GSS}^{-1}\mathcal{B}$]{
				\includegraphics[width=0.31 \linewidth]{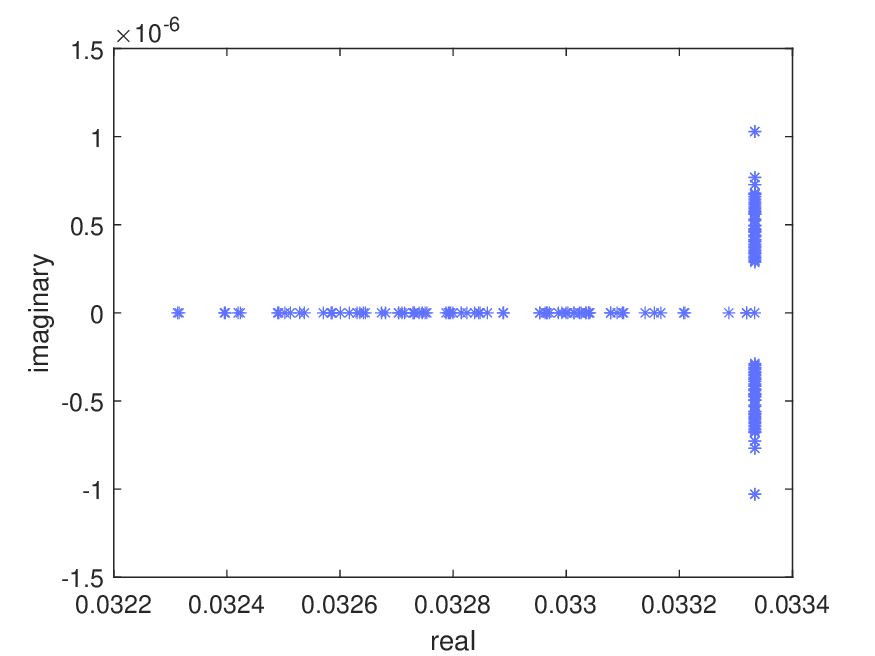}
				}
                
			\subfigure[$\calP_{\rm BD}^{-1}\mathcal{C}$]{
				\includegraphics[width=0.31 \linewidth]{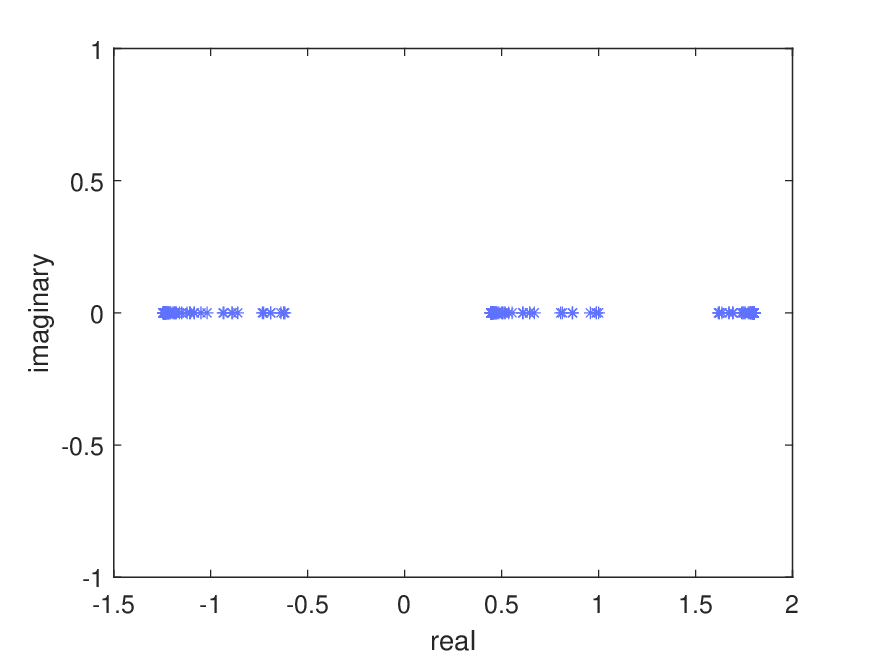}
			}
			\subfigure[$\calP_{\rm BPP}^{-1}\mathcal{C}$]{
				\includegraphics[width=0.31 \linewidth]{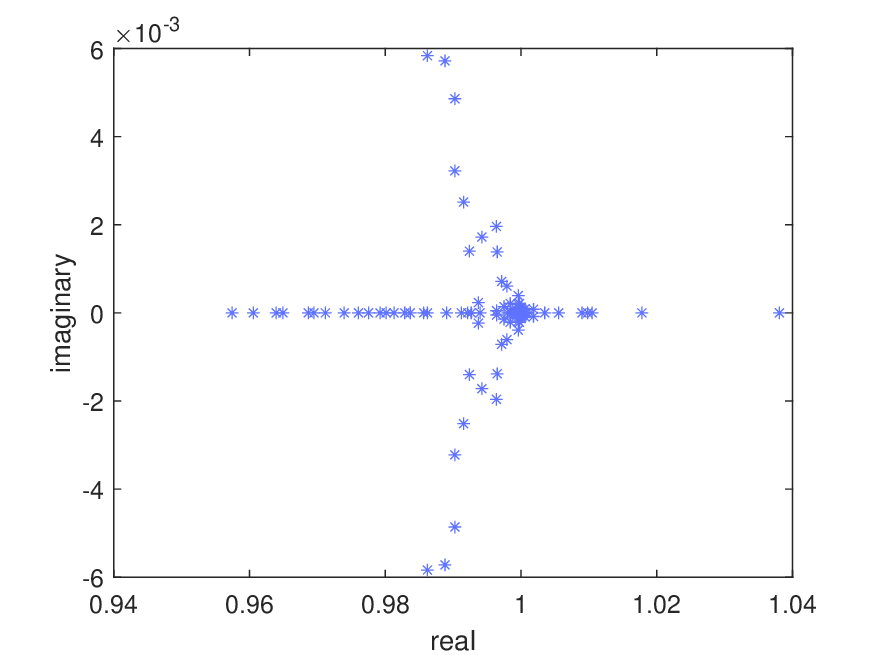}
			}
            \subfigure[$\calP^{-1}\calA$]{
				\includegraphics[width=0.31 \linewidth]{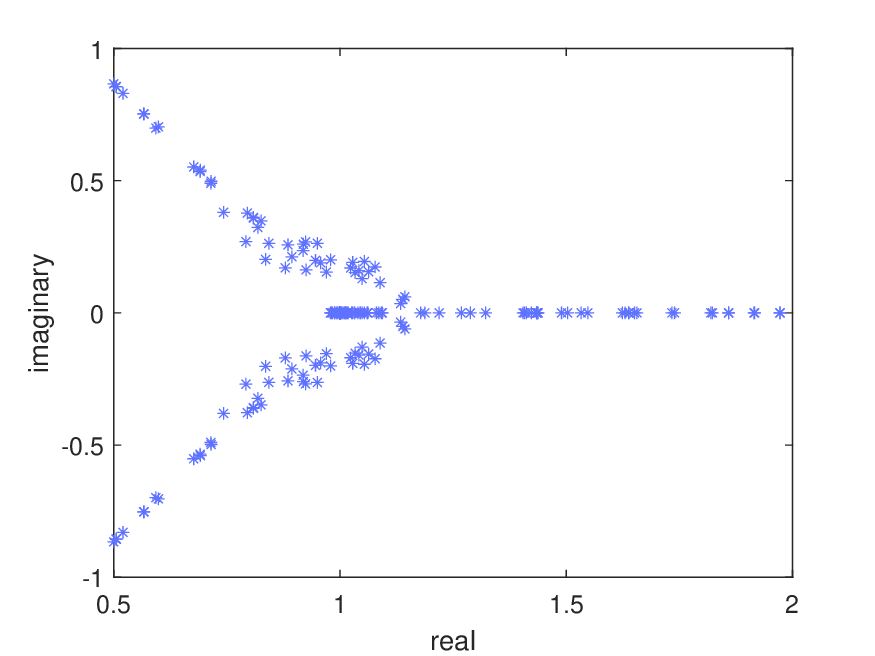}
			}
			\caption{Eigenvalue distributions of original coefficient matrix and the preconditioned matrices for \Cref{ex2} with $\gamma=10^{-4}$ and $h = 2^{-3}$.}
			\label{fig:eigenvalue_ex2}
		\end{figure}

        \Cref{table:ex2_result,table:ex2_result2,table:ex2_result3} illustrate that IMD and RIMD outperforms the other tested methods in terms of CPU time. Similar to the results in \Cref{ex1}, the randomized low-rank approximation after diagonal approximation performs well, demonstrating the feasibility and effectiveness of the randomized approach during preconditioner construction. Although DIAG and TBD exhibit good robustness with respect to the parameters, the efficiency of HSS, DIAG, TBD, and TPSS is far inferior to that of BPP, IMD, and RIMD, and this difference gradually becomes more evident as the scale of the problem increases. DS and RDF fail to solve the double saddle-point problem from \Cref{ex2} when $h \le 2^{-6}$, and TPSS fails when $h \le 2^{-7}$. While GSS performs favorably in terms of the number of iterations, it suffers from a clear disadvantage in CPU time relative to other methods. IMD, RIMD, and BPP enjoy substantial advantages in CPU time over other tested methods. This demonstrates that the inexact preconditioning technique, which avoids the explicit computation and storage of the Schur complement, is highly effective for solving double saddle-point system. Furthermore, it can be found that RIMD exhibits $h$-robustness. From \Cref{fig:res_ex2}, it is clear that the residual norms of all tested methods exhibit an overall decreasing trend with iterations and satisfy the stopping criterion in finite steps. We can see from \Cref{fig:eigenvalue_ex2} that the eigenvalues of all preconditioned matrices are more clustered than those of the original coefficient matrix.

        \begin{table}[htb]\small
			\caption{Numerical results for saddle-point system from \Cref{ex2} with $\gamma=10^{-4}$.}
			\centering
			\label{table:ex2_result2}
			\setlength{\tabcolsep}{2.5mm}{
				\begin{tabular}{@{}ccccccccccccccc@{}}
					\toprule
					\multirow{2}{*}{Methods} & \multirow{2}{*}{} & \multicolumn{5}{c}{$h$}                        \\ \cmidrule(l){3-7}
					&    & $2^{-3}$  & $2^{-4}$  & $2^{-5}$  & $2^{-6}$  & $2^{-7}$   \\ \midrule
                    HSS 
                    & IT & 81 & 104 & 323 & - & - \\ 
                    & CPU & 0.05 & 0.16 & 9.06 & - & - \\
                    & RES & 5.40e-09 & 5.12e-09 & 9.19e-09 & - & - \\
					DIAG
                    & IT & 2 & 2 & 2 & 2 & 2\\
                    & CPU & 0.09 & 0.17 & 0.23 & 3.73 & 147.17\\
                    & RES & 7.55e-15 & 1.57e-13 & 4.08e-12 & 8.36e-11 & 1.45e-09\\
                    TBD
                    & IT & 1 & 1 & 1 & 1 & 1\\
                    & CPU & 0.02 & 0.06 & 0.16 & 3.40 & 133.46\\
                    & RES & 9.01e-15 & 1.51e-13 & 4.14e-12 & 8.39e-11 & 1.45e-09\\
                    TPSS
                    & IT & 32 & 29 & 30 & 30 & -\\
                    & CPU & 0.03 & 0.12 & 1.66 & 50.99 & -\\
                    & RES & 1.98e-09 & 9.87e-09 & 6.47e-09 & 7.47e-09 & -\\
                    DS
                    & IT & 26 & 32 & 35 & - & -\\
                    & CPU & 0.02 & 0.08 & 0.52 & - & -\\
                    & RES & 7.61e-09 & 6.76e-09 & 4.78e-09 & - & -\\
                    RDF
                    & IT & 23 & 26 & 27 & - & -\\
                    & CPU & 0.02 & 0.06 & 0.31 & - & -\\
                    & RES & 1.00e-08 & 6.00e-09 & 5.70e-09 & - & -\\
                    GSS
                    & IT & 6 & 5 & 5 & 5 & 6\\
                    & CPU & 0.02 & 0.10 & 0.37 & 13.23 & 357.78\\
                    & RES & 2.44e-10 & 1.54e-09 & 1.91e-09 & 8.40e-09 & 3.18e-10\\
                    BD
                    & IT & 35 & 37 & 37 & 37 & 37\\
                    & CPU & 0.03 & 0.10 & 0.45 & 12.87 & 293.99\\
                    & RES & 1.16e-09 & 3.15e-09 & 1.75e-09 & 1.66e-09 & 2.92e-09\\
                    BPP
                    & IT & 14 & 16 & 16 & 16 & 20\\
                    & CPU & 0.01 & 0.03 & 0.13 & 0.63 & 3.17\\
                    & RES & 9.11e-09 & 2.43e-09 & 1.64e-09 & 1.66e-09 & 9.56e-09\\
                    IMD
                    & IT & 32 & 32 & 34 & 38 & 41\\
                    & CPU & 0.02 & 0.03 & 0.08 & 0.52 & 2.34\\
                    & RES & 2.89e-09 & 8.10e-09 & 1.78e-09 & 4.99e-09 & 8.58e-09\\
                    RIMD
                    & IT & 32 & 34 & 34 & 36 & 41\\
                    & CPU & 0.02 & 0.02 & 0.08 & 0.48 & 2.26\\
                    & RES & 7.64e-09 & 2.14e-09 & 4.25e-09 & 9.64e-09 & 8.58e-09\\
					\hline
			\end{tabular}}
		\end{table}

        \begin{table}[htb]\small
			\caption{Numerical results for saddle-point system from \Cref{ex2} with $\gamma=10^{-5}$.}
			\centering
			\label{table:ex2_result3}
			\setlength{\tabcolsep}{2.5mm}{
				\begin{tabular}{@{}ccccccccccccccc@{}}
					\toprule
					\multirow{2}{*}{Methods} & \multirow{2}{*}{} & \multicolumn{5}{c}{$h$}                        \\ \cmidrule(l){3-7}
					&    & $2^{-3}$  & $2^{-4}$  & $2^{-5}$  & $2^{-6}$  & $2^{-7}$   \\ \midrule
                    HSS 
                    & IT & 108 & 272 & 576 & - & - \\ 
                    & CPU & 0.07 & 0.71 & 18.60 & - & - \\
                    & RES & 8.06e-09 & 9.76e-09 & 9.15e-09 & - & - \\
					DIAG
                    & IT & 2 & 2 & 2 & 2 & 2\\
                    & CPU & 0.03 & 0.11 & 0.16 & 5.87 & 156.04\\
                    & RES & 1.29e-15 & 2.15e-14 & 4.83e-13 & 1.06e-11 & 2.01e-10\\
                    TBD
                    & IT & 1 & 1 & 1 & 1 & 1\\
                    & CPU & 0.02 & 0.04 & 0.15 & 5.06 & 148.13\\
                    & RES & 1.42e-15 & 2.19e-14 & 4.85e-13 & 1.06e-11 & 2.02e-10\\
                    TPSS
                    & IT & 32 & 30 & 25 & 25 & -\\
                    & CPU & 0.05 & 0.10 & 1.76 & 57.30 & -\\
                    & RES & 3.78e-09 & 2.94e-09 & 7.18e-09 & 4.19e-09 & -\\
                    DS
                    & IT & 28 & 40 & 50 & - & -\\
                    & CPU & 0.12 & 0.22 & 0.70 & - & -\\
                    & RES & 6.39e-09 & 9.09e-09 & 8.01e-09 & - & -\\
                    RDF
                    & IT & 24 & 29 & 31 & - & -\\
                    & CPU & 0.03 & 0.10 & 0.34 & - & -\\
                    & RES & 8.07e-09 & 8.07e-09 & 8.59e-09 & - & -\\
                    GSS
                    & IT & 7 & 6 & 5 & - & -\\
                    & CPU & 0.03 & 0.10 & 0.36 & - & -\\
                    & RES & 2.43e-09 & 1.70e-09 & 1.64e-09 & - & -\\
                    BD
                    & IT & 48 & 49 & 48 & 48 & 48\\
                    & CPU & 0.03 & 0.06 & 0.53 & 15.73 & 343.51\\
                    & RES & 2.25e-09 & 9.72e-09 & 4.39e-09 & 3.67e-09 & 3.63e-09\\
                    BPP
                    & IT & 16 & 16 & 16 & 16 & 16\\
                    & CPU & 0.02 & 0.03 & 0.19 & 0.67 & 3.46\\
                    & RES & 2.55e-09 & 9.59e-09 & 1.01e-09 & 2.23e-09 & 8.91e-09\\
                    IMD
                    & IT & 29 & 33 & 35 & 39 & 43\\
                    & CPU & 0.02 & 0.06 & 0.13 & 0.58 & 2.64\\
                    & RES & 8.09e-09 & 7.01e-09 & 4.58e-09 & 3.16e-09 & 7.42e-09\\
                    RIMD
                    & IT & 29 & 32 & 35 & 39 & 43\\
                    & CPU & 0.03 & 0.03 & 0.09 & 0.53 & 2.57\\
                    & RES & 8.13e-09 & 9.41e-09 & 4.55e-09 & 2.69e-09 & 7.42e-09\\
					\hline
			\end{tabular}}
		\end{table}

\section{Conclusions}\label{sec:conclusion}
	
This work developed efficient preconditioning strategies for double saddle-point systems by exploiting their hierarchical block structure. The proposed framework enables the construction of effective inexact block triangular preconditioners while reducing computational and storage costs. Theoretical analysis and numerical results demonstrate the robustness of the proposed approaches and highlight the potential of randomized approximation techniques for large-scale saddle-point problems.

Future work will focus on adaptive approximation strategies to further balance accuracy and efficiency. It is also of interest to extend the proposed framework to broader classes of saddle-point systems and explore more advanced randomized techniques for preconditioner design.

	\vspace{1cm}
	
	\noindent\textbf{Acknowledgements} We are grateful to Professor Andreas Potschka for providing the codes for generating the test matrices of Example 2 in our numerical experiments.

	\vspace{1cm}
		
	\noindent\textbf{Data Availability} No new experimental data were generated in this work. All data used are available from the corresponding author upon reasonable request.
	
	\section*{Declarations}

	\noindent\textbf{Conflict of interest} The authors have not disclosed any competing interests.

	\bibliographystyle{abbrvnat}
	\bibliography{references}

@string{bit = {BIT Numer. Math.}}

@string{siam = {SIAM J. Appl. Math.}}

@article{rees2010optimal,
	title={Optimal solvers for {PDE}-constrained optimization},
	author={Rees, Tyrone and Dollar, H Sue and Wathen, Andrew J},
	journal={SIAM J. Sci. Comput.},
	volume={32},
	number={1},
	pages={271--298},
	year={2010},
	publisher={SIAM}
}

@article{ahmad2026class,
	title={A class of generalized shift-splitting preconditioners for double saddle point problems},
	author={Ahmad, Sk Safique and Khatun, Pinki},
	journal={Appl. Math. Comput.},
	volume={509},
	pages={129658},
	year={2026},
	publisher={Elsevier}
}

@article{bradley2023eigenvalue,
	title={Eigenvalue bounds for double saddle-point systems},
	author={Bradley, Susanne and Greif, Chen},
	journal={IMA J. Numer. Anal.},
	volume={43},
	number={6},
	pages={3564--3592},
	year={2023},
	publisher={Oxford University Press}
}

@misc{tyronerees,
	author={Tyrone Rees},
	title={Github-tyronerees/poisson-control},
	year={2010},
	howpublished={\url{https://github.com/tyronerees/poisson-control}},
	note={Retrieved: 2026/1/29},
}

@article{bergamaschi2025spectral,
	title={Spectral analysis of block preconditioners for double saddle-point linear systems with application to {PDE}-constrained optimization},
	author={Bergamaschi, Luca and Mart{\'\i}nez, {\'A}ngeles and Pearson, John W and Potschka, Andreas},
	journal={Comput. Optim. Appl.},
	volume={91},
	number={2},
	pages={423--455},
	year={2025},
	publisher={Springer}
}

@article{benzi2011dimensional,
	title={A dimensional split preconditioner for {Stokes} and linearized {Navier--Stokes} equations},
	author={Benzi, Michele and Guo, Xueping},
	journal={Appl. Numer. Math.},
	volume={61},
	number={1},
	pages={66--76},
	year={2011},
	publisher={Elsevier}
}

@article{benzi2011relaxed,
	title={A relaxed dimensional factorization preconditioner for the incompressible {Navier--Stokes} equations},
	author={Benzi, Michele and Ng, Michael and Niu, Qiang and Wang, Zhen},
	journal={J. Comput. Phys.},
	volume={230},
	number={16},
	pages={6185--6202},
	year={2011},
	publisher={Elsevier}
}

@article{ai2024multi,
	title={Multi-parameter dimensional split preconditioner for three-by-three block system of linear equations},
	author={Yang, Aili and Zhu, Junli and Wu, Yujiang},
	journal={Numer. Algor.},
	volume={95},
	number={2},
	pages={721--745},
	year={2024},
	publisher={Springer}
}

@article{pearson2024symmetric,
	title={On symmetric positive definite preconditioners for multiple saddle-point systems},
	author={Pearson, John W and Potschka, Andreas},
	journal={IMA J. Numer. Anal.},
	volume={44},
	number={3},
	pages={1731--1750},
	year={2024},
	publisher={Oxford University Press}
}

@article{rhebergen2015three,
	title={Three-field block preconditioners for models of coupled magma/mantle dynamics},
	author={Rhebergen, Sander and Wells, Garth N and Wathen, Andrew J and Katz, Richard F},
	journal={SIAM J. Sci. Comput.},
	volume={37},
	number={5},
	pages={A2270--A2294},
	year={2015},
	publisher={SIAM}
}

@book{elman2014finite,
	title={Finite elements and fast iterative solvers: with applications in incompressible fluid dynamics},
	author={Elman, Howard C and Silvester, David J and Wathen, Andrew J},
	year={2014},
	publisher={Oxford university press}
}

@article{han2013local,
	title={Local linear convergence of the alternating direction method of multipliers for quadratic programs},
	author={Han, Deren and Yuan, Xiaoming},
	journal={SIAM J. Numer. Anal.},
	volume={51},
	number={6},
	pages={3446--3457},
	year={2013},
	publisher={SIAM}
}

@article{huang2016consensus,
	title={{Consensus-ADMM} for general quadratically constrained quadratic programming},
	author={Huang, Kejun and Sidiropoulos, Nicholas D},
	journal={IEEE T. Signal Proces.},
	volume={64},
	number={20},
	pages={5297--5310},
	year={2016},
	publisher={IEEE}
}

@article{yuan1996numerical,
	title={Numerical methods for generalized least squares problems},
	author={Yuan, Jin Yun},
	journal={J. Comput. Appl. Math.},
	volume={66},
	pages={571--584},
	year={1996},
}

@article{bojanczyk2003equality,
	title={The equality constrained indefinite least squares problem: theory and algorithms},
	author={Bojanczyk, Adam and Higham, Nicholas J and Patel, Harikrishna},
	journal={BIT Numer. Math.},
	volume={43},
	number={3},
	pages={505--517},
	year={2003},
	publisher={Springer}
}

@book{bjorck2024numerical,
	title={Numerical methods for least squares problems},
	author={Bj{\"o}rck, {\AA}ke},
	year={2024},
	publisher={SIAM}
}

@article{ramage2013preconditioned,
	title={A preconditioned nullspace method for liquid crystal director modeling},
	author={Ramage, Alison and Gartland Jr, Eugene C},
	journal={SIAM J. Sci. Comput.},
	volume={35},
	number={1},
	pages={B226--B247},
	year={2013},
	publisher={SIAM}
}

@article{cai2009preconditioning,
	title={Preconditioning techniques for a mixed {Stokes/Darcy} model in porous media applications},
	author={Cai, Mingchao and Mu, Mo and Xu, Jinchao},
	journal={J. Comput. Appl. Math.},
	volume={233},
	number={2},
	pages={346--355},
	year={2009},
	publisher={Elsevier}
}

@article{ali2018iterative,
	title={Iterative methods for double saddle point systems},
	author={Beik, Fatemeh Panjeh Ali and Benzi, Michele},
	journal={SIAM J. Matrix Anal. Appl.},
	volume={39},
	number={2},
	pages={902--921},
	year={2018},
	publisher={SIAM}
}

@article{huang2023gsor,
	title={On {GSOR}, the generalized successive overrelaxation method for double saddle-point problems},
	author={Huang, Na and Dai, Yuhong and Orban, Dominique and Saunders, Michael A},
	journal={SIAM J. Sci. Comput.},
	volume={45},
	number={5},
	pages={A2185--A2206},
	year={2023},
	publisher={SIAM}
}

@article{dou2023class,
	title={A class of block alternating splitting implicit iteration methods for double saddle point linear systems},
	author={Dou, Yan and Liang, Zhaozheng},
	journal={Numer. Linear Algebra Appl.},
	volume={30},
	number={1},
	pages={e2455},
	year={2023},
	publisher={Wiley Online Library}
}

@article{beik2022preconditioning,
	title={Preconditioning techniques for the coupled {Stokes--Darcy} problem: spectral and field-of-values analysis},
	author={Beik, Fatemeh Panjeh Ali and Benzi, Michele},
	journal={Numer. Math.},
	volume={150},
	number={2},
	pages={257--298},
	year={2022},
	publisher={Springer}
}

@article{cao2019shift,
	title={Shift-splitting preconditioners for a class of block three-by-three saddle point problems},
	author={Cao, Yang},
	journal={Appl. Math. Lett.},
	volume={96},
	pages={40--46},
	year={2019},
	publisher={Elsevier}
}

@article{zhang2022lopsided,
	title={Lopsided shift-splitting preconditioner for saddle point problems with three-by-three structure.},
	author={Zhang, Na and Li, Ruixia and Li, Jian},
	journal={Comput. Appl. Math.},
	volume={41},
	number={6},
    pages={261},
	year={2022},
    publisher={Springer}
}

@article{ahmad2025robust,
	title={A robust parameterized enhanced shift-splitting preconditioner for three-by-three block saddle point problems},
	author={Ahmad, Sk Safique and Khatun, Pinki},
	journal={J. Comput. Appl. Math.},
	volume={459},
	pages={116358},
	year={2025},
	publisher={Elsevier}
}

@article{huang2019uzawa,
	title={Uzawa methods for a class of block three-by-three saddle-point problems},
	author={Huang, Na and Dai, Yuhong and Hu, QiYa},
	journal={Numer. Linear Algebra Appl.},
	volume={26},
	number={6},
	pages={e2265},
	year={2019},
	publisher={Wiley Online Library}
}

@article{huang2020variable,
	title={Variable parameter {Uzawa} method for solving a class of block three-by-three saddle point problems},
	author={Huang, Na},
	journal={Numer. Algor.},
	volume={85},
	number={4},
	pages={1233--1254},
	year={2020},
	publisher={Springer}
}

@book{saad2003iterative,
	title={Iterative methods for sparse linear systems},
	author={Saad, Yousef},
	year={2003},
	publisher={SIAM}
}

@article{huang2019spectral,
	title={Spectral analysis of the preconditioned system for the 3 $\times$ 3 block saddle point problem},
	author={Huang, Na and Ma, Changfeng},
	journal={Numer. Algor.},
	volume={81},
	number={2},
	pages={421--444},
	year={2019},
	publisher={Springer}
}

@article{abdolmaleki2022new,
	title={A new block-diagonal preconditioner for a class of 3 $\times$ 3 block saddle point problems},
	author={Abdolmaleki, Maryam and Karimi, Saeed and Salkuyeh, Davod Khojasteh},
	journal={Mediterr. J. Math.},
	volume={19},
	number={1},
	pages={43},
	year={2022},
	publisher={Springer}
}

@article{aslani2023block,
	title={A block triangular preconditioner for a class of three-by-three block saddle point problems},
	author={Aslani, Hamed and Salkuyeh, Davod Khojasteh},
	journal={Jpn. J. Ind. Appl. Math.},
	volume={40},
	number={2},
	pages={1015--1030},
	year={2023},
	publisher={Springer}
}

@article{balani2024some,
	title={Some preconditioning techniques for a class of double saddle point problems},
	author={Balani Bakrani, Fariba and Bergamaschi, Luca and Mart{\'\i}nez, {\'A}ngeles and Hajarian, Masoud},
	journal={Numer. Linear Algebra Appl.},
	volume={31},
	number={4},
	pages={e2551},
	year={2024},
	publisher={Wiley Online Library}
}

@article{liang2024improvement,
	title={On the improvement of shift-splitting preconditioners for double saddle point problems},
	author={Liang, Zhaozheng and Zhu, Muzheng},
	journal={J. Appl. Math. Comput.},
	volume={70},
	number={2},
	pages={1339--1363},
	year={2024},
	publisher={Springer}
}

@article{salkuyeh2021alternating,
	title={An alternating positive semidefinite splitting preconditioner for the three-by-three block saddle point problems},
	author={Salkuyeh, Davod Khojasteh and Aslani, Hamed and Liang, Zhaozheng},
	journal={Math. Commun.},
	volume={26},
	number={2},
	pages={177--195},
	year={2021},
	publisher={Sveu{\v{c}}ili{\v{s}}te Josipa Jurja Strossmayera u Osijeku, Odjel za matematiku}
}

@article{li2025uzawa,
	title={The {Uzawa-type} shift-splitting preconditioners for double saddle point problems},
	author={Li, Chengliang and Xu, Ying and Ma, Changfeng},
	journal={J. Appl. Math. Comput.},
	volume={71},
	number={5},
	pages={7837--7861},
	year={2025},
	publisher={Springer}
}

@article{pearson2014preconditioners,
	title={Preconditioners for state-constrained optimal control problems with {Moreau--Yosida} penalty function},
	author={Pearson, John W and Stoll, Martin and Wathen, Andrew J},
	journal={Numer. Linear Algebra Appl.},
	volume={21},
	number={1},
	pages={81--97},
	year={2014},
	publisher={Wiley Online Library}
}

@article{fan2024preconditioners,
	title={Preconditioners based on matrix splitting for the structured systems from elliptic {PDE-constrained} optimization problems},
	author={Fan, Hongtao and Li, Yajing and Zhang, Hongbing and Zhu, Xinyun},
	journal={Appl. Math. Comput.},
	volume={463},
	pages={128341},
	year={2024},
	publisher={Elsevier}
}

@article{ke2018some,
	title={Some preconditioners for elliptic {PDE-constrained} optimization problems},
	author={Ke, Yifen and Ma, Changfeng},
	journal={Comput. Math. Appl.},
	volume={75},
	number={8},
	pages={2795--2813},
	year={2018},
	publisher={Elsevier}
}

@article{rees2010block,
	title={Block-triangular preconditioners for {PDE-constrained} optimization},
	author={Rees, Tyrone and Stoll, Martin},
	journal={Numer. Linear Algebra Appl.},
	volume={17},
	number={6},
	pages={977--996},
	year={2010},
	publisher={Wiley Online Library}
}

@article{zhang2014block,
	title={On block preconditioners for {PDE-constrained} optimization problems},
	author={Zhang, Xiaoying and Huang, Yumei},
	journal={J. Comput. Math.},
    volume={32},
	pages={272--283},
	year={2014},
	publisher={JSTOR}
}

@article{pearson2012new,
	title={A new approximation of the {Schur} complement in preconditioners for {PDE-constrained} optimization},
	author={Pearson, John W and Wathen, Andrew J},
	journal={Numer. Linear Algebra Appl.},
	volume={19},
	number={5},
	pages={816--829},
	year={2012},
	publisher={Wiley Online Library}
}

@article{liang2025inexact,
	title={Inexact block triangular preconditioners for double saddle-point systems arising from coupled {Stokes--Darcy} model},
	author={Liang, Siqi and Huang, Na},
	journal={J. Comput. Appl. Math.},
	volume={476},
	pages={117079},
	year={2026},
	publisher={Elsevier}
}

@article{halko2011finding,
	title={Finding structure with randomness: {Probabilistic} algorithms for constructing approximate matrix decompositions},
	author={Halko, Nathan and Martinsson, Per-Gunnar and Tropp, Joel A},
	journal={SIAM Rev.},
	volume={53},
	number={2},
	pages={217--288},
	year={2011},
	publisher={SIAM}
}

@article{martinsson2020randomized,
	title={Randomized numerical linear algebra: {Foundations} and algorithms},
	author={Martinsson, Per-Gunnar and Tropp, Joel A},
	journal={Acta Numer.},
	volume={29},
	pages={403--572},
	year={2020},
	publisher={Cambridge University Press}
}

@article{murray2023randomized,
	title={Randomized numerical linear algebra: {A} perspective on the field with an eye to software},
	author={Murray, Riley and Demmel, James and Mahoney, Michael W and Erichson, N Benjamin and Melnichenko, Maksim and Malik, Osman Asif and Grigori, Laura and Luszczek, Piotr and Derezi{\'n}ski, Micha{\l} and Lopes, Miles E and others},
	journal={arXiv preprint arXiv:2302.11474},
	year={2023}
}

@article{frangella2023randomized,
	title={Randomized {Nystr{\"o}m} preconditioning},
	author={Frangella, Zachary and Tropp, Joel A and Udell, Madeleine},
	journal={SIAM J. Matrix Anal. Appl.},
	volume={44},
	number={2},
	pages={718--752},
	year={2023},
	publisher={SIAM}
}

@article{diaz2023robust,
	title={Robust, randomized preconditioning for kernel ridge regression},
	author={D{\'\i}az, Mateo and Epperly, Ethan N and Frangella, Zachary and Tropp, Joel A and Webber, Robert J},
	journal={arXiv preprint arXiv:2304.12465},
	year={2023}
}

@article{balabanov2025preconditioning,
	title={Preconditioning via randomized range deflation {(RandRAND)}},
	author={Balabanov, Oleg and Ju, Caleb and He, Kaiwen and Jeendgar, Aryaman and Mahoney, Michael W},
	journal={arXiv preprint arXiv:2509.19747},
	year={2025}
}

@article{woodruff2014sketching,
	title={Sketching as a tool for numerical linear algebra},
	author={Woodruff, David P},
	journal={Foundations Trends Theor. Comput. Sci.},
	volume={10},
	number={1},
	pages={1--157},
	year={2014},
    publisher={Now Publishers Boston-Delft}
}

@article{nakatsukasa2023randomized,
	title={Randomized low-rank approximation for symmetric indefinite matrices},
	author={Nakatsukasa, Yuji and Park, Taejun},
	journal={SIAM J. Matrix Anal. Appl.},
	volume={44},
	number={3},
	pages={1370--1392},
	year={2023},
	publisher={SIAM}
}

@article{benzi2005numerical,
  title={Numerical solution of saddle point problems},
  author={Benzi, Michele and Golub, Gene H and Liesen, J{\"o}rg},
  journal={Acta Numer.},
  volume={14},
  pages={1--137},
  year={2005},
  publisher={Cambridge University Press}
}

@incollection{benzi2008some,
  title={Some preconditioning techniques for saddle point problems},
  author={Benzi, Michele and Wathen, Andrew J},
  booktitle={Model order reduction: theory, research aspects and applications},
  pages={195--211},
  year={2008},
  publisher={Springer}
}

@article{song2022two,
  title={A two-parameter shift-splitting preconditioner for saddle point problems},
  author={Song, Shengzhong and Huang, Zhengda},
  journal={Comput. Math. Appl.},
  volume={124},
  pages={7--20},
  year={2022},
  publisher={Elsevier}
}

@article{liang2026partial,
  title={Partial shift-splitting preconditioners for double saddle-point systems},
  author={Liang, Siqi and Huang, Na},
  journal={Comput. Math. Appl.},
  volume={214},
  pages={230--247},
  year={2026},
  publisher={Elsevier}
}

@article{cai2022fast,
  title={Fast deterministic approximation of symmetric indefinite kernel matrices with high dimensional datasets},
  author={Cai, Difeng and Nagy, James and Xi, Yuanzhe},
  journal={SIAM J. Matrix Anal. Appl.},
  volume={43},
  number={2},
  pages={1003--1028},
  year={2022},
  publisher={SIAM}
}

@article{ray2022sublinear,
  title={Sublinear time approximation of text similarity matrices},
  author={Ray, Archan and Monath, Nicholas and McCallum, Andrew and Musco, Cameron},
  journal={Proceedings of the AAAI Conference on Artificial Intelligence},
  volume={36},
  number={7},
  pages={8072--8080},
  year={2022}
}

@article{nakatsukasa2020fast,
  title={Fast and stable randomized low-rank matrix approximation},
  author={Nakatsukasa, Yuji},
  journal={arXiv preprint arXiv:2009.11392},
  year={2020}
}

@book{horn2012matrix,
  title={Matrix analysis},
  author={Horn, Roger A and Johnson, Charles R},
  year={2012},
  publisher={Cambridge university press}
}

@book{tao2023topics,
  title={Topics in random matrix theory},
  author={Tao, Terence},
  volume={132},
  year={2023},
  publisher={American Mathematical Society}
}

@book{vershynin2018high,
  title={High-dimensional probability: {An} introduction with applications in data science},
  author={Vershynin, Roman},
  volume={47},
  year={2018},
  publisher={Cambridge university press}
}

@book{golub2013matrix,
  title={Matrix computations},
  author={Golub, Gene H and Van Loan, Charles F},
  year={2013},
  publisher={JHU press}
}

@article{rudelson2009smallest,
  title={Smallest singular value of a random rectangular matrix},
  author={Rudelson, Mark and Vershynin, Roman},
  journal={Commun. Pur. Appl. Math.},
  volume={62},
  number={12},
  pages={1707--1739},
  year={2009},
  publisher={Wiley Online Library}
}

@article{cao2014shift,
  title={Shift-splitting preconditioners for saddle point problems},
  author={Cao, Yang and Du, Jun and Niu, Qiang},
  journal={J. Comput. Appl. Math.},
  volume={272},
  pages={239--250},
  year={2014},
  publisher={Elsevier}
}

@article{elman1994inexact,
  title={Inexact and preconditioned {Uzawa} algorithms for saddle point problems},
  author={Elman, Howard C and Golub, Gene H},
  journal={SIAM J. Numer. Anal.},
  volume={31},
  number={6},
  pages={1645--1661},
  year={1994},
  publisher={SIAM}
}

@article{bai2008parameterized,
  title={On parameterized inexact {Uzawa} methods for generalized saddle point problems},
  author={Bai, Zhongzhi and Wang, Zengqi},
  journal={Linear Algebra Appl.},
  volume={428},
  number={11-12},
  pages={2900--2932},
  year={2008},
  publisher={Elsevier}
}

@article{benzi2004preconditioner,
  title={A preconditioner for generalized saddle point problems},
  author={Benzi, Michele and Golub, Gene H},
  journal={SIAM J. Matrix Anal. Appl.},
  volume={26},
  number={1},
  pages={20--41},
  year={2004},
  publisher={SIAM}
}

@article{simoncini2004block,
  title={Block triangular preconditioners for symmetric saddle-point problems},
  author={Simoncini, Valeria},
  journal={Appl. Numer. Math.},
  volume={49},
  number={1},
  pages={63--80},
  year={2004},
  publisher={Elsevier}
}

@article{bramble1997analysis,
  title={Analysis of the inexact {Uzawa} algorithm for saddle point problems},
  author={Bramble, James H and Pasciak, Joseph E and Vassilev, Apostol T},
  journal={SIAM J. Numer. Anal.},
  volume={34},
  number={3},
  pages={1072--1092},
  year={1997},
  publisher={SIAM}
}

@article{cao2017preconditioned,
  title={On preconditioned generalized shift-splitting iteration methods for saddle point problems},
  author={Cao, Yang and Miao, Shuxin and Ren, Zhiru},
  journal={Comput. Math. Appl.},
  volume={74},
  number={4},
  pages={859--872},
  year={2017},
  publisher={Elsevier}
}

@article{cao2016simplified,
  title={A simplified {HSS} preconditioner for generalized saddle point problems},
  author={Cao, Yang and Ren, Zhiru and Shi, Quan},
  journal={BIT Numer. Math.},
  volume={56},
  number={2},
  pages={423--439},
  year={2016},
  publisher={Springer}
}
	
\end{document}